\documentclass[12pt]{amsart}

\usepackage[left=2cm, right=2cm, top=3cm, bottom=3cm]{geometry}
\usepackage{hyperref}
\usepackage{amsmath, amsthm, amsfonts, amssymb}
\usepackage{mathtools}
\usepackage{mathrsfs}
\usepackage{upgreek}

\hypersetup{
	colorlinks=true,
	pdftitle={Quadratic Gauss sums over Zn/cZn: explicit formulas, a duality theorem, and applications to Weil representations and cubic hypersurfaces over Fp},
	pdfauthor={Xiao-Jie Zhu},
	pdfsubject={Mathematics},
	pdfkeywords={Gauss sum, duality theorem, Weil representation, quadratic hypersurface, Markoff equation, Rosenberger variation, finite quadratic module, Hecke Gauss sum}
}

\numberwithin{equation}{section}

\makeatletter
\renewcommand\l@subsection{\@tocline{2}{0pt}{2pc}{5pc}{}}
\makeatother

\allowdisplaybreaks
\makeatletter
\@namedef{subjclassname@2020}{\textup{2020} Mathematics Subject Classification}
\makeatother

\theoremstyle{plain}
\newtheorem{thm}{Theorem}[section]
\newtheorem*{thm*}{Theorem}
\newtheorem{coro}[thm]{Corollary}
\newtheorem{conj}[thm]{Conjecture}
\newtheorem{prop}[thm]{Proposition}

\newtheorem{lemm}[thm]{Lemma}

\theoremstyle{definition}
\newtheorem{deff}[thm]{Definition}
\newtheorem{examp}[thm]{Example}

\theoremstyle{remark}
\newtheorem{rema}[thm]{Remark}

\newcommand\twoscript[2]{\substack{{#1} \\ {#2}}}

\newcommand\rmi{\mathrm{i}}
\newcommand\rme{\mathrm{e}}
\newcommand\uhp{\mathfrak{H}}
\newcommand\tp{\mathrm{T}}
\newcommand\legendre[2]{\genfrac{(}{)}{}{}{#1}{#2}}
\newcommand\tbtmat[4]{\left(\begin{smallmatrix}{#1} & {#2} \\ {#3} & {#4}\end{smallmatrix}\right)}
\newcommand\tbtMat[4]{\left(\begin{matrix}{#1} & {#2} \\ {#3} & {#4}\end{matrix}\right)}
\newcommand*\abs[1]{\lvert#1\rvert}
\newcommand\etp[1]{\mathfrak{e}\left(#1\right)}
\newcommand\vol[1]{\mathop{\mathrm{vol}}\left(#1\right)}

\newcommand\sgn[1]{\mathop{\mathrm{sgn}}\left(#1\right)}
\newcommand\tr{\mathop{\mathrm{tr}}}
\newcommand\ord{\mathop{\mathrm{ord}}}
\newcommand\rank{\mathop{\mathrm{rank}}}

\newcommand\Hom{\mathop{\mathrm{Hom}}}

\newcommand\numZ{\mathbb{Z}}
\newcommand\numQ{\mathbb{Q}}
\newcommand\numR{\mathbb{R}}
\newcommand\numC{\mathbb{C}}

\newcommand\numgeq[2]{\mathbb{#1}_{\geq #2}}

\newcommand\slZ{\mathrm{SL}_2(\mathbb{Z})}

\begin{document}

\title[Quadratic Gauss sums over $\numZ^n/c\numZ^n$]{Quadratic Gauss sums over $\numZ^n/c\numZ^n$: explicit formulas, a duality theorem, and applications to Weil representations and cubic hypersurfaces over $\mathbb{F}_p$}

\author{Xiao-Jie Zhu}
\address{School of Mathematical Sciences,
Key Laboratory of MEA(Ministry of Education) \& Shanghai Key Laboratory of PMMP,
East China Normal University}
\email{zhuxiaojiemath@outlook.com}



\begin{abstract}
We provide explicit formulas for quadratic Gauss sums over $\numZ^n/c\numZ^n$, which generalize some of the existing formulas, e.g., Skoruppa and Zagier's (for $n=2$), and Iwaniec and Kowalski's (for arbitrary $n$). We then give four main applications. As the first application, we prove a duality theorem, which relates a sum over a subgroup of $\numZ^n/c\numZ^n$ to another sum over the orthogonal complement. This allows us to give explicit formulas for quadratic Gauss partial sums over certain subgroups. As the second application, we give an explicit formula for the coefficients of Weil representations of $\mathrm{Mp}_2(\numZ)$, which has the advantage, compared to Scheithauer's, Str\"omberg's, and Boylan and Skoruppa's formulas, that it involves neither local data nor limits of theta series. As the third application, we provide an explicit formula for the number of solutions of an arbitrary quadratic congruence in $n$ variables modulo a prime, as well as an efficient formula modulo a general integer. As the final application, we give a uniform formula for the number of solutions of the Diophantine equation $dxyz=Q(x,y,z)$ over $\mathbb{F}_p$, where $Q$ is an arbitrary integral quadratic form with some weak restrictions. This generalizes a result of Baragar regarding the Markoff equation. There is also a formula relating Hecke Gauss sums to sums over $\numZ^n/c\numZ^n$, an explicit formula for Hecke Gauss sums over quadratic fields, and an example of computing a Hecke Gauss sum over a cyclotomic field.
\end{abstract}

\maketitle

\section{Introduction}
Let $G$ be a positive definite even integral\footnote{A symmetric matrix $G$ is said to be even integral if all entries of $G$ are in $\numZ$ and the entries at the diagonal are in $2\numZ$.} symmetric matrix of size $n\in\numgeq{Z}{1}$, let $r\in\numQ$, $\mathbf{w}, \mathbf{x}\in\numQ^n$, and let $\mathbf{t}=(t_1,t_2,\dots,t_n)\in\numgeq{Z}{1}^n$. The main object of this paper is the sum
\begin{equation}
\label{eq:GaussSumDef}
\mathfrak{G}_{G}(r,\mathbf{t};\mathbf{w},\mathbf{x}):=\sum_{v_1=0}^{t_1-1}\dots\sum_{v_n=0}^{t_n-1}\etp{r\cdot\left(\frac{1}{2}(\mathbf{v}+\mathbf{w})^\tp\cdot G\cdot (\mathbf{v}+\mathbf{w})+(\mathbf{v}+\mathbf{w})^\tp\cdot\mathbf{x}\right)},
\end{equation}
where $\mathbf{v}=(v_1,v_2,\dots,v_n)$, $\etp{a}:=\exp(2\uppi\rmi a)$ and where vectors are interpreted as column vectors and $\mathbf{v}^\tp$ is the matrix transpose. Moreover, if $r=a/c$ with $a\in\numZ$, $c\in\numgeq{Z}{1}$ and $a,c$ coprime, set
\begin{align}
\mathfrak{G}_{G}(r;\mathbf{w},\mathbf{x})=\mathfrak{G}_{G}(a/c;\mathbf{w},\mathbf{x})&:=\mathfrak{G}_{G}(r,(c,c,\dots,c);\mathbf{w},\mathbf{x}),\label{eq:GaussSumDef2}\\
\mathfrak{G}_{G}(r;\mathbf{w})=\mathfrak{G}_{G}(a/c;\mathbf{w})&:=\mathfrak{G}_{G}(a/c;\mathbf{w},\mathbf{0}).\notag
\end{align}
We call $\mathfrak{G}_{G}(r,\mathbf{t};\mathbf{w},\mathbf{x})$ a \emph{(quadratic) Gauss sum} in $n$ variables for if $n=1$, $G=(2)$, $\mathbf{w}=\mathbf{x}=\mathbf{0}$, then
\begin{equation*}
\mathfrak{G}_{(2)}(a/c;0,0)=\sum_{v=0}^{c-1}\etp{\frac{av^2}{c}}
\end{equation*}
is the classical quadratic Gauss sum, which plays an important role in number theory, especially in the theory of quadratic forms, Diophantine equations over finite fields and their associated zeta functions; cf. \cite[Chapter 1]{BEW98}.

The book just mentioned is concerned with Gauss, Jacobi and closely related sums and gives the most extensive and systematic treatment of these subjects. Theorem 1.2.2 of this book is a reciprocity theorem for a kind of generalized Gauss sum
\begin{equation*}
S(a,b,c):=\sum_{n=0}^{\abs{c}-1}\rme^{\uppi\rmi(an^2+bn)/c}.
\end{equation*}
This sum can also be interpreted as a one-dimensional case of $\mathfrak{G}_{G}(r,\mathbf{t};\mathbf{w},\mathbf{x})$. For instance, if $a$, $c$ are coprime integers with $c>0$ and $a$ odd, then
\begin{equation*}
S(a,b,c)=\mathfrak{G}_{(2)}\left(\frac{a}{2c};0,\frac{b}{a}\right).
\end{equation*}

Let us mention another type of Gauss sums that can be treated as a special case of $\mathfrak{G}_{G}(r,\mathbf{t};\mathbf{w},\mathbf{x})$, namely, the quadratic Gauss sums over finite fields. They are defined by (cf. \cite[p. 11]{BEW98})
\begin{equation*}
g(\beta):=\sum_{x\in \mathbb{F}_q}\etp{\frac{\tr(\beta x^2)}{p}},\qquad (\beta\in\mathbb{F}_q)
\end{equation*}
where $p$ is a prime, $q=p^r$ is a power of $p$, $\mathbb{F}_q$ is the finite field of $q$ elements and $\tr\colon\mathbb{F}_q\to\mathbb{F}_p$ is the trace map. Let $(e_1,e_2\dots,e_r)$ be an $\mathbb{F}_p$-basis of $\mathbb{F}_q$. Set $\widetilde{G}=(2\tr(\beta e_ie_j))_{1\leq i,j\leq r}$, which is a symmetric matrix of size $r$ over $\mathbb{F}_p$. Let $G$ be a lifting of $\widetilde{G}$ in the ring of integral matrices. We can require that $G$ is still symmetric. If $p>2$ we can also require that $G$ is positive definite and even integral. In this setting we have
\begin{equation}
\label{eq:GaussSumFq}
g(\beta)=\mathfrak{G}_{G}(1/p;0,0).
\end{equation}
This sum plays an important role in many problems in number theory, for instance, in determining the zeta functions of diagonal quadratic projective hypersurfaces; cf. the $m=2$ case of \cite[Theorem 2, p. 162]{IR90}.

Many other types of quadratic Gauss sums, besides the above ones mentioned, can be regarded as special cases of \eqref{eq:GaussSumDef}. We claim another two. One is the Hecke Gauss sum for number fields. See Section \ref{subsec:hecke_gauss_sums} for this account. Another one is the quadratic Gauss sum with summation range being a matrix ring over a finite field (cf. \cite{Kur04}), which is a special case of the Gauss character sum introduced in \cite{Sai91}.

To further illustrate the importance of \eqref{eq:GaussSumDef}, note that special cases of \eqref{eq:GaussSumDef} show up in the transformation equations of theta series associated with lattices, and consequently in the functional equations and residues of corresponding Epstein zeta functions or $L$-functions. See \cite[Lemma 7.4]{Zhu23} for the appearance in theta series, and \cite[Theorems 2 and 3]{CS76} for that in $L$-functions. See also the more recent work of Hu and Lao \cite[Theorem 1.1]{HL24}, where a special case of \eqref{eq:GaussSumDef} shows up in the main term of the mean value of certain shifted convolution sum.

\subsection{The main theorem}
Our main theorem gives, under certain weak restrictions, an explicit formula for $\mathfrak{G}_{G}(r,\mathbf{t};\mathbf{w},\mathbf{x})$, which comprises known formulas of special cases in the literature and has the advantage that one need not first calculate any local data. By Proposition \ref{prop:ttoc} we need only focus on \eqref{eq:GaussSumDef2}. Of course, one can assume that $\mathbf{w}=\mathbf{0}$ without loss of generality, but it is more natural to state the duality theorem (Theorem \ref{thm:duality1}) using the parameter $\mathbf{w}$.
\begin{thm}
\label{thm:main}
Let $G$ be a positive definite even integral symmetric matrix of size $n\in\numgeq{Z}{1}$. Let $c$ be a positive integer and $a$ be a nonzero integer coprime to $c$. Let $D=\det(G)$ and $N$ be the level of $G$ (which means the least positive integer such that $N\cdot G^{-1}$ is even integral). Suppose that
\begin{equation*}
G\cdot\mathbf{w}\in\numZ^n,\qquad a\cdot\mathbf{x}\in\numZ^n.
\end{equation*}
Then:

(1) If $\gcd(N,c)=1$, we have
\begin{multline}
\label{eq:GGacwxWhenNccoprime}
\mathfrak{G}_{G}(a/c;\mathbf{w},\mathbf{x})=c^{n/2}\cdot\legendre{\abs{a}}{c}_K^n\cdot\legendre{D}{c}_K\cdot\etp{\frac{n\cdot\sgn a\cdot(1-c_0)}{8}}\\
\times\etp{-\frac{a}{2c}\mathbf{x}^\tp G^{-1}\mathbf{x}+\frac{ac'}{2}(\mathbf{w}+G^{-1}\mathbf{x})^\tp G (\mathbf{w}+G^{-1}\mathbf{x})},
\end{multline}
where $c_0$ is the odd part of $c$, and $c'$ is any solution of $cc'\equiv 1\bmod{aN}$ and where $\legendre{\cdot}{\cdot}_K$ is the Kronecker-Jacobi symbol.

(2) If $N\mid c$, we have\footnote{All fractional powers of a complex number are understood to be the principal branch. Hence $\sqrt{\rmi}=\etp{1/4}$ and $\rmi^{n/2}=\etp{n/8}$.}
\begin{equation}
\label{eq:GGacwxWhenNdivc}
\mathfrak{G}_{G}(a/c;\mathbf{w},\mathbf{x})=\delta\cdot c^{n/2}\cdot D^{1/2}\cdot\rmi^{n/2}\cdot \mu_G(a,c) \cdot\etp{-\frac{a'a^2\mathbf{x}^\tp G^{-1}\mathbf{x}}{2c}},
\end{equation}
where $a'$ is any solution of $aa'\equiv 1\bmod{cN}$, $\delta=1$ if $a(\mathbf{w}+G^{-1}\mathbf{x})\in\numZ^n$ and $\delta=0$ otherwise, and where
\begin{equation}
\label{eq:mainB}
\mu_G(a,c)=\begin{dcases}
\legendre{-\sgn a\cdot c}{\abs{a}}_K^n\cdot\legendre{D}{a}_K\cdot\etp{\frac{n\cdot\sgn a\cdot(1-a)}{8}} & \text{if }2\nmid a,\\
\legendre{a}{D}_K & \text{if }2\mid a.
\end{dcases}
\end{equation}
\end{thm}
The proof will be given in Section \ref{sec:main_theorem_explicit_formulas_for_quadratic_gauss_sums_over_numz_n_c_numz_n_} after necessary tools have been introduced in Sections \ref{sec:weil_representations_and_finite_quadratic_modules} and \ref{sec:two_auxiliary_gauss_sums_and_an_extension_of_milgram_s_formula}. During the proof, we establish another formula, which may be of independent interest, as follows.
\begin{thm}
\label{thm:MilgramExtension}
Let $L$ be an even integral lattice in $\numR^n$ equipped with the standard inner product and let $L^\sharp$ be the dual lattice. Set $D=\vol{\numR^n/L}^2=\abs{L^\sharp/L}$. Then for all integers $c$ coprime to $D$, we have
\begin{equation*}
\sum_{\mathbf{v}\in L^\sharp/L}\etp{\frac{c\mathbf{v}^\tp\mathbf{v}}{2}}=\sqrt{D}\cdot\legendre{D}{c}_K\cdot\etp{\frac{nc_0}{8}}
\end{equation*}
where $c_0$ is the odd part of $c$.
\end{thm}
When $c=\pm1$, the above formula reduces to the positive definite case of Milgram's formula; cf. \cite[Appendix 4]{MH73}. For the notation and the proof of this theorem (including the indefinite case), see Section \ref{sec:two_auxiliary_gauss_sums_and_an_extension_of_milgram_s_formula}. Note that Scheithauer \cite[Theorem 3.9]{Sch09} and Str\"omberg \cite[Coro. 3.11]{Str13} both gave closed formulas for $\sum_{\mathbf{v}\in L^\sharp/L}\etp{\frac{c\mathbf{v}^\tp\mathbf{v}}{2}}$ where $c\in\numZ$ but their formulas require firstly finding out a Jordan decomposition of $L^\sharp/L$, which makes them not so explicit as Theorem \ref{thm:MilgramExtension}. In addition, when $D$ is odd, Theorem \ref{thm:MilgramExtension} is equivalent to the positive definite case of \cite[Lemma 3.9]{Str13}, while when $D$ is even, Str\"omberg's formula need to first evaluate the signature of the $2$-adic Jordan component.

Some remarks of Theorem \ref{thm:main} are in order. For the formula of $\mathfrak{G}_{G}(a/c;\mathbf{w},\mathbf{x})$ in the case $G$ is indefinite, see Theorem \ref{thm:mainIndefinite}, which is an immediate consequence of Theorem \ref{thm:main}. For the case $\gcd(N,c)\neq 1,\,N$, there is a formula that requires first calculating the Jordan decomposition of $G^{-1}\numZ^n/\numZ^n$, as a consequence of the works of Scheithauer \cite{Sch09} and of Str\"omberg \cite{Str13}. The interested reader can obtain this formula by, e.g., combining \eqref{eq:GaussSumWeilRepr} and \cite[Remark 6.8]{Str13}.

Below is a list of some special cases and analogs of Theorem \ref{thm:main} that appear in the literature.
\begin{enumerate}
	\item Set $G=(2)$ and $\mathbf{w}=\mathbf{x}=\mathbf{0}$; then $N=4$, $D=2$ and Theorem \ref{thm:main} gives
	\begin{equation}
	\label{eq:classicalGauss}
	\sum_{v=0}^{c-1}\etp{\frac{av^2}{c}}=\begin{dcases}
	\sqrt{c}\cdot\legendre{2a}{c}_K\cdot\etp{\frac{1-c}{8}} & \text{if }2\nmid c,\\
	\sqrt{2c}\cdot\legendre{2c}{a}_K\cdot\etp{\frac{a}{8}} & \text{if } 4\mid c.
	\end{dcases}
	\end{equation}
	This is a reformulation of classical results due to Gauss (cf. \cite[\S 1.5]{BEW98}). Gauss also proved that $\sum_{v=0}^{c-1}\etp{\frac{av^2}{c}}=0$ if $c\equiv2\bmod{4}$.
	\item The sum in the case $n=2$ is called a double Gauss sum according to Weber and Jordan around 1870. They also derived some formulas in some special cases. To our best knowledge, the first general explicit formula for such sums (with $\mathbf{w}=\mathbf{x}=\mathbf{0}$) was obtained by Skoruppa and Zagier \cite[Theorem 3]{SZ89}. Their formula deals with all positive integer $c$, with no restriction on $\gcd(N,c)$. Alaca, Alaca and Williams \cite{AAW14}, and Alaca and Doyle \cite{AD17} provided explicit formulas (for $c=p^\alpha$, $p$ an odd prime) in a different form, together with applications to determining the numbers of solutions of certain quadratic congruences.
	\item The sum considered by Callahan and Smith \cite[eq. (6)]{CS76} (also see Stark \cite{Sta67}) is, in our notation, $\mathfrak{G}_{G}(a/c;\mathbf{0},\mathbf{x})$. They did not give an explicit formula but reduce the calculation of $\mathfrak{G}_{G}(a/c;\mathbf{0},\mathbf{x})$ to $\mathfrak{G}_{G}(1/c;\mathbf{0},\mathbf{0})$ under the restriction $a\cdot\mathbf{x}\in\numZ^n$, $c$ is odd and $\gcd(N,c)=1$. Iwaniec and Kowalski \cite[Lemma 20.13]{IK04} gave an explicit formula under the same restriction, as a lemma to the Kloosterman circle method. Note that their proof does not work if $c$ is even.
	\item A formula for $\mathfrak{G}_{G}(a/c;\mathbf{w},\mathbf{0})$ was given when $G\cdot\mathbf{w}\in\numZ^n$, $c$ is odd and $\gcd(N,c)=1$ in \cite[p. 572]{CS17}. The method is essentially the same as \cite[Lemma 20.13]{IK04}.
\end{enumerate}

We give four applications of Theorem \ref{thm:main}.

\subsection{A duality theorem}
Under certain assumptions (see Lemma \ref{lemm:independentRepr}) $\mathfrak{G}_{G}(a/c;\mathbf{w},\mathbf{x})$ can be seen as a sum with summation range $\numZ^n/c\numZ^n$. Sometimes one is interested in a subsum (partial sum), namely, a sum $\sum_{\mathbf{v}\in H}$ where $H$ is a subset of $\numZ^n/c\numZ^n$. The precise definition is
\begin{equation}
\label{eq:defGaussH}
\mathfrak{G}_{G}^H(r;\mathbf{w}):=\sum_{\mathbf{v}\in H}\etp{r\cdot\left(\frac{1}{2}(\mathbf{v}+\mathbf{w})^\tp\cdot G\cdot (\mathbf{v}+\mathbf{w})\right)}.
\end{equation}
Provided that $G\cdot\mathbf{w}\in\numZ^n$ and $H\subseteq\numZ^n/c\numZ^n$, this sum is well-defined (see Corollary \ref{coro:welldefinedGaussH}). Moreover, we define
\begin{equation}
\label{eq:Hbot}
H^{\bot}:=\left\{\mathbf{v}\in\numZ^n/c\numZ^n\colon\mathbf{v}^\tp\mathbf{w}=0+c\numZ^n\text{ for all }\mathbf{w}\in H\right\},
\end{equation}
and $G^\bot=N\cdot G^{-1}$ where $N$ is the level of $G$. (For more details on the bilinear map module $\numZ^n/c\numZ^n$ and its Fourier analysis, see Section \ref{sec:weil_representations_and_finite_quadratic_modules}.) As the first application of Theorem \ref{thm:main}, we obtain:
\begin{thm}
\label{thm:duality1}
Let $G$, $D$, $N$, $n$, $a$, $c$, $c_0$, $\mathbf{w}$ be as in Theorem \ref{thm:main}, except that $G$ is not required to be positive definite but required merely to be nonsingular. Let $H$ be a subgroup of $\numZ^n/c\numZ^n$ (not only a subset). Then if $\gcd(N,c)=1$, we have
\begin{equation}
\label{eq:duality1}
\frac{1}{\sqrt{\abs{H}}}\cdot\mathfrak{G}_{G}^H(a/c;\mathbf{w})=\gamma\cdot\frac{1}{\sqrt{\abs{H^\bot}}}\cdot\mathfrak{G}_{G^\bot}^{H^\bot}(a^\bot/c;-aG\mathbf{w})
\end{equation}
where $a^\bot$ is any integer satisfying $a^\bot aN\equiv-1\bmod{c}$ and where $\gamma$ is a root of unity given by
\begin{equation*}
\gamma=\legendre{\abs{a}}{c}_K^n\cdot\legendre{D}{c}_K\cdot\etp{\frac{n\cdot\sgn a\cdot(1-c_0)}{8}}\cdot\etp{\frac{a}{2c}\mathbf{w}^\tp G\mathbf{w}}.
\end{equation*}
\end{thm}
The proof is provided in Section \ref{sec:application_i_a_duality_theorem_and_gauss_subsums}, where a more involved formula under the assumption $N\mid c$ is given and proved as well (see Theorem \ref{thm:duality2}). Examples and interesting corollaries can be found in Section \ref{sec:intermission_examples_of_the_main_theorem_and_the_duality_theorem}.

\subsection{A formula for Weil representations}
The Weil representations of the metaplectic group of a symplectic space $W$ are introduced by Weil \cite{Wei64} in 1964. We are concerned with $W=\numZ$. In this case, with each positive definite even integral lattice $\underline{L}$, one can associate a Weil representation $\rho_{\underline{L}}\colon\mathrm{Mp}_2(\numZ)\to\mathrm{GL}(\numC[L^\sharp/L])$. For the notations, definitions and basic properties, see Section \ref{sec:weil_representations_and_finite_quadratic_modules}. Good references are \cite[\S 1.1]{Bru02}, \cite[\S3]{Sko08} and \cite[\S 14.5]{CS17}. Weil representations (with $W=\numZ$) play an important role in the theory of vector-valued modular forms, Borcherds products and modular forms on orthogonal groups; cf., e.g., \cite{Bor98,Sch09,WW25}.

Shintani \cite[Prop. 1.6]{Shi75} first expressed the coefficients of $\rho_{\underline{L}}(\gamma)$ in terms of \eqref{eq:GaussSumDef} for all $\gamma\in\mathrm{Mp}_2(\numZ)$, where he did not further simplify \eqref{eq:GaussSumDef}. Borcherds \cite[\S 3]{Bor00} pointed out the zero coefficients of $\rho_{\underline{L}}(\gamma)$. Scheithauer \cite[Theorem 4.7]{Sch09} calculated the coefficients under the restriction that $\rho_{\underline{L}}$ factors through $\slZ$ and Str\"omberg \cite[Theorem 6.4]{Str13} obtained a formula without this restriction.

The formulas of Scheithauer and Str\"omberg mentioned above both depend on a (not necessarily unique) Jordan decomposition of the finite quadratic module $L^\sharp/L$, a fact that makes their formulas not very elementary and explicit. Boylan \cite[Theorem 3.1, p. 82]{Boy15} gave another formula, which, although does not contain any local data, involves a limit of certain theta series. (Boylan's result is concerned with lattices over rings of algebraic integers.) As our second application of Theorem \ref{thm:main}, we give an explicit formula without referring to any local data and limits of theta series, under the assumption $\gcd(N,c)=1$:
\begin{thm}
\label{thm:WeilCoef}
Let $\underline{L}=(L,B)$ be a positive definite even integral lattice\footnote{The notations are given in \S \ref{subsec:lattices_and_weil_representations}.} of rank $n\in\numgeq{Z}{1}$ and set $Q(x)=\frac{1}{2}B(x,x)$. Let $A=\tbtmat{a}{b}{c}{d}\in\slZ$ with $c>0$, $a\neq0$ and $\gcd(N,c)=1$ where $N$ is the level of $\underline{L}$. Set $D=\det(\underline{L})=\abs{L^\sharp/L}$. Then for $w,v\in L^\sharp/L$ we have
\begin{multline*}
\rho_{\underline{L}}\left(A,1\right)_{w,v}=\frac{1}{\sqrt{D}}\cdot(-\rmi)^{\frac{n}{2}}\cdot\legendre{\abs{a}}{c}_K^n\cdot\legendre{D}{c}_K\cdot\etp{\frac{n\cdot\sgn a\cdot(1-c_0)}{8}}\\
\times\etp{\frac{b+c'}{a}Q(v)-c'B(w,v)+ac'Q(w)},
\end{multline*}
where $c_0$ is the odd part of $c$, and where $c'$ is any solution of $cc'\equiv 1\bmod{aN}$.
\end{thm}
The proof is provided in Section \ref{sec:application_ii_explicit_formulas_for_coefficients_of_weil_representations_of_slz_}, where a formula in the case $N\mid c$, and a translation of these formulas into the language of modular transformation equations of theta functions associated with lattices can be found.

\subsection{Affine quadratic hypersurfaces over $\mathbb{F}_p$} Let $G$ be a nonsingular even integral symmetric matrix of size $n\in\numgeq{Z}{1}$. Let $\mathbf{v}\in\numZ^n$, $m$ be an integer, and $p$ be a prime. We consider the congruence
\begin{equation}
\label{eq:affineQuadratic}
\frac{1}{2}\mathbf{x}^\tp G\mathbf{x}+\mathbf{v}^\tp\mathbf{x}\equiv m\bmod{p}.
\end{equation}
Let $r_{G,\mathbf{v}}(m)$ denote the number of $\mathbf{x}\in\mathrm{A}^n(\mathbb{F}_p):=\numZ^n/p\numZ^n$ such that \eqref{eq:affineQuadratic} holds. We are concerned with explicit formulas for $r_{G,\mathbf{v}}(m)$. When $G$ is diagonal, this problem is widely studied in the literature. Indeed, a special case of the well-known Hasse-Davenport relation (cf. \cite[p. 162]{IR90}) together with known explicit formula for Gauss sums over finite fields (cf. \cite[Theorem 11.5.4]{BEW98}) would imply an explicit formula for $r_{G,\mathbf{v}}(m)$. Li and Ouyang \cite{LO18} gave formulas and an algorithm for the number of solutions of a diagonal equation modulo a prime power. For the quadratic case of their results, see \cite[Theorem 4.4]{LO18}.

To illustrate the power of Theorem \ref{thm:main} (actually, of Theorem \ref{thm:mainIndefinite}), we first give a formula for $r_{G,\mathbf{v}}(m)$ which is valid for possibly non-diagonal $G$, and for all $\mathbf{v}$ and $m$ (see Theorem \ref{thm:affineQuadraticmodulusc}) where the modulus $p$ can be a general positive integer. The arithmetic complexity of this formula (with $G$, $\mathbf{v}$ and $m$ fixed) is $\Theta(p)$. By specializing we deduce succinct formulas for $p$ being prime provided that $p\nmid \det(G)$:
\begin{thm}
\label{thm:rGvm}
Let $G$ be a nonsingular even integral symmetric matrix of size $n\in\numgeq{Z}{1}$, let $p$ be a prime, and let $\mathbf{v}$, $m$ be as above. Let $N$ be the level of $G$, $D=\det(G)$ and $G^\bot=N\cdot G^{-1}$. Suppose that $p\nmid D$ (which is equivalent to $p\nmid N$). Then:

(1) If $2\mid n$ and $p>2$, then
\begin{equation*}
r_{G,\mathbf{v}}(m)=p^{n-1}+\legendre{D}{p}_K\cdot\etp{\frac{n(1-p)}{8}}\cdot p^{\frac{n}{2}-1}(\delta p-1)
\end{equation*}
where $\delta=1$ if $Nm+\frac{1}{2}\mathbf{v}^\tp G^\bot\mathbf{v}\equiv0\bmod{p}$, and $\delta=0$ otherwise.

(2) If $2\mid n$ and $p=2$, then
\begin{equation*}
r_{G,\mathbf{v}}(m)=2^{n-1}+\legendre{D}{2}_K\cdot(-1)^{Nm+\frac{1}{2}\mathbf{v}^\tp G^\bot\mathbf{v}}\cdot2^{\frac{n}{2}-1}.
\end{equation*}

(3) If $2\nmid n$, then we have $2\mid N$ (namely, $2\mid D$), and hence $p>2$ and
\begin{equation*}
r_{G,\mathbf{v}}(m)=p^{n-1}+\legendre{2D}{p}_K\cdot\legendre{-t_m}{p}_K\cdot\etp{\frac{(n+1)(1-p)}{8}}\cdot p^{\frac{n-1}{2}}
\end{equation*}
where $t_m\in\numZ$ is any integer with $Nt_m\equiv Nm+\frac{1}{2}\mathbf{v}^\tp G^\bot\mathbf{v}\bmod p$.
\end{thm}
The proof is provided in Section \ref{sec:application_iii_number_of_points_of_arbitrary_affine_quadratic_hypersurfaces_over_prime_finite_fields}, where formulas of $r_{G,\mathbf{v}}(m)$ for more general moduli are discussed.

\subsection{Generalized Markoff equations over $\mathbb{F}_p$}
\label{subsec:generalized_markoff_equations}
Let $a_{11}$, $a_{22}$, $a_{33}$, $a_{12}$, $a_{13}$, $a_{23}$, $d$ be integers. We consider the Diophantine equation
\begin{equation}
\label{eq:gMarkoff}
a_{11}x^2+a_{22}y^2+a_{33}z^2+a_{12}xy+a_{13}xz+a_{23}yz=dxyz.
\end{equation}
When $a_{11}=a_{22}=a_{33}=1$, $a_{12}=a_{13}=a_{23}=0$ and $d=3$, this is called the Markoff equation. Together with the associated notions Markoff numbers and Markoff graphs, the Markoff equation arises in many areas of mathematics, e.g., in Lobachevsky geometry and symplectic geometry; cf., e.g., \cite{Aig13,Via16}. In 1991, Baragar \cite{Bar91} conjectured that the Markoff graph over $\mathbb{F}_p$ is connected. Significant progress toward this conjecture was made in 2016 by Bourgain, Gamburd, and Sarnak \cite{BGS16}. Recently, Chen \cite{Che24} proved that Baragar's conjecture is true for all but finitely many primes $p$. See \cite{Mar25} for a one-page proof of Chen's theorem.

If $a_{11}=a_{22}=a_{33}=1$, $a_{12}=a_{13}=a_{23}=0$ and $d\neq0$, then \eqref{eq:gMarkoff} is called a Hurwitz equation; if merely $a_{12}=a_{13}=a_{23}=0$, then \eqref{eq:gMarkoff} is called a Rosenberger variation; cf. \cite[p. 1 and 57]{Bar91}. Here we shall consider \eqref{eq:gMarkoff} in most generality. As is the case of the Markoff equation we have seen above, it is important to study the equation \eqref{eq:gMarkoff} over $\mathbb{F}_p$. Our final application of Theorem \ref{thm:main} is a uniform formula for the solution number of \eqref{eq:gMarkoff} over $\mathbb{F}_p$. Set
\begin{equation*}
G=\begin{pmatrix}
2a_{11} & a_{12} & a_{13} \\
a_{12} & 2a_{22} & a_{23} \\
a_{13} & a_{23} & 2a_{33}
\end{pmatrix},
\end{equation*}
and set
$$A_{11}=4a_{22}a_{33}-a_{23}^2,\quad A_{22}=4a_{11}a_{33}-a_{13}^2, \quad A_{33}=4a_{11}a_{22}-a_{12}^2,\quad D=\det(G).$$
\begin{thm}
\label{thm:solnumbergMarkoff}
Suppose that $p$ is a prime with $p\nmid a_{11}a_{22}a_{33}d$.

(1) If $p\nmid D$, then $p>2$, and the number of solutions of \eqref{eq:gMarkoff} mod $p$ equals
\begin{equation*}
p^2+\left(\legendre{-A_{11}}{p}_K+\legendre{-A_{22}}{p}_K+\legendre{-A_{33}}{p}_K\right)\cdot p+1.
\end{equation*}

(2) If $2<p\mid D$ and at least two of $A_{11}, A_{22}, A_{33}$ are divisible by $p$, then the solution number is
\begin{equation*}
p^2+1.
\end{equation*}

(3) If $2<p\mid D$ and at most one of $A_{11}, A_{22}, A_{33}$ is divisible by $p$, then the solution number is
\begin{equation*}
p^2+\left(\legendre{-A_{11}}{p}_K+\legendre{-A_{22}}{p}_K+\legendre{-A_{33}}{p}_K-\legendre{-A_{jj}}{p}_K\right)\cdot p+1,
\end{equation*}
where $1\leq j\leq3$ is any index such that $p\nmid A_{jj}$.

(4) If $2=p\mid D$, then the solution number is
\begin{equation*}
\begin{dcases}
5 &\text{ if }a_{12}\equiv a_{13}\equiv a_{23}\equiv0\bmod{2}\\
1 &\text{ if }a_{12}\equiv a_{13}\equiv a_{23}\equiv1\bmod{2}\\
3 &\text{ otherwise.}
\end{dcases}
\end{equation*}
\end{thm}

The proof is provided in Section \ref{sec:application_iv_number_of_solutions_of_generalized_markoff_equations_over_prime_finite_fields}. In the case of Markoff equation, Theorem \ref{thm:solnumbergMarkoff} implies that its solution number is $p^2+(-1/p)_K\cdot 3p+1$ over $\mathbb{F}_p$ where $p>3$. This was first obtained by Baragar \cite[p. 121]{Bar91}.

\subsection{Other results, and notations}
\label{subsec:Other_results_and_notations}
Other results, besides the above mentioned ones, are collected below:
\begin{itemize}
	\item Lemma \ref{lemm:twoGaussSumRelation} relates Gauss sums over $\numZ^n/c\numZ^n$ to Gauss sums over the discriminant modules of lattices.
	\item Lemma \ref{lemm:GGUforGeneralc} gives a formula for \eqref{eq:defGaussH} where no assumption on $\gcd(N,c)$ is required.
	\item In \S\ref{subsec:Anlattices} we present an example showing how our main theorem looks like for the root lattices $A_n$.
	\item Theorem \ref{thm:GaussSubsumExplicit} gives an explicit formula for \eqref{eq:defGaussH} when $H^\bot$ is cyclic.
	\item Proposition \ref{prop:HeckeGaussToMatrixGauss} relates Hecke Gauss sums (Gauss sums over number fields) to $\mathfrak{G}_{G}(a/c;\mathbf{0})$.
	\item Theorem \ref{thm:explicitHeckeGaussQuadratic} gives an explicit formula for Hecke Gauss sums over quadratic fields under certain restrictions. (The first explicit formula was given by Boylan and Skoruppa \cite{BS10}.)
	\item Proposition \ref{prop:HeckeGaussCyclotomic} shows explicitly the relation between Hecke Gauss sums over cyclotomic fields and $\mathfrak{G}_{G}(a/c;\mathbf{0})$. In particular, the matrix $G$ is determined.
	\item In Example \ref{examp:Q13HeckeGaussSum}, we compute the exact value of a particular Hecke Gauss sum over $\numQ(\zeta_{13})$.
	\item We provide an equivalent form of Theorem \ref{thm:WeilCoef}, namely, Theorem \ref{thm:thetaTransformation}, where explicit Fourier expansions of the actions of modular matrices on Jacobi theta series of lattice index are presented.
\end{itemize}

All notations are introduced at their first appearance, except the following ones. The greatest common divisor of two integers $a$, $b$ are denoted by $\gcd(a,b)$, or merely $(a,b)$. For a real number $x$, we define $\sgn{x}=1$ if $x>0$, $\sgn{x}=-1$ if $x<0$, and $\sgn{x}=0$ if $x=0$. For a finite set $S$, let $\abs{S}$ denote its cardinality. For sets $A$ and $B$, set $A\setminus B=\{x\in A\colon x\not\in B\}$. For an abelian group $G$ and a subgroup $H$, $G/H$ denotes the quotient group. The notation $\numgeq{Z}{a}$ denotes the set of integers greater than or equal to $a$. For an $m\times n$ matrix $A$, define $A\numZ^n:=\{A\mathbf{x}\colon\mathbf{x}\in\numZ^n\}\subseteq\numZ^m$. (When a vector is involved in a matrix multiplication, it is interpreted as a column vector.) The Kronecker-Jacobi symbol is used throughout the paper. It is defined by the following sentences:
\begin{itemize}
\item $\legendre{m}{p}_K$ is the usual Legendre symbol if $p$ is an odd prime.
\item $\legendre{m}{2}_K$ equals $0$ if $2 \mid m$, and equals $(-1)^{(m^2-1)/8}$ if $2 \nmid m$.
\item $\legendre{m}{-1}_K$ equals $1$ if $m \geq 0$, and equals $-1$ otherwise.
\item $\legendre{m}{1}_K=1$ by convention.
\item $\legendre{m}{n}_K$ is defined to make it a complete multiplicative function of $n \in \numZ\setminus\{0\}$.
\item $\legendre{m}{0}_K=0$ if $m \neq \pm1$, and $\legendre{\pm1}{0}_K=1$.
\end{itemize}

Throughout the paper, we use the computer algebra system SageMath \cite{Sage} to do the computations and check the theorems.

\tableofcontents

\section{Weil representations and finite quadratic modules}
\label{sec:weil_representations_and_finite_quadratic_modules}
In this section we recall some basic definitions and facts on lattices, Weil representations, finite bilinear map modules, finite quadratic modules, and Fourier analysis on finite abelian groups. They serve as tools to prove Theorems \ref{thm:main}--\ref{thm:rGvm}. For references on lattices, finite quadratic modules and Weil representations, cf. e.g. \cite{Nik79}, \cite[\S 1.1]{Bru02}, \cite[\S 3]{Sko08}, \cite{Str13}, \cite{ES17}, \cite[\S 14.3 and 14.5]{CS17}, and \cite[\S2, 3]{MS25}. For references on Fourier analysis on finite abelian groups, cf. e.g. \cite[Part One]{Ter99}, \cite[Chap. 4]{Nat00} 

\subsection{Lattices and Weil representations}
\label{subsec:lattices_and_weil_representations}
By a \emph{lattice} $\underline{L}=(L,B)$, we understand a free $\numZ$-module $L$ of finite rank $n\in\numgeq{Z}{1}$ equipped with a nondegenerate symmetric bilinear form $B\colon L\times L\to\numR$. If an ambient space $\numR^n$ is given, where an inner product (nondegenerate symmetric bilinear form) $B$ is equipped, we say a subset $L\subseteq \numR^n$ is a \emph{lattice} in $(\numR^n,B)$ if $L$ is a $\numZ$-submodule of $\numR^n$ such that there is a $\numZ$-basis $(e_1,e_2,\dots,e_n)$ of $L$ that is simultaneously a $\numR$-basis of $\numR^n$. These two languages can be exchanged. The ambient space $\numR^n$ can also be replaced by $\numQ^n$, which will be used in \S\ref{subsec:hecke_gauss_sums}.

Throughout the paper, we set $Q(x)=\frac{1}{2}B(x,x)$, which is a quadratic form on $L$. The lattice $\underline{L}=(L,B)$ is called \emph{integral} if $B(L,L)\subseteq\numZ$ and is called \emph{even integral} if $Q(L)\subseteq\numZ$. If it is integral but not even then it is called odd. (Some authors say type II and type I instead of even and odd respectively.) It is called \emph{positive definite} if $Q(x)>0$ for all $0\neq x\in L$.

If $L$ is defined as a lattice in $(\numR^n,B)$, then the \emph{dual lattice} is defined as $L^\sharp:=\{x\in\numR^n\colon B(x,y)\in\numZ \text{ for all }y\in L\}$. It turns out that $L^\sharp$ is again a lattice in $(\numR^n,B)$, and $L$ is integral if and only if $L\subseteq L^\sharp$. For an even integral lattice $L$, the \emph{level} $N$ is defined as the least positive integer such that $N\cdot Q(x)\in\numZ$ for all $x\in L^\sharp$. To check this inclusion for all $x\in L^\sharp$, it suffices to check it for all $x\in L^\sharp/L$, which is a finite abelian group called the \emph{discriminant module} of $L$. Let $E$ be the exponent of $L^\sharp/L$, namely, $E$ is the least positive integer such that $E\cdot L^\sharp\subseteq L$ and let $D=\abs{L^\sharp/L}$. Then $N$, $D$ and $E$ have the same set of prime factors. This follows from the structure theory of finite abelian groups and \cite[Remark 14.3.23]{CS17}.

If $G$ is a positive definite even integral symmetric matrix of size $n\in\numgeq{Z}{1}$, then $B_G\colon(\mathbf{x},\mathbf{y})\mapsto\mathbf{x}^\tp G\mathbf{y}$ defines a (positive definite) inner product on $\numR^n$, and $\numZ^n$ is a positive definite even integral lattice in $(\numR^n, B_G)$. The dual lattice is then $G^{-1}\numZ^n:=\{G^{-1}\mathbf{v}\colon\mathbf{v}\in\numZ^n\}$, and the discriminant module is $G^{-1}\numZ^n/\numZ^n$. We have $\det(G)=\abs{G^{-1}\numZ^n/\numZ^n}$ and the level of $\numZ^n$ (as a lattice in $(\numR^n, B_G)$) equals the level of $G$. The above discussion remains true if $G$ is not assumed to be positive definite but merely nondegenerate.

Let $\underline{L}=(L,B)$ be a positive definite even integral lattice. (Although one can define Weil representations for arbitrary even integral lattices, or for finite quadratic modules, we need only the positive definite case.) Let $\numC[L^\sharp/L]$ be the group algebra, whose elements are maps from $L^\sharp/L$ to $\numC$. For $x\in L^\sharp/L$, let $\delta_x\colon L^\sharp/L\to\numC$ be the map $\delta_x(y)=\delta_{x,y}$, where the right-hand side is the Kronecker-delta. Then $(\delta_x)_{x\in L^\sharp/L}$ is a $\numC$-basis of $\numC[L^\sharp/L]$. The algebra $\numC[L^\sharp/L]$ is a unital associative algebra over $\numC$, and becomes a Hilbert space under the inner product $\langle \sum_{x}c_x\delta_x,\sum_{x}d_x\delta_x\rangle=\sum_xc_x\overline{d_x}$.

Let $\slZ$ be the group of integral $2\times2$ matrices of determinant $1$. The double cover $\mathrm{Mp}_2(\numZ)$ is the group of $(A,\varepsilon)$ where $A\in\slZ$ and $\varepsilon\in\{\pm1\}$. The composition is given by
$$\left(\tbtmat{a_1}{b_1}{c_1}{d_1},\varepsilon_1\right)\cdot\left(\tbtmat{a_2}{b_2}{c_2}{d_2},\varepsilon_1\right)=\left(\tbtmat{a_3}{b_3}{c_3}{d_3},\varepsilon_1\varepsilon_2\delta\right)$$
where $\tbtmat{a_3}{b_3}{c_3}{d_3}=\tbtmat{a_1}{b_1}{c_1}{d_1}\cdot\tbtmat{a_2}{b_2}{c_2}{d_2}$ and
\begin{equation*}
\delta=\frac{\sqrt{c_1(a_2\tau+b_2)/(c_2\tau+d_2)+d_1}\sqrt{c_2\tau+d_2}}{\sqrt{c_1(a_2\tau+b_2)+d_1(c_2\tau+d_2)}}
\end{equation*}
with $\Im\tau>0$. It turns out that $\delta$ is independent of the choice of $\tau$; cf., e.g., \cite[Theorem 4.1]{Str13} or \cite[\S 2]{Zhu25}. Alternatively, $\mathrm{Mp}_2(\numZ)$ is the nontrivial central extension of $\slZ$ by the group $\{\pm1\}$. The group $\mathrm{Mp}_2(\numZ)$ has a presentation with generators $\widetilde{T}=\left(\tbtmat{1}{1}{0}{1},1\right)$, $\widetilde{S}=\left(\tbtmat{0}{-1}{1}{0},1\right)$ and relations $\widetilde{S}^8=(I,1)$, $(\widetilde{S}\widetilde{T})^3=\widetilde{S}^2$, $\widetilde{S}^4\widetilde{T}=\widetilde{T}\widetilde{S}^4$. See \cite[Lemma 5.2]{Zhu23_2} (with $D=2$ there) for a detailed proof. (In fact, there in the case $D=2$ the third relation can be discarded since it can be generated by the first two.)

Now let $\underline{L}=(L,B)$ be a positive definite even integral lattice and define $t\in\mathrm{GL}(\numC[L^\sharp/L])$ and $s\in\mathrm{GL}(\numC[L^\sharp/L])$ by $t\delta_x=\etp{Q(x)}\delta_x$ and
\begin{equation}
\label{eq:Weils}
s\delta_x=\frac{(-\rmi)^{n/2}}{\sqrt{\abs{L^\sharp/L}}}\sum_{y\in L^\sharp/L}\etp{-B(x,y)}\delta_y.
\end{equation}
A direct calculation shows that $s^8=\mathrm{id}$, $(st)^3=s^2$, $s^4t=ts^4$. (In the calculation of $s^2$, one may use Milgram's formula; cf. \cite[p. 127]{MH73}.) Hence if we define $\rho_{\underline{L}}(\widetilde{T})=t$ and $\rho_{\underline{L}}(\widetilde{S})=s$ then it extends uniquely to a representation $\rho_{\underline{L}}\colon\mathrm{Mp}_2(\numZ)\to\mathrm{GL}(\numC[L^\sharp/L])$, called the Weil representation associated with $\underline{L}$. It is a unitary representation, namely, we have $\langle\rho_{\underline{L}}(\gamma)v,\langle\rho_{\underline{L}}(\gamma)w\rangle=\langle v,w\rangle$ for all $\gamma\in \mathrm{Mp}_2(\numZ)$ and all $v,w\in \numC[L^\sharp/L]$.

For $y,x\in L^\sharp/L$ and $\gamma\in\mathrm{Mp}_2(\numZ)$, we define $\rho_{\underline{L}}(\gamma)_{y,x}\in\numC$ by
\begin{equation*}
\rho_{\underline{L}}(\gamma)\delta_x=\sum_{y\in L^\sharp/L}\delta_y\cdot\rho_{\underline{L}}(\gamma)_{y,x}.
\end{equation*}
In other words, $(\rho_{\underline{L}}(\gamma)_{y,x})_{y,x\in L^\sharp/L}$ is the matrix representation of $\rho_{\underline{L}}(\gamma)$ with respect to the basis $(\delta_y)_{y\in L^\sharp/L}$. Furthermore, if $y,x\in L^\sharp$, then we define $\rho_{\underline{L}}(\gamma)_{y,x}:=\rho_{\underline{L}}(\gamma)_{y+L,x+L}$. The following formula is essentially due to Shintani \cite[Prop. 1.6]{Shi75}.
\begin{prop}[Shintani]
Let $\underline{L}=(L,B)$ be a positive definite even integral lattice of rank $n$. Let $A=\tbtmat{a}{b}{c}{d}\in\slZ$ with $c\neq0$ and let $y, x\in L^\sharp/L$. Then
\begin{equation}
\label{eq:WeilReprShintani}
\rho_{\underline{L}}(A,1)_{y,x}=(-\rmi\sgn{c})^{\frac{n}{2}}\abs{L^\sharp/L}^{-\frac{1}{2}}\abs{c}^{-\frac{n}{2}}\cdot\sum_{t\in L/cL}\etp{\frac{aQ(y+t)-B(y+t,x)+dQ(x)}{c}}.
\end{equation}
\end{prop}
A brief proof can be found in \cite[Eq. (7.9)]{Zhu23} and a detailed proof in \cite[Theorem 14.3.7]{CS17}. Both proofs and Shintani's original proof regard $\rho_{\underline{L}}(\gamma)_{y,x}$ as coefficients of the transformation equations of theta series associated with $\underline{L}$; cf. \cite[Theorem 5.1]{Zhu23} or \cite[\S 3.5]{Boy15}.

\subsection{Finite quadratic modules}
This subsection contains some tools needed for the proofs of Theorems \ref{thm:MilgramExtension} and \ref{thm:duality1}. Let $M$ be a $\numZ$-module equipped with a symmetric $\numZ$-bilinear map $M\times M\rightarrow N$ where $N$ is a fixed $\numZ$-module. Set $\underline{M}=(M,B)$, which is called a symmetric bilinear map $\numZ$-module. We say $\underline{M}$ is \emph{nondegenerate} if the map $M\rightarrow\Hom(M,N)$, $m_1\mapsto(m_2\mapsto B(m_1,m_2))$ is injective; it is called \emph{strongly nondegenerate} if this map is bijective. Let $H$ be a submodule of $M$; we define $H^\bot=\{m\in M\colon B(m,H)=0\}$, which is also a submodule of $M$. Two subsets $S_1$, $S_2$ of $M$ are called \emph{orthogonal} if $B(S_1,S_2)=0$. Let $(M_1,B_1)$ and $(M_2,B_2)$ be symmetric bilinear map $\numZ$-modules. They are called \emph{isometrically isomorphic} if there is a $\numZ$-linear isomorphism $\sigma\colon M_1\rightarrow M_2$ such that $B_1(x,y)=B_2(\sigma(x),\sigma(y))$ for all $x,y\in M_1$. This is denoted by $(M_1,B_1)\cong(M_2,B_2)$.
\begin{examp}
\label{examp:ZncZn}
Let $M=\numZ^n/c\numZ^n$ where $c$ and $n$ are positive integers. We equip it with the bilinear map $B\colon\numZ^n/c\numZ^n\times \numZ^n/c\numZ^n\rightarrow \numZ/c\numZ$ defined by
\begin{equation*}
B((v_1,v_2,\dots,v_n)+c\numZ^n, (w_1,w_2,\dots,w_n)+c\numZ^n)=v_1w_1+v_2w_2+\dots v_nw_n+c\numZ.
\end{equation*}
It is immediate that $(\numZ^n/c\numZ^n,B)$ is nondegenerate. Indeed, it is strongly nondegenerate since we have
\begin{equation*}
\abs{\Hom(\numZ^n/c\numZ^n,\numZ/c\numZ)}=\abs{\widehat{\numZ^n/c\numZ^n}}=\abs{\numZ^n/c\numZ^n},
\end{equation*}
where $\widehat{M}$ is the dual group of $M$. Note that the first equality follows from the isomorphism
\begin{align*}
\Hom(\numZ^n/c\numZ^n,\numZ/c\numZ)&\rightarrow\widehat{\numZ^n/c\numZ^n}\\
\sigma &\mapsto (x\mapsto\etp{\sigma(x)/c}).
\end{align*}
Now let $H$ be a submodule of $\numZ^n/c\numZ^n$. Then $H^\bot$ we just defined is equivalent to \eqref{eq:Hbot}.
\end{examp}

\begin{lemm}
\label{lemm:HHbot}
Let $\underline{M}=(M,B\colon M\times M\rightarrow N)$ be a nondegenerate symmetric bilinear map $\numZ$-module with the restrictions $\abs{M}<+\infty$ and $N$ can be embedded into $\numC^\times$, the multiplicative group of nonzero complex numbers. Then for each submodule $H\subseteq M$ we have
\begin{equation*}
\abs{H}\cdot\abs{H^\bot}=\abs{M}\qquad\text{and}\qquad H^{\bot\bot}=H.
\end{equation*}
\end{lemm}
\begin{proof}
When $N=\numQ/\numZ$ (which can be embedded into $\numC^\times$ by $x\mapsto\etp{x}$), proofs can be found in, e.g., \cite[Prop. 1.7]{Boy15} or \cite[Prop. 3.10]{Zhu24}, where quadratic map modules are considered. However, those proofs can be easily adapted for bilinear map modules. For general $N$, let $N_0$ be the submodule of $N$ generated by $B(m_1,m_2)$, $m_1,m_2\in M$. Since $\abs{M}<+\infty$, $N_0$ is a finitely generated $\numZ$-module. Since $B(m_1,m_2)$ are all torsion element, $N_0$ is a torsion module. Hence $N_0$ is a finite module. This fact, together with the fact $N_0$ can be embedded into $\numC^\times$, implies that $N_0$ can be embedded into $\numQ/\numZ$. Let $\sigma$ denote this embedding. Then $(M,B'\colon M\times M\rightarrow\numQ/\numZ)$ is a nondegenerate symmetric bilinear map $\numZ$-module, where $B'(m_1,m_2)=\sigma(B(m_1,m_2))$. Since $H^\bot$ is independent of the choice of $B$ and $B'$, the conclusion follows from the special case $N=\numQ/\numZ$, which has already been proved.
\end{proof}
\begin{rema}
In the above lemma, if additionally $N$ is a submodule of $\numQ/\numZ$, then $\underline{M}$ is nondegenerate would imply that it is strongly nondegenerate; cf., e.g., \cite[Corollary 3.9]{Zhu24}.
\end{rema}

\begin{coro}
\label{coro:HHbotequalcn}
Let $H$ be a subgroup of $\numZ^n/c\numZ^n$, where $n,c$ are positive integers. Then
$$
\abs{H}\cdot\abs{H^\bot}=c^n.
$$
\end{coro}
\begin{proof}
We let $(M,B)$ be the one discussed in Example \ref{examp:ZncZn}, and then apply Lemma \ref{lemm:HHbot}.
\end{proof}

Let $M$ be a finite $\numZ$-module and let $Q\colon M\rightarrow\numQ/\numZ$ be a quadratic map, which means $Q(rx)=r^2Q(x)$ for all $r\in\numZ$ and $x\in M$ and the associated map $B_Q\colon M\times M\rightarrow\numQ/\numZ$, $(x,y)\mapsto Q(x+y)-Q(x)-Q(y)$ is $\numZ$-bilinear. The bilinear map $B_Q$ is automatically symmetric. We call the pair $\underline{M}=(M,Q)$ a \emph{finite quadratic module} if $B_Q$ is nondegenerate. The structure theory of finite quadratic modules was developed by Skoruppa \cite{Sko25}.
\begin{thm}
\label{thm:indecomposableFQM}
Each finite quadratic module is isometrically isomorphic to the orthogonal direct sum\footnote{For basic definitions and properties of external or internal orthogonal direct sums, cf. \cite[p. 608]{Zhu24}} of some of the following indecomposable modules:
\begin{align*}
\underline{A}_{p^r}^a &:= \left(\mathbb{Z}/p^r\mathbb{Z},\,x+p^r\mathbb{Z} \mapsto \frac{a}{p^r}x^2+\mathbb{Z}\right) \, \text{for prime }p>2, \\
\underline{A}_{2^r}^a &:= \left(\mathbb{Z}/2^r\mathbb{Z},\,x+2^r\mathbb{Z} \mapsto \frac{a}{2^{r+1}}x^2+\mathbb{Z}\right) \, \text{for } p=2, \\
\underline{B}_{2^r}   &:= \left(\mathbb{Z}/2^r\mathbb{Z} \times \mathbb{Z}/2^r\mathbb{Z},\,(x+2^r\mathbb{Z},y+2^r\mathbb{Z}) \mapsto \frac{x^2+xy+y^2}{2^r}+\mathbb{Z} \right), \\
\underline{C}_{2^r}   &:= \left(\mathbb{Z}/2^r\mathbb{Z} \times \mathbb{Z}/2^r\mathbb{Z},\,(x+2^r\mathbb{Z},y+2^r\mathbb{Z}) \mapsto \frac{xy}{2^r}+\mathbb{Z}\right).
\end{align*}
\end{thm}
Proofs can be found in, e.g., \cite[Prop. 2.11]{Str13} or \cite[Theorem 3.19]{Zhu24}. Note that we say a finite quadratic module $M$ is \emph{indecomposable} if there are no nontrivial submodules $M_1$ and $M_2$ such that $M=M_1\oplus M_2$ and $M_1\perp M_2$. For instance, $\underline{C}_{2^r}$ is indecomposable as one can verify directly, although as a mere abelian group, it can be decomposed as a direct sum of two copies of $\numZ/2^r\numZ$.

\begin{examp}
\label{examp:discmodule}
Let $\underline{L}=(L,B)$ be an even integral lattice and set $Q(v)=\frac{1}{2}B(v,v)$ where $v\in L$. Then $L^\sharp/ L$ is a finite $\numZ$-module since it is finitely generated and all elements are torsion. By abuse of notation, we define $Q\colon L^\sharp/ L\rightarrow\numQ/\numZ$ by $Q(v+L)=Q(v)+\numZ$. The condition ``even integral'' ensures $Q$ is well-defined. The associated bilinear map is $B_Q(v_1+L,v_2+L)=B(v_1,v_2)+\numZ$. It is straightforward that $(L^\sharp/ L, Q)$ is a finite quadratic module. Conversely, Wall \cite{Wal63} proved that each finite quadratic module can be realized as a discriminant module of an even integral lattice. See \cite[\S 4 and 5]{Zhu24} for explicit realizations.
\end{examp}

\subsection{Fourier analysis on finite abelian groups}
\label{subsec:fourier_analysis_on_finite_abelian_groups}
The materials of this subsection are only needed in the proof of Theorem \ref{thm:duality1}. The proofs of all assertions here can be found in, e.g., \cite[\S 4.3 and 4.4]{Nat00}.

Let $M$ be a finite abelian group (whose composition is denoted by addition $+$). As is customary in harmonic analysis, let $L^2(M)$ be the space of all complex-valued functions on $M$, equipped with an inner product $(f_1,f_2):=\sum_{m\in M}f_1(m)\overline{f_2(m)}$. Let $\widehat{M}$ be the dual group that consists of all complex characters of $M$. Then $\widehat{M}$ is (non-canonically) isomorphic to $M$. The well-known orthogonality relations of characters in $\widehat{M}$ implies that $\widehat{M}$ is an orthogonal basis of $L^2(M)$. The \emph{Fourier transform} is the linear transformation from $L^2(M)$ to $L^2(\widehat{M})$ that sends $f\in L^2(M)$ to $\widehat{f}\in L^2(\widehat{M})$, where
\begin{equation*}
\widehat{f}(\chi)=(f,\chi)=\sum_{m\in M}f(m)\overline{\chi(m)}.
\end{equation*}

Let $H$ be a subgroup of $M$ and let
\begin{equation}
\label{eq:defMdualH}
\widehat{M}^{H}=\{\chi\in\widehat{M}\colon\chi(h)=1\text{ for all }h\in H\}.
\end{equation}
\begin{thm}[Poisson summation formula]
\label{thm:PoissonSummationFinite}
Let $M$ and $H$ be as above. If $f\in L^2(M)$ and $x\in M$, then
\begin{equation*}
\frac{1}{\abs{H}}\sum_{y\in H}f(x+y)=\frac{1}{\abs{M}}\sum_{\chi\in\widehat{M}^{H}}\widehat{f}(\chi)\overline{\chi(x)}.
\end{equation*}
\end{thm}
\begin{proof}
See \cite[p. 142]{Nat00}.
\end{proof}

\section{Two auxiliary Gauss sums and an extension of Milgram's formula}
\label{sec:two_auxiliary_gauss_sums_and_an_extension_of_milgram_s_formula}
In this section, we begin with some basic definitions and properties on the main object $\mathfrak{G}_{G}(r,\mathbf{t};\mathbf{w},\mathbf{x})$, and then introduce another two types of Gauss sums, whose evaluations serve as necessary lemmas to Theorem \ref{thm:main}. In particular, we will prove Theorem \ref{thm:MilgramExtension} in the third subsection.

\subsection{Basic properties of $\mathfrak{G}_{G}(r,\mathbf{t};\mathbf{w},\mathbf{x})$}
The notations $\mathfrak{G}_{G}(r,\mathbf{t};\mathbf{w},\mathbf{x})$, $\mathfrak{G}_{G}(r;\mathbf{w},\mathbf{x})$ and $\mathfrak{G}_{G}(r;\mathbf{w})$, and the related notations $n$, $G$, $r$, $\mathbf{t}$, $\mathbf{w}$, $\mathbf{x}$ and $\mathbf{v}$, have been introduced around \eqref{eq:GaussSumDef}. Here we continue to use these notations. We emphasize that $G$ is only assumed to be symmetric and even integral; we do not require positive-definiteness unless explicitly declared. Set $r=a/c$ with $a\in\numZ$, $c\in\numgeq{Z}{1}$ and $a,c$ coprime.
\begin{deff}
\label{deff:intPara}
We say $\mathfrak{G}_{G}(r,\mathbf{t};\mathbf{w},\mathbf{x})$ is \emph{integral-parametric} if
\begin{equation}
\label{eq:intPara}
r\mathbf{t}\in\numZ^n,\qquad G\mathbf{w}\in\numZ^n\quad\text{and}\quad a\mathbf{x}\in\numZ^n.
\end{equation}
\end{deff}

Thus, $\mathfrak{G}_{G}(r;\mathbf{w},\mathbf{x})$ is integral-parametric if and only if $G\mathbf{w}\in\numZ^n$ and $a\mathbf{x}\in\numZ^n$; $\mathfrak{G}_{G}(r;\mathbf{w})$ is integral-parametric if and only if $G\mathbf{w}\in\numZ^n$. Also note that the condition $r\mathbf{t}\in\numZ^n$ in \eqref{eq:intPara} is equivalent to $\mathbf{t}\in c\numZ^n$.

\begin{lemm}
\label{lemm:independentRepr}
Let $\mathfrak{G}_{G}(r,\mathbf{t};\mathbf{w},\mathbf{x})$ be integral-parametric. For each $j=1,2,\dots,n$, let $T_j$ be a complete system of residues modulo $t_j$. Then
\begin{equation*}
\mathfrak{G}_{G}(r,\mathbf{t};\mathbf{w},\mathbf{x})=\sum_{v_1\in T_1}\dots\sum_{v_n\in T_n}\etp{r\cdot\left(\frac{1}{2}(\mathbf{v}+\mathbf{w})^\tp\cdot G\cdot (\mathbf{v}+\mathbf{w})+(\mathbf{v}+\mathbf{w})^\tp\cdot\mathbf{x}\right)}.
\end{equation*}
In other words, the sum $\mathfrak{G}_{G}(r,\mathbf{t};\mathbf{w},\mathbf{x})$ is independent of the choices of the representatives of $v_j\in\numZ/t_j\numZ$.
\end{lemm}
\begin{proof}
Indeed, a stronger conclusion holds. If $\mathbf{v}_1-\mathbf{v}_2=\Delta\mathbf{v}\in\prod_{j=1}^nt_j\numZ$, then it follows from \eqref{eq:intPara} and $G$ is even integral symmetric that
\begin{align*}
&\etp{r\cdot\left(\frac{1}{2}(\mathbf{v}_1+\mathbf{w})^\tp G(\mathbf{v}_1+\mathbf{w})+(\mathbf{v}_1+\mathbf{w})^\tp\mathbf{x}\right)}\\
=&\etp{r\cdot\left(\frac{1}{2}(\mathbf{v}_2+\mathbf{w})^\tp G(\mathbf{v}_2+\mathbf{w})+\Delta\mathbf{v}^\tp G\mathbf{w}+\frac{1}{2}\Delta\mathbf{v}^\tp G\Delta\mathbf{v}+(\mathbf{v}_2+\mathbf{w})^\tp\cdot\mathbf{x}+\Delta\mathbf{v}^\tp\mathbf{x}}\right)\\
=&\etp{r\cdot\left(\frac{1}{2}(\mathbf{v}_2+\mathbf{w})^\tp G(\mathbf{v}_2+\mathbf{w})+(\mathbf{v}_2+\mathbf{w})^\tp\mathbf{x}\right)}.
\end{align*}
The desired conclusion thus follows.
\end{proof}

\begin{coro}
\label{coro:welldefinedGaussH}
Suppose $\mathfrak{G}_{G}(r;\mathbf{w})$ is integral-parametric. Then for all subsets $H\subseteq\numZ^n/c\numZ^n$, the sum $\mathfrak{G}_{G}^H(r;\mathbf{w})$ (see \eqref{eq:defGaussH}) is well-defined.
\end{coro}
\begin{proof}
We only need to repeat the proof of Lemma \ref{lemm:independentRepr}.
\end{proof}
\begin{rema}
We will encounter a non-integral-parametric $\mathfrak{G}_{G}(r;\mathbf{w})$, in which case \eqref{eq:defGaussH} is not well-defined for general $H$. However, it is still well-defined for some specific nontrivial $H$. See the paragraph followed by Lemma \ref{lemm:nonintegralpara}.
\end{rema}

\begin{prop}
\label{prop:ttoc}
Let $\mathfrak{G}_{G}(r,\mathbf{t};\mathbf{w},\mathbf{x})$ be integral-parametric. Then $\mathfrak{G}_{G}(r;\mathbf{w},\mathbf{x})$ is also integral-parametric, and
\begin{equation}
\label{eq:ttoc}
\mathfrak{G}_{G}(r,\mathbf{t};\mathbf{w},\mathbf{x})=\frac{\prod_{j=1}^nt_j}{c^n}\cdot\mathfrak{G}_{G}(r;\mathbf{w},\mathbf{x}).
\end{equation}
\end{prop}
\begin{proof}
The Gauss sum $\mathfrak{G}_{G}(r;\mathbf{w},\mathbf{x})$ is a sum over $\mathbf{v}\in\numZ^n/c\numZ^n$ by Lemma \ref{lemm:independentRepr}, so $\prod_{j=1}^nt_j/c^n\cdot\mathfrak{G}_{G}(r;\mathbf{w},\mathbf{x})$ can be regarded as a sum of $\prod_{j=1}^nt_j/c^n$ sums over $\numZ^n/c\numZ^n$. We can choose a particular set of representatives for $\numZ^n/c\numZ^n$ in each $\mathfrak{G}_{G}(r;\mathbf{w},\mathbf{x})$ such that they, when glued together, form a set of representatives of $\numZ^n/\prod_{j=1}^n(t_j\numZ)$. Since the general terms of both sides of \eqref{eq:ttoc} are the same, $\prod_{j=1}^nt_j/c^n\cdot\mathfrak{G}_{G}(r;\mathbf{w},\mathbf{x})$ equals the Gauss sum over $\numZ^n/\prod_{j=1}^n(t_j\numZ)$, which, according to Lemma \ref{lemm:independentRepr}, is $\mathfrak{G}_{G}(r,\mathbf{t};\mathbf{w},\mathbf{x})$.
\end{proof}

This proposition reduces the evaluation of integral-parametric $\mathfrak{G}_{G}(r,\mathbf{t};\mathbf{w},\mathbf{x})$ to the evaluation of $\mathfrak{G}_{G}(r;\mathbf{w},\mathbf{x})$, as is treated in Theorem \ref{thm:main}. Furthermore, the following basic fact reduces the problem to the case where $a=\pm1$.
\begin{prop}
\label{prop:GtoaG}
Let $c\in\numgeq{Z}{1}$ and $a\in\numZ$ with $\gcd(a,c)=1$. Suppose that $\mathfrak{G}_{G}(a/c;\mathbf{w},\mathbf{x})$ is integral-parametric. Then $\mathfrak{G}_{\abs{a}G}(\sgn{a}/c;\mathbf{w},\abs{a}\mathbf{x})$ is integral-parametric as well, and\footnote{We define $\sgn{0}=0$. Then this formula allows $a=0$.}
\begin{equation*}
\mathfrak{G}_{G}(a/c;\mathbf{w},\mathbf{x})=\mathfrak{G}_{\abs{a}G}(\sgn{a}/c;\mathbf{w},\abs{a}\mathbf{x}).
\end{equation*}
\end{prop}
\begin{proof}
This is immediate.
\end{proof}

The following lemma relates $\mathfrak{G}_{G}(r;\mathbf{w},\mathbf{x})$ to the coefficients of certain Weil representation. It plays a key role in proving Theorem \ref{thm:main}.
\begin{lemm}
\label{lemm:GaussSumWeilRepr}
Suppose that $G$ is a positive definite even integral symmetric matrix of size $n$, and that $\mathfrak{G}_{G}(a/c;\mathbf{w},\mathbf{x})$ is integral-parametric, where $a\in\numZ$, $c\in\numgeq{Z}{1}$ are coprime. Let $\underline{L}=(\numZ^n,B)$ where $B(\mathbf{v},\mathbf{w})=\mathbf{v}^\tp G\mathbf{w}$. Let $b$, $d$ be integers such that $ad-bc=1$ and let $A=\tbtmat{a}{b}{c}{d}$. Then
\begin{equation}
\label{eq:GaussSumWeilRepr}
\mathfrak{G}_{G}(a/c;\mathbf{w},\mathbf{x})=\etp{-\frac{a^2d}{2c}\mathbf{x}^\tp G^{-1}\mathbf{x}}\rmi^{\frac{n}{2}}c^{\frac{n}{2}}\sqrt{\det(G)}\rho_{\underline{L}}(A,1)_{\mathbf{w},-aG^{-1}\mathbf{x}}.
\end{equation}
\end{lemm}
\begin{proof}
Note that $L^\sharp=G^{-1}\numZ^n$ and $\abs{L^\sharp/L}=\det(G)$. Since $\mathfrak{G}_{G}(a/c;\mathbf{w},\mathbf{x})$ is integral-parametric, we have $\mathbf{w},-aG^{-1}\mathbf{x}\in L^\sharp$ and hence $\rho_{\underline{L}}(A,1)_{\mathbf{w},-aG^{-1}\mathbf{x}}$ makes sense. Set $Q(\mathbf{v})=\frac{1}{2}B(\mathbf{v},\mathbf{v})=\frac{1}{2}\mathbf{v}^\tp G\mathbf{v}$. A straightforward calculation shows that
\begin{equation*}
\sum_{\mathbf{v}\in L/cL}\etp{\frac{aQ(\mathbf{w}+\mathbf{v})-B(\mathbf{w}+\mathbf{v},-aG^{-1}\mathbf{x})+dQ(aG^{-1}\mathbf{x})}{c}}=\etp{\frac{d}{c}Q\left(aG^{-1}\mathbf{x}\right)}\mathfrak{G}_{G}(a/c;\mathbf{w},\mathbf{x}).
\end{equation*}
We then obtain \eqref{eq:GaussSumWeilRepr} by inserting the above formula into \eqref{eq:WeilReprShintani} with $y=\mathbf{w}$ and $x=-aG^{-1}\mathbf{x}$.
\end{proof}
\begin{rema}
By this lemma, the calculation of $\mathfrak{G}_{G}(a/c;\mathbf{w},\mathbf{x})$ and that of $\rho_{\underline{L}}\left(\tbtmat{a}{b}{c}{d},1\right)_{y,x}$ are essentially the same problem. Formulas for the latter problem with all Gauss sums explicitly evaluated are obtained by Scheithauer \cite{Sch09} for even $n$ and by Str\"omberg \cite{Str13} for general $n$. Thus, explicit formulas for integral-parametric $\mathfrak{G}_{G}(a/c;\mathbf{w},\mathbf{x})$ can be thought of to be already known. These formulas work very well if $G$ is fixed. However, to acquire the exact value of $\mathfrak{G}_{G}(a/c;\mathbf{w},\mathbf{x})$ in elementary symbols using Scheithauer and Str\"omberg's formulas, one need to first calculate a Jordan decomposition (an orthogonal direct sum decomposition with all summands being the form given in Theorem \ref{thm:indecomposableFQM}) of the finite quadratic module $L^\sharp/L$. To explain more concretely, we insert the formula in \cite[Remark 6.8]{Str13} into \eqref{eq:GaussSumWeilRepr}, and thus obtain a formula for $\mathfrak{G}_{G}(a/c;\mathbf{w},\mathbf{x})$ where all factors are elementary arithmetic functions except $\xi(a,c)$. One can see from \cite[Def. 6.1]{Str13} that it still need some effort to obtain the value of $\xi(a,c)$. Our objective is a formula for $\mathfrak{G}_{G}(a/c;\mathbf{w},\mathbf{x})$ that does not need to calculate any local data of $L^\sharp/L$. We achieve this in Theorem \ref{thm:main}, but with a restriction $\gcd(N,c)=1$ or $N$.
\end{rema}

By the above remark, we need more tools.

\subsection{Two auxiliary Gauss sums}
We need to calculate $\rho_{\underline{L}}(A,1)_{\mathbf{w},-aG^{-1}\mathbf{x}}$ by Lemma \ref{lemm:GaussSumWeilRepr}. For this purpose, we introduce the following Gauss-type sum.
\begin{deff}[Str\"omberg]
\label{def:GMcx}
Let $\underline{M}=(M,Q)$ be a finite quadratic module with associated bilinear map $B$. Let $c\in\numZ$ and $x\in M$. We define
\begin{equation*}
\mathscr{G}_{\underline{M}}(c,x)=\frac{1}{\sqrt{\abs{M}}\sqrt{\abs{M[c]}}}\sum_{y\in M}\etp{cQ(y)+B(x,y)},
\end{equation*}
where $M[c]$ is the kernel of the homomorphism $M\rightarrow M$, $y\mapsto cy$.
\end{deff}
The following notations are also needed:
\begin{align*}
M[c]^*&=cM=\{cy\colon y\in M\},\\
M[c]^\bullet&=\{y\in M\colon cQ(z)+B(y,z)=0+\numZ,\,\forall z\in M[c]\}.
\end{align*}
Scheithauer \cite[Prop. 2.1]{Sch09} showed that $M[c]^\bullet=y+M[c]^*$ for some $y\in M$. As a consequence, $\abs{M[c]^\bullet}=\abs{M[c]^*}$. Also note a nontrivial fact that $M[c]^\bullet=M[-c]^\bullet$, for $cQ(z)\in\frac{1}{2}\numZ/\numZ$ if $z\in M[c]=M[-c]$.

If $\underline{L}$ is a lattice with discriminant module $\underline{D}$ (which is a finite quadratic module by Example \ref{examp:discmodule}), then we set
\begin{equation*}
\mathscr{G}_{\underline{L}}(c,x):=\mathscr{G}_{\underline{D}}(c,x+L),\quad c\in\numZ,\,x\in L^\sharp.
\end{equation*}

\begin{prop}[Str\"omberg]
\label{prop:StrombergscrG}
If $x\not\in M[c]^\bullet$, then $\mathscr{G}_{\underline{M}}(c,x)=0$. Moreover, let $x_0\in M[c]^\bullet$. Then for each $x\in M[c]^\bullet$ there is a $y\in M$ such that $x=x_0+cy$ and for all such $y$ we have
\begin{equation*}
\mathscr{G}_{\underline{M}}(c,x)=\etp{-cQ(y)-B(x_0,y)}\mathscr{G}_{\underline{M}}(c,x_0).
\end{equation*}
\end{prop}
\begin{proof}
See \cite[Lemma 2.4]{CZ25}.
\end{proof}

\begin{prop}
\label{prop:GMmultiplicative}
Let $\underline{M}=(M,Q)$ be a finite quadratic module and let $M_1$, $M_2$ be submodules with $M_1\perp M_2$ and $M=M_1 \oplus M_2$. Set $\underline{M}_j=(M_j,Q\vert_{M_j})$, $j=1,2$. Then $\underline{M}_1$ and $\underline{M}_2$ themselves are finite quadratic modules, and for $x_1\in M_1$ and $x_2\in M_2$ we have
\begin{equation}
\label{eq:GMGM1GM2}
\mathscr{G}_{\underline{M}}(c,x_1+x_2)=\mathscr{G}_{\underline{M}_1}(c,x_1)\mathscr{G}_{\underline{M}_2}(c,x_2).
\end{equation}
\end{prop}
\begin{proof}
Let $B$ be the associated bilinear map of $\underline{M}=(M,Q)$. To show $\underline{M}_1$ is a finite quadratic module, it suffices to show $B\vert_{M_1\times M_1}$ is nondegenerate. Suppose the contrary; then there is $0\neq m\in M_1$ such that $m\perp M_1$. Since $M_1\perp M_2$, we have $m\perp M_1+M_2=M$, which contracts the assumption $\underline{M}$ is a finite quadratic module. Therefore, $\underline{M}_1$ is a finite quadratic module, and so is $\underline{M}_2$. This shows that $\mathscr{G}_{\underline{M}_1}(c,x_1)$ and $\mathscr{G}_{\underline{M}_2}(c,x_2)$ fall within the scope of Definition \ref{def:GMcx}.

By the assumptions we have $M[c]=M_1[c]\oplus M_2[c]$ and hence $\abs{M[c]}=\abs{M_1[c]}\cdot\abs{M_2[c]}$. Also we have $\abs{M}=\abs{M_1}\cdot \abs{M_2}$. Now expanding the definitions of the three sums in \eqref{eq:GMGM1GM2} and then substituting $y\in M$ with $y_1\in M_1,y_2\in M_2$, $y_1+y_2=y$, we obtain \eqref{eq:GMGM1GM2}.
\end{proof}

This proposition, together with Theorem \ref{thm:indecomposableFQM}, reduces the calculation of $\mathscr{G}_{\underline{M}}(c,x)$ to that of $\mathscr{G}_{\underline{A}_{p^r}^a}(c,x)$, $\mathscr{G}_{\underline{B}_{2^r}}(c,x)$ and $\mathscr{G}_{\underline{C}_{2^r}}(c,x)$, which was given in \cite[Corollary 3.11]{Str13}. We recall Str\"omberg's results in the case $x=0$ here (Propositions \ref{prop:GApr}--\ref{prop:GBC2r}).
\begin{prop}
\label{prop:GApr}
Let $p$ be an odd prime, $r\in\numgeq{Z}{1}$, $c\in\numZ$, and $a\in\numZ$ with $p\nmid a$. Then
\begin{equation*}
\mathscr{G}_{\underline{A}_{p^r}^a}(c,0)=\legendre{2ac/(p^r,c)}{p^r/(p^r,c)}_K\cdot\etp{\frac{1-p^r/(p^r,c)}{8}}.
\end{equation*}
\end{prop}
\begin{proof}
See \cite[Lemma 3.6]{Str13}.
\end{proof}
\begin{prop}
\label{prop:GA2r}
Let $r\in\numgeq{Z}{1}$, $c\in\numZ$, and $a\in\numZ$ with $2\nmid a$. Then
\begin{equation*}
\mathscr{G}_{\underline{A}_{2^r}^a}(c,0)=\begin{dcases}
\legendre{ac/(2^r,c)}{2^r/(2^r,c)}_K\cdot\etp{\frac{ac/(2^r,c)}{8}} &\text{if }2^r\nmid c,\\
0 &\text{if }2^r\parallel c,\\
1 &\text{if }2^{r+1}\mid c.
\end{dcases}
\end{equation*}
\end{prop}
\begin{proof}
See \cite[Lemma 3.7]{Str13}.
\end{proof}
\begin{prop}
\label{prop:GBC2r}
Let $r\in\numgeq{Z}{1}$ and $c\in\numZ$. Then
\begin{align*}
\mathscr{G}_{\underline{B}_{2^r}}(c,0)&=\legendre{3}{2^r/(2^r,c)}_K,\\
\mathscr{G}_{\underline{C}_{2^r}}(c,0)&=1.
\end{align*}
\end{prop}
\begin{proof}
See \cite[Lemma 3.8]{Str13}.
\end{proof}

Besides $\mathscr{G}_{\underline{M}}(c,x)$, we need the following type of Gauss sums:
\begin{deff}
Let $\underline{L}=(L,B)$ be an even integral lattice of rank $n\in\numgeq{Z}{1}$, and set $Q(v)=\frac{1}{2}B(v,v)$. Let $a,c\in\numZ$ with $\gcd(a,c)=1$ and $c\neq0$. Let $w\in L^\sharp$. We define
\begin{equation*}
\mathfrak{G}_{\underline{L}}(a/c;w)=\abs{c}^{-n/2}\sum_{v\in L/cL}\etp{\frac{a}{c}Q(w+v)}.
\end{equation*}
\end{deff}
Some remarks are in order. It is immediate that the sum $\sum_{v\in L/cL}\etp{\frac{a}{c}Q(w+v)}$ in above is well-defined. Moreover, if $w_1-w_2\in L$ then $\mathfrak{G}_{\underline{L}}(a/c;w_1)=\mathfrak{G}_{\underline{L}}(a/c;w_2)$, and hence $\mathfrak{G}_{\underline{L}}(a/c;w)$ is also defined if $w\in L^\sharp/L$. Finally, if $G$ is an even integral symmetric nondegenerate matrix of size $n$ and $c>0$, then $\mathfrak{G}_{\underline{L}}(a/c;\mathbf{w})=c^{-n/2}\mathfrak{G}_{G}(a/c;\mathbf{w})$ where $\underline{L}=(\numZ^n,(\mathbf{x},\mathbf{y})\mapsto\mathbf{x}^\tp G\mathbf{y})$ and where $\mathbf{w}\in L^\sharp=G^{-1}\numZ^n$.
\begin{prop}
\label{prop:GwtoG0}
Let $\underline{L}=(L,B)$ be an even integral lattice of rank $n\in\numgeq{Z}{1}$, and set $Q(v)=\frac{1}{2}B(v,v)$. Suppose $a,b,c,d\in\numZ$ and $w\in L^\sharp$ with $ad-bc=1$, $d\neq0$ and $bcw\in L$. Then
\begin{equation}
\label{eq:GwtoG0}
\mathfrak{G}_{\underline{L}}(b/d;w)=\etp{a^2bdQ(w)}\mathfrak{G}_{\underline{L}}(b/d;0).
\end{equation}
\end{prop}
\begin{proof}
Since $ad-bc=1$ we have
\begin{equation*}
\mathfrak{G}_{\underline{L}}(b/d;w)=\abs{d}^{-n/2}\sum_{v\in L/dL}\etp{\frac{b}{d}Q(adw+(v-bcw))}.
\end{equation*}
Since $bcw\in L$, the map $L/dL\rightarrow L/dL$, $v+dL\mapsto v-bcw+dL$ is well-defined and bijective. Therefore
\begin{align*}
\mathfrak{G}_{\underline{L}}(b/d;w)&=\abs{d}^{-n/2}\sum_{v\in L/dL}\etp{\frac{b}{d}Q(adw+v)}\\
&=\abs{d}^{-n/2}\sum_{v\in L/dL}\etp{\frac{bQ(adw)+bB(adw,v)+bQ(v)}{d}}\\
&=\etp{a^2bdQ(w)}\mathfrak{G}_{\underline{L}}(b/d;0).
\end{align*}
\end{proof}

The roles of $\mathscr{G}_{\underline{M}}(c,x)$ and $\mathfrak{G}_{\underline{L}}(a/c;w)$ are shown in the following proposition.
\begin{prop}
\label{prop:rhoA1yx2}
Let $\underline{L}=(L,B)$ be a positive definite even integral lattice of rank $n\in\numgeq{Z}{1}$, and set $Q(v)=\frac{1}{2}B(v,v)$. Let $D_{\underline{L}}$ be the discriminant module of $\underline{L}$ (which is a finite quadratic module by Example \ref{examp:discmodule}). Let $A=\tbtmat{a}{b}{c}{d}\in\slZ$ with $d\neq 0$ and $bc\cdot L^\sharp\subseteq L$. Let $y,x\in D_{\underline{L}}$. Then we have\footnote{Note that $\legendre{c}{-1}_K=1,-1$ if $c\geq0$ and $c<0$, respectively. If $c\neq0$ we can replace $\legendre{c}{-1}_K$ by $\sgn{c}$. We use $\legendre{c}{-1}_K$ here to make the formula valid also for $c=0$.}
\begin{equation}
\label{eq:rhoA1yx2}
\rho_{\underline{L}}(A,1)_{y,x}=(-\rmi)^{\frac{n(1-\sgn{d})}{2}\legendre{c}{-1}_K}\cdot\mathfrak{G}_{\underline{L}}(b/d;y)\sqrt{\frac{\abs{D_{\underline{L}}[c]}}{\abs{D_{\underline{L}}}}}\mathscr{G}_{D_{\underline{L}}}(-cd,y-dx).
\end{equation}
\end{prop}
\begin{proof}
Note that in $\mathrm{Mp}_2(\numZ)$ we have $(\tbtmat{a}{b}{c}{d},1)=(\tbtmat{-b}{a}{-d}{c},1)\cdot(\tbtmat{0}{-1}{1}{0},1)\cdot(\tbtmat{1}{0}{0}{1},\varepsilon(c,d))$, where $\varepsilon(c,d)=-1$ if both $c$ and $d$ are negative and $\varepsilon(c,d)=1$ otherwise. Therefore
\begin{equation}
\label{eq:rhoA1Temp}
\rho_{\underline{L}}(A,1)=\rho_{\underline{L}}(\tbtmat{-b}{a}{-d}{c},1)\rho_{\underline{L}}(\tbtmat{0}{-1}{1}{0},1)\rho_{\underline{L}}(\tbtmat{1}{0}{0}{1},\varepsilon(c,d)).
\end{equation}
If $\varepsilon(c,d)=1$ then
\begin{equation*}
\rho_{\underline{L}}(A,1)_{y,x}=\sum_{t\in D_{\underline{L}}}\rho_{\underline{L}}(\tbtmat{-b}{a}{-d}{c},1)_{y,t}\cdot\rho_{\underline{L}}(\tbtmat{0}{-1}{1}{0},1)_{t,x}.
\end{equation*}
Since $-d\neq0$, in above we replace $\rho_{\underline{L}}(\tbtmat{-b}{a}{-d}{c},1)_{y,t}$ by the right-hand side of \eqref{eq:WeilReprShintani} where $A,y,x$ are set to $\tbtmat{-b}{a}{-d}{c},y,t$ respectively, and replace $\rho_{\underline{L}}(\tbtmat{0}{-1}{1}{0},1)_{t,x}$ by the right-hand side of \eqref{eq:WeilReprShintani} where $A,y,x$ are set to $\tbtmat{0}{-1}{1}{0},t,x$ respectively. Thus we obtain
\begin{multline}
\label{eq:rhoA1yxTemp}
\rho_{\underline{L}}(A,1)_{y,x}=(-\rmi)^{\frac{n(1-\sgn{d})}{2}\legendre{c}{-1}_K}\cdot\abs{D_{\underline{L}}}^{-1}\abs{d}^{-n/2}\\
\times\sum_{t\in L^\sharp/L}\sum_{v\in L/dL}\etp{\frac{-bQ(y+v)-B(y+v,t)+cQ(t)}{-d}-B(t,x)},
\end{multline}
Since $bc\cdot L^\sharp\subseteq L$, the group homomorphism $D_{\underline{L}}\rightarrow D_{\underline{L}}$, $t\mapsto dt$ is an isomorphism with the inverse map $t\mapsto at$. Replacing $t$ by $dt$ in the right-hand side of \eqref{eq:rhoA1yxTemp} gives
\begin{multline*}
\rho_{\underline{L}}(A,1)_{y,x}=(-\rmi)^{\frac{n(1-\sgn{d})}{2}\legendre{c}{-1}_K}\\
\times\abs{d}^{-n/2}\sum_{v\in L/dL}\etp{\frac{b}{d}Q(y+v)}\abs{D_{\underline{L}}}^{-1}\sum_{t\in L^\sharp/L}\etp{-cdQ(t)+B(y-dx+v,t)}.
\end{multline*}
Inserting the definitions of $\mathfrak{G}_{\underline{L}}(b/d;y)$ and $\mathscr{G}_{D_{\underline{L}}}(-cd,y-dx)$ into the above equation and using the facts $D_{\underline{L}}[-cd]=D_{\underline{L}}[c]$ (because $bc\cdot L^\sharp\subseteq L$ and $ad-bc=1$) and $B(v,t)\in\numZ$, we obtain the desired formula.

Now consider the case $\varepsilon(c,d)=-1$, namely, $c$ and $d$ are negative. We have $(\tbtmat{1}{0}{0}{1},-1)=(\tbtmat{-1}{0}{0}{-1},1)^2=(\tbtmat{0}{-1}{1}{0},1)^4$. Thus $\rho_{\underline{L}}(\tbtmat{1}{0}{0}{1},-1)=s^4$ where $s$ is given by \eqref{eq:Weils}. It follows that $\rho_{\underline{L}}(\tbtmat{1}{0}{0}{1},-1)=(-1)^{n}\cdot \mathrm{Id}$ where $\mathrm{Id}$ is the identity element in $\mathrm{GL}(\numC[L^\sharp/L])$. It then follows from this and \eqref{eq:rhoA1Temp} that
\begin{equation*}
\rho_{\underline{L}}(A,1)_{y,x}=(-1)^n\sum_{t\in D_{\underline{L}}}\rho_{\underline{L}}(\tbtmat{-b}{a}{-d}{c},1)_{y,t}\cdot\rho_{\underline{L}}(\tbtmat{0}{-1}{1}{0},1)_{t,x}.
\end{equation*}
The remaining part of the proof is the same as in the case $\varepsilon(c,d)=1$, with the only difference being that the factor $(-\rmi)^{\frac{n(1-\sgn{d})}{2}\legendre{c}{-1}_K}$ now comes from $(-1)^n(-\rmi)^{n/2}(-\rmi)^{n/2}$.
\end{proof}
\begin{rema}
(1) See \cite[Lemma 2.14]{CZ25} for an alternative proof based on the Jacobi theta series associated with lattices.

(2) We have $$ay-x\not\in D_{\underline{L}}[c]^\bullet\Longleftrightarrow y-dx\not\in D_{\underline{L}}[-cd]^\bullet.$$
This can be proved by expanding the definitions of both sides, applying the change of variables $D_{\underline{L}}\rightarrow D_{\underline{L}}$, $t\mapsto dt$ and using the assumptions $bc\cdot L^\sharp\subseteq L$ and $ad-bc=1$.

(3) If $ay-x\not\in D_{\underline{L}}[c]^\bullet$, then $\mathscr{G}_{D_{\underline{L}}}(-cd,y-dx)=0$ by Proposition \ref{prop:StrombergscrG} and the last remark, and hence $\rho_{\underline{L}}(A,1)_{y,x}=0$ by \eqref{eq:rhoA1yx2}.

(4) If $\underline{L}=(\numZ^n,(\mathbf{x},\mathbf{y})\mapsto\mathbf{x}^\tp G\mathbf{y})$ where $G$ is an even integral symmetric nondegenerate matrix of size $n$, then $D_{\underline{L}}=G^{-1}\numZ^n/\numZ^n$, and $D_{\underline{L}}[c]=(G^{-1}\numZ^n\cap c^{-1}\numZ^n)/\numZ^n$.
\end{rema}

\subsection{Proof of Theorem \ref{thm:MilgramExtension}}
\label{subsec:proof_of_milgramextension}
Combining Lemma \ref{lemm:GaussSumWeilRepr} and Propositions \ref{prop:GwtoG0}, \ref{prop:rhoA1yx2}, we find that to deduce Theorem \ref{thm:main}, it suffices to calculate $\mathfrak{G}_{\underline{L}}(b/d;0)$ and $\mathscr{G}_{D_{\underline{L}}}(-cd,y-dx)$. The key step is Theorem \ref{thm:MilgramExtension}, which we restate in the following form.
\begin{thm}[Equivalent form of Theorem \ref{thm:MilgramExtension}]
\label{thm:MilgramExtension2}
Let $\underline{L}=(L,B)$ be a positive definite even integral lattice of rank $n\in\numgeq{Z}{1}$, and set $D=\det(\underline{L})=\abs{L^\sharp/L}$. If $c\in\numZ$ is coprime to $D$, then we have
\begin{equation}
\label{eq:MilgramExtension2}
\mathscr{G}_{\underline{L}}(c,0)=\legendre{D}{c}_K\cdot\etp{\frac{nc_0}{8}},
\end{equation}
where $c_0$ is the odd part of $c$.
\end{thm}
This theorem implies Theorem \ref{thm:MilgramExtension}, for we have $D_{\underline{L}}[c]=\{0\}$ since $\gcd(c,D)=1$.

It is possible to deduce Theorem \ref{thm:MilgramExtension2} from \cite[Lemma 3.9]{Str13}. However, since Str\"omberg omitted the proof of \cite[Lemma 3.9]{Str13}, we prefer to provide an independent proof with more details.

We begin with $\underline{L}$ whose discriminant module is simple, in which case we need the following lemma.
\begin{lemm}[Lemma 5.4 in \cite{Zhu24}]
\label{lemm:rankMod8OfPositive DefLattice}
Let $\underline{L}=(L,B_L)$ be a positive definite even integral lattice of rank $n \in \mathbb{Z}_{\geq 1}$, and $D_{\underline{L}}$ be its discriminant module. Let $\underline{M}=(M,Q_M)$ be a finite quadratic module. Let $p$ be a prime and $a$ be an integer with $p\nmid a$, and let $r$ be a positive integer. Suppose that $D_{\underline{L}}$ and $\underline{M}$ are isometrically isomorphic.
\begin{enumerate}
\item If $\underline{M}$ is trivial, i.e., $\vert M \vert=1$, then $n \equiv 0 \bmod 8$.
\item If $\underline{M}=\underline{A}_{p^r}^a$, with $r$ even and $p > 2$, then $n \equiv 0 \bmod 8$.
\item If $\underline{M}=\underline{A}_{p^r}^a$, with $r$ odd and $p > 2$, then $n \equiv 3 - \legendre{-1}{p}_K - 2\legendre{a}{p}_K \bmod 8$.
\item If $\underline{M}=\underline{A}_{2^r}^a$, with $r$ even and $\gcd(a,2)=1$, then $n \equiv a \bmod 8$.
\item If $\underline{M}=\underline{A}_{2^r}^a$, with $r$ odd and $\gcd(a,2)=1$, then $n \equiv a + 2 - 2\legendre{a}{2}_K \bmod 8$.
\item If $\underline{M}=\underline{B}_{2^r}$, then $n \equiv 2 - 2(-1)^r \bmod 8$.
\item If $\underline{M}=\underline{C}_{2^r}$, then $n \equiv 0 \bmod 8$.
\end{enumerate}
\end{lemm}
\begin{proof}
This is an immediate consequence of Milgram's formula and the $c=1$ cases of Propositions \ref{prop:GApr}--\ref{prop:GBC2r}. See \cite[Lemma 5.4]{Zhu24}.
\end{proof}

\begin{lemm}
\label{lemm:MilgramExtension2Indecomp}
Let the notations and assumptions be as in Theorem \ref{thm:MilgramExtension2}. Suppose that the discriminant module $D_{\underline{L}}$ of $\underline{L}$ is indecomposable as a finite quadratic module, then \eqref{eq:MilgramExtension2} holds.
\end{lemm}
\begin{proof}
The indecomposable finite quadratic modules are exactly those listed\footnote{For when any two of these indecomposable modules are isometrically isomorphic to each other, cf. \cite[Remark 3.13]{Zhu24}.} in Theorem \ref{thm:indecomposableFQM}. Hence the proof is divided into the following cases: (we suppose $c=2^ec_0$)

\textbf{Case (i)} $D_{\underline{L}}$ is trivial. Then $\mathscr{G}_{\underline{L}}(c,0)=1$. By Lemma \ref{lemm:rankMod8OfPositive DefLattice}(1) and the fact $D=1$, \eqref{eq:MilgramExtension2} holds in this case.

\textbf{Case (ii)} $D_{\underline{L}}\cong\underline{A}_{p^r}^a$ where $p$ is an odd prime, $r$ is even. By Proposition \ref{prop:GApr} and Lemma \ref{lemm:rankMod8OfPositive DefLattice}(2), \eqref{eq:MilgramExtension2} holds with both sides equal to $1$.

\textbf{Case (iii)} $D_{\underline{L}}\cong\underline{A}_{p^r}^a$ where $p$ is an odd prime, $r$ is odd, and $\legendre{a}{p}_K=1$. By Proposition \ref{prop:GApr} we have
\begin{equation}
\label{eq:indecompTemp1}
\mathscr{G}_{\underline{L}}(c,0)=\mathscr{G}_{\underline{A}_{p^r}^a}(c,0)=\legendre{2c}{p}_K\cdot\etp{\frac{1-p}{8}}=\legendre{2^{e+1}}{p}_K\legendre{c_0}{p}_K\cdot\etp{\frac{1-p}{8}}.
\end{equation}
By Lemma \ref{lemm:rankMod8OfPositive DefLattice}(3) we have
\begin{equation}
\label{eq:indecompTemp2}
\legendre{D}{c}_K\cdot\etp{\frac{nc_0}{8}}=\legendre{p}{2^e}_K\legendre{p}{c_0}_K\cdot\etp{\frac{\left(1-\legendre{-1}{p}_K\right)c_0}{8}}.
\end{equation}
If $p\equiv1\bmod 8$, then the right-hand sides of both \eqref{eq:indecompTemp1} and \eqref{eq:indecompTemp2} are equal to $\legendre{c_0}{p}_K$, and hence \eqref{eq:MilgramExtension2} holds. If $p\equiv3\bmod 8$, then the right-hand sides of both \eqref{eq:indecompTemp1} and \eqref{eq:indecompTemp2} are equal to $\rmi\cdot\legendre{2^{e}}{p}_K\legendre{c_0}{p}_K$, where we have used the Jacobi reciprocity law. Therefore, \eqref{eq:MilgramExtension2} holds again. The proofs for the cases $p\equiv5,7\bmod 8$ are similar.

\textbf{Case (iv)} $D_{\underline{L}}\cong\underline{A}_{p^r}^a$ where $p$ is an odd prime, $r$ is odd, and $\legendre{a}{p}_K=-1$. By Proposition \ref{prop:GApr} we have
\begin{equation*}
\mathscr{G}_{\underline{L}}(c,0)=\mathscr{G}_{\underline{A}_{p^r}^a}(c,0)=-\legendre{2^{e+1}}{p}_K\legendre{c_0}{p}_K\cdot\etp{\frac{1-p}{8}}.
\end{equation*}
By Lemma \ref{lemm:rankMod8OfPositive DefLattice}(3) we have
\begin{equation*}
\legendre{D}{c}_K\cdot\etp{\frac{nc_0}{8}}=\legendre{p}{2^e}_K\legendre{p}{c_0}_K\cdot\etp{\frac{\left(5-\legendre{-1}{p}_K\right)c_0}{8}}.
\end{equation*}
As in the case (iii), we divide the proof into four subcases according to $p\bmod8$, and thus find that \eqref{eq:MilgramExtension2} holds.

\textbf{Case (v)} $D_{\underline{L}}\cong\underline{A}_{2^r}^a$ where $r$ is even, and $2\nmid a$. By Proposition \ref{prop:GA2r},
\begin{equation}
\label{eq:indecompTemp3}
\mathscr{G}_{\underline{L}}(c,0)=\mathscr{G}_{\underline{A}_{2^r}^a}(c,0)=\etp{\frac{ac}{8}}=\etp{\frac{ac_0}{8}}.
\end{equation}
By Lemma \ref{lemm:rankMod8OfPositive DefLattice}(4),
\begin{equation}
\label{eq:indecompTemp4}
\legendre{D}{c}_K\cdot\etp{\frac{nc_0}{8}}=\etp{\frac{ac_0}{8}}.
\end{equation}
Now \eqref{eq:MilgramExtension2} follows from \eqref{eq:indecompTemp3} and \eqref{eq:indecompTemp4}.

\textbf{Case (vi)} $D_{\underline{L}}\cong\underline{A}_{2^r}^a$ where $r$ is odd, and $2\nmid a$. By Proposition \ref{prop:GA2r},
\begin{equation}
\label{eq:indecompTemp5}
\mathscr{G}_{\underline{L}}(c,0)=\mathscr{G}_{\underline{A}_{2^r}^a}(c,0)=\legendre{a}{2}_K\legendre{c_0}{2}_K\cdot\etp{\frac{ac_0}{8}}.
\end{equation}
By Lemma \ref{lemm:rankMod8OfPositive DefLattice}(5),
\begin{equation}
\label{eq:indecompTemp6}
\legendre{D}{c}_K\cdot\etp{\frac{nc_0}{8}}=\legendre{2}{c_0}_K\cdot\etp{\frac{\left(a + 2 - 2\legendre{a}{2}_K\right)c_0}{8}}.
\end{equation}
We divide the proof into four subcases according to $a\bmod8$, and thus find that the right-hand sides of \eqref{eq:indecompTemp5} and \eqref{eq:indecompTemp6} are equal to each other, from which \eqref{eq:MilgramExtension2} follows.

\textbf{Case (vii)} $D_{\underline{L}}\cong\underline{B}_{2^r}$. The proof is based on Proposition \ref{prop:GBC2r} and Lemma \ref{lemm:rankMod8OfPositive DefLattice}(6). The detail is omitted.

\textbf{Case (viii)} $D_{\underline{L}}\cong\underline{C}_{2^r}$. The proof is based on Proposition \ref{prop:GBC2r} and Lemma \ref{lemm:rankMod8OfPositive DefLattice}(7). The detail is omitted as well.
\end{proof}


\begin{proof}[Proof of Theorem \ref{thm:MilgramExtension2} (Theorem \ref{thm:MilgramExtension})]
By Theorem \ref{thm:indecomposableFQM}, there exist mutually orthogonal submodules $D_j$ of $D_{\underline{L}}$ such that $D_{\underline{L}}=\bigoplus_{j}D_j$ and that each $(D_j,Q\vert_{D_j})$ is an indecomposable module listed in Theorem \ref{thm:indecomposableFQM}. By \cite[Corollary 5.11]{Zhu24}, there exists a positive definite even integral lattice $\underline{L}_j$ such that $D_{\underline{L}_j}$ is isometrically isomorphic to $D_j$, for each $j$. Taking into account Proposition \ref{prop:GMmultiplicative} and Lemma \ref{lemm:MilgramExtension2Indecomp}, we have
\begin{equation}
\label{eq:MilgramExtensionTemp}
\mathscr{G}_{\underline{L}}(c,0)=\prod_j\mathscr{G}_{\underline{L}_j}(c,0)=\prod_j\legendre{\det{\underline{L}_j}}{c}_K\etp{\frac{n_jc_0}{8}}=\legendre{D}{c}_K\etp{\frac{\sum_jn_jc_0}{8}},
\end{equation}
where $\det{\underline{L}_j}=\abs{L_j^\sharp/L_j}$ and $n_j=\rank(L_j)$. Since the discriminant modules of $\underline{L}$ and of the orthogonal (external) direct sum of $\underline{L}_j$ are isometrically isomorphic, we have $n\equiv\sum_jn_j\bmod{8}$ according to Milgram's formula (cf. e.g. \cite[p. 127]{MH73}). Inserting this into \eqref{eq:MilgramExtensionTemp} we obtain \eqref{eq:MilgramExtension2}, as desired.
\end{proof}

\begin{rema}
\label{rema:MilgramExtension2indef}
Theorem \ref{thm:MilgramExtension2} can be adapted to include indefinite lattices. Let $\underline{L}=(L,B)$ be an even integral lattice of signature\footnote{The signature $\sigma$ of $\underline{L}=(L,B)$ is defined as follows. We choose a particular basis $(e_j)$ of $\numR\otimes L$ such that the Gram matrix $(B(e_j,e_j))_{i,j}$ is diagonal. Then $\sigma$ is the number of positive entries minus the number of negative entries. It is independent of the choice of  $(e_j)$ by the Jacobi-Sylvester Inertia Theorem. If $\underline{L}$ is positive definite, then $\sigma=\rank(L)$.} $\sigma\in\numZ$. Then \eqref{eq:MilgramExtension2} should be adapted to
\begin{equation}
\label{eq:GLc0indefinite}
\mathscr{G}_{\underline{L}}(c,0)=\legendre{\abs{L^\sharp/L}}{c}_K\cdot\etp{\frac{\sigma c_0}{8}}.
\end{equation}
This general formula follows directly from the special case \eqref{eq:MilgramExtension2}. We choose a positive definite even integral lattice $\underline{\Lambda}$ such that $D_{\underline{\Lambda}}\cong D_{\underline{L}}$ by \cite[Corollary 5.11]{Zhu24}. Then $\mathscr{G}_{\underline{L}}(c,0)=\mathscr{G}_{\underline{\Lambda}}(c,0)$ by the definition. Now we apply Theorem \ref{thm:MilgramExtension2} to $\underline{\Lambda}$ and then replace $\rank(\Lambda)$ by $\sigma$ since they are congruent to each other modulo $8$ due to Milgram's formula.
\end{rema}

\section{Main Theorem: Explicit formulas for quadratic Gauss sums over $\numZ^n/c\numZ^n$}
\label{sec:main_theorem_explicit_formulas_for_quadratic_gauss_sums_over_numz_n_c_numz_n_}

\subsection{Lemmas}
We shall prove Theorem \ref{thm:main} in this section. As we have seen at the beginning of \S\ref{subsec:proof_of_milgramextension}, we need to calculate $\mathfrak{G}_{\underline{L}}(b/d;0)$ and $\mathscr{G}_{D_{\underline{L}}}(-cd,y-dx)$. The latter was almost done in Theorem \ref{thm:MilgramExtension2} (Theorem \ref{thm:MilgramExtension}), while the former relies on the following lemma.
\begin{lemm}
\label{lemm:twoGaussSumRelation}
Let $\underline{L}=(L,B)$ be an even integral lattice of signature $\sigma$ (set $Q(v)=\frac{1}{2}B(v,v)$) and let $D_{\underline{L}}$ be its discriminant module. Let $c$ be a nonzero integer and $x\in L^\sharp$. If $x+L\in D_{\underline{L}}[c]^\bullet$, then
\begin{equation}
\label{eq:reciprocity}
\mathscr{G}_{\underline{L}}(c,x)\cdot\mathfrak{G}_{\underline{L}}(1/c;x)=\etp{\frac{\sigma\sgn{c}}{8}}\cdot\sqrt{\abs{D_{\underline{L}}[c]}}.
\end{equation}
In particular, if $cQ(z)\in\numZ$ for all $z+L\in D_{\underline{L}}[c]$, then \eqref{eq:reciprocity} holds for $x=0$.
\end{lemm}
\begin{proof}
We set $V=\numQ\otimes_\numZ L$ and extend $B$ to $V$ by bilinearity. Let $\Lambda$ be an arbitrary lattice in $(V,B)$ that is included in $L$. Then
\begin{equation*}
\Lambda\subseteq L\subseteq L^\sharp\subseteq\Lambda^\sharp.
\end{equation*}
In this sequence, each lattice is of finite index in the next lattice and $[L\colon\Lambda]=[\Lambda^\sharp\colon L^\sharp]$. Moreover, $\Lambda$ is even integral since $L$ is, and the signature of $\Lambda$ is $\sigma$ as well. Let us now compare $\mathscr{G}_{\underline{L}}(c,x)$ and $\mathscr{G}_{\underline{\Lambda}}(c,x)$ where $x\in L^\sharp$.

Let $\mathscr{R}_{\Lambda^\sharp/L}$ be a complete set of coset representatives of $\Lambda^\sharp/L$. Similarly, we fix some $\mathscr{R}_{L/\Lambda}$.
The following map
\begin{equation*}
\mathscr{R}_{\Lambda^\sharp/L}\times \mathscr{R}_{L/\Lambda}\rightarrow\Lambda^\sharp/\Lambda,\quad (y_1,y_2)\mapsto y_1+y_2+\Lambda
\end{equation*}
is a bijection. Applying the change of variables corresponding to this bijection, we find that
\begin{align*}
\sum_{y+\Lambda\in\Lambda^\sharp/\Lambda}\etp{cQ(y)+B(y,x)}&=\sum_{y_1\in\mathscr{R}_{\Lambda^\sharp/L}}\etp{cQ(y_1)+B(y_1,x)}\sum_{y_2\in\mathscr{R}_{L/\Lambda}}\etp{B(y_2,cy_1+x)}\\
&=\sum_{\twoscript{y_1\in\mathscr{R}_{\Lambda^\sharp/L}}{cy_1+x\in L^\sharp}}\etp{cQ(y_1)+B(y_1,x)}\cdot\abs{L/\Lambda},
\end{align*}
where we have used the orthogonality relation of the characters of $L/\Lambda$. In other words,
\begin{equation*}
\sum_{y+\Lambda\in\Lambda^\sharp/\Lambda}\etp{cQ(y)+B(y,x)}=\abs{L/\Lambda}\cdot\sum_{y+L\in(\Lambda^\sharp\cap c^{-1}L^\sharp)/L}\etp{cQ(y)+B(y,x)}.
\end{equation*}
In above we apply the change of variables corresponding to the bijection
\begin{equation*}
\mathscr{R}_{(\Lambda^\sharp\cap c^{-1}L^\sharp)/L^\sharp}\times \mathscr{R}_{L^\sharp/L}\rightarrow(\Lambda^\sharp\cap c^{-1}L^\sharp)/L,\quad (y_1,y_2)\mapsto y_1+y_2+L,
\end{equation*}
and thus obtain
\begin{multline*}
\sum_{y+\Lambda\in\Lambda^\sharp/\Lambda}\etp{cQ(y)+B(y,x)}=\abs{L/\Lambda}\\
\cdot\sum_{y_1\in\mathscr{R}_{(\Lambda^\sharp\cap c^{-1}L^\sharp)/L^\sharp}}\etp{cQ(y_1)+B(y_1,x)}\sum_{y_2\in\mathscr{R}_{L^\sharp/L}}\etp{cQ(y_2)+B(y_2,x+cy_1)}.
\end{multline*}
Namely,
\begin{equation*}
\mathscr{G}_{\underline{\Lambda}}(c,x)=\sqrt{\frac{\abs{D_{\underline{L}}[c]}}{\abs{D_{\underline{\Lambda}}[c]}}}\cdot\sum_{y\in\mathscr{R}_{(\Lambda^\sharp\cap c^{-1}L^\sharp)/L^\sharp}}\etp{cQ(y)+B(y,x)}\mathscr{G}_{\underline{L}}(c,x+cy).
\end{equation*}
In the right-hand side we can assume that $x+cy+L\in D_{\underline{L}}[c]^\bullet$ by Proposition \ref{prop:StrombergscrG}. It follows from the condition $x+L\in D_{\underline{L}}[c]^\bullet$ that
\begin{equation*}
x+cy+L\in D_{\underline{L}}[c]^\bullet\Longleftrightarrow cy+L\in D_{\underline{L}}[c]^*\Longleftrightarrow y\in L^\sharp+c^{-1}L.
\end{equation*}
We now choose $\Lambda$ such that $\Lambda\subseteq cL$, and hence $L^\sharp+c^{-1}L\subseteq c^{-1}L^\sharp=\Lambda^\sharp\cap c^{-1}L^\sharp$. Consequently,
\begin{equation}
\label{eq:reciprocityTemp}
\mathscr{G}_{\underline{\Lambda}}(c,x)=\sqrt{\frac{\abs{D_{\underline{L}}[c]}}{\abs{D_{\underline{\Lambda}}[c]}}}\cdot\sum_{y+L^\sharp\in (L^\sharp+c^{-1}L)/L^\sharp}\etp{cQ(y)+B(y,x)}\mathscr{G}_{\underline{L}}(c,x+cy),
\end{equation}
provided that $\Lambda\subseteq cL$.
Since $c^{-1}L/(c^{-1}L\cap L^\sharp)\rightarrow (L^\sharp+c^{-1}L)/L^\sharp$, $y+(c^{-1}L\cap L^\sharp)\mapsto y+L^\sharp$ is a group isomorphism, \eqref{eq:reciprocityTemp} can be simplified to
\begin{equation}
\label{eq:relationGLGLambda}
\mathscr{G}_{\underline{\Lambda}}(c,x)=\sqrt{\frac{\abs{D_{\underline{L}}[c]}}{\abs{D_{\underline{\Lambda}}[c]}}}\cdot\mathfrak{g}\cdot\mathscr{G}_{\underline{L}}(c,x),
\end{equation}
where
\begin{align*}
\mathfrak{g}&=\sum_{y+(c^{-1}L\cap L^\sharp)\in c^{-1}L/(c^{-1}L\cap L^\sharp)}\etp{cQ(y)+B(y,x)}\\
&=\sum_{y+(L\cap cL^\sharp)\in L/(L\cap cL^\sharp)}\etp{\frac{Q(y)+B(y,x)}{c}}.
\end{align*}
We now calculate $\mathfrak{g}$. Using the change of variables corresponding to the bijection
\begin{equation*}
\mathscr{R}_{L/(L\cap cL^\sharp)}\times \mathscr{R}_{(L\cap cL^\sharp)/cL}\rightarrow L/cL,\quad (y_1,y_2)\mapsto y_1+y_2+cL,
\end{equation*}
we find that, since $B(y_1,y_2)/c\in\numZ$,
\begin{multline*}
\sum_{y+cL\in L/cL}\etp{\frac{Q(y)+B(y,x)}{c}}\\
=\sum_{y_1\in\mathscr{R}_{L/(L\cap cL^\sharp)}}\etp{\frac{Q(y_1)+B(y_1,x)}{c}}\cdot\sum_{y_2\in\mathscr{R}_{(L\cap cL^\sharp)/cL}}\etp{\frac{Q(y_2)+B(y_2,x)}{c}}.
\end{multline*}
By the condition $x+L\in D_{\underline{L}}[c]^\bullet$, we have $\etp{\frac{Q(y_2)+B(y_2,x)}{c}}=1$ for all $y_2\in\mathscr{R}_{(L\cap cL^\sharp)/cL}$. Therefore,
\begin{align}
\mathfrak{g}=\sum_{y_1\in\mathscr{R}_{L/(L\cap cL^\sharp)}}\etp{\frac{Q(y_1)+B(y_1,x)}{c}}&=\abs{(L\cap cL^\sharp)/cL}^{-1}\sum_{y+cL\in L/cL}\etp{\frac{Q(y)+B(y,x)}{c}}\notag\\
&=\abs{(L\cap cL^\sharp)/cL}^{-1}\abs{c}^{n/2}\mathfrak{G}_{\underline{L}}(1/c;x)\etp{-\frac{Q(x)}{c}}\label{eq:mathfrakg},
\end{align}
where $n=\rank(L)=\dim(V)$. We have $\abs{D_{\underline{L}}[c]}=\abs{(c^{-1}L\cap L^\sharp)/L}=\abs{(L\cap cL^\sharp)/cL}$, and $\abs{D_{\underline{\Lambda}}[c]}=\abs{(c^{-1}\Lambda\cap \Lambda^\sharp)/\Lambda}=\abs{c}^{n}$ since $\Lambda\subseteq cL$. Inserting these two equations and \eqref{eq:mathfrakg} into \eqref{eq:relationGLGLambda} yields
\begin{equation*}
\mathscr{G}_{\underline{\Lambda}}(c,x)=\frac{1}{\sqrt{\abs{D_{\underline{L}}[c]}}}\etp{-\frac{Q(x)}{c}}\mathfrak{G}_{\underline{L}}(1/c;x)\mathscr{G}_{\underline{L}}(c,x),\quad\text{if } \Lambda\subseteq cL.
\end{equation*}
To deduce \eqref{eq:reciprocity} from the above formula, it remains to prove
\begin{equation}
\label{eq:GLambdaTemp}
\mathscr{G}_{\underline{\Lambda}}(c,x)=\etp{\frac{\sigma\sgn{c}}{8}}\etp{-\frac{Q(x)}{c}}
\end{equation}
for some\footnote{Indeed, if we can prove \eqref{eq:GLambdaTemp} for one lattice $\Lambda\subseteq cL$, then it will hold for all lattices $\Lambda\subseteq cL$.} lattice $\Lambda\subseteq cL$. To this end, we choose $n$ mutually orthogonal nonzero vectors $e_1,e_2,\dots,e_n\in cL$. Then they form a basis of $V$. Set $\Lambda=\bigoplus_j\numZ e_{j}$; it is a lattice in $(V,B)$ and $\Lambda\subseteq cL$. We now calculate \eqref{eq:GLambdaTemp} for this particular $\Lambda$. Set $d_j=B(e_j,e_j)$. Since $\Lambda\subseteq cL$ and $L$ is even integral, we have $2c^2\mid d_j$. Let $n_1$ and $n_2$ be the numbers of $j$ such that $d_j/c>0$ and $d_j/c<0$, respectively. Then $n_1-n_2=\sigma\sgn{c}$. Set $x=\sum_{j}r_je_j$, where $r_j\in\numQ$. We have $\abs{\Lambda^\sharp/\Lambda}=\prod_j\abs{d_j}$ and $\abs{\Lambda^\sharp/\Lambda[c]}=\abs{c}^n$. Therefore,
\begin{align}
\sqrt{\prod\nolimits_j\abs{d_j}}\cdot\sqrt{\abs{c}^n}\cdot\mathscr{G}_{\underline{\Lambda}}(c,x)&=\sum_{y\in\Lambda^\sharp/\Lambda}\etp{cQ(y)+B(y,x)}\notag\\
&=\prod_{j=1}^n\sum_{y_j\in d_j^{-1}\numZ/\numZ}\etp{cQ(y_je_j)+B(y_je_j,x)}\notag\\
&=\prod_{j=1}^n\sum_{y\in \numZ/d_j\numZ}\etp{\frac{c}{2d_j}y^2+r_jy}.\label{eq:GLambdaTemp2}
\end{align}
Since $e_j\in cL$ and $x\in L^\sharp$, we have $c^{-1}r_jd_j=B(c^{-1}e_j,x)\in\numZ$. Thus,
\begin{align*}
\sum_{y\in \numZ/d_j\numZ}\etp{\frac{c}{2d_j}y^2+r_jy}&=\sum_{y\in \numZ/d_j\numZ}\etp{\frac{c}{2d_j}\left(y+\frac{d_jr_j}{c}\right)^2}\etp{-\frac{d_jr_j^2}{2c}}\\
&=\frac{\abs{c}}{2}\sum_{y\bmod{\abs{2d_j/c}}}\etp{\frac{y^2}{2d_j/c}}\etp{-\frac{d_jr_j^2}{2c}}.
\end{align*}
Inserting the classical formula of Gauss (see \eqref{eq:classicalGauss}) into the above formula, we have, since $4\mid 2d_j/c$,
\begin{equation*}
\sum_{y\in \numZ/d_j\numZ}\etp{\frac{c}{2d_j}y^2+r_jy}=\sqrt{\abs{d_j}}\sqrt{\abs{c}}\etp{-\frac{d_jr_j^2}{2c}}\etp{\frac{\sgn{d_j/c}}{8}}.
\end{equation*}
Substituting this into \eqref{eq:GLambdaTemp2} gives \eqref{eq:GLambdaTemp}, as desired.

For the last statement, $cQ(z)\in\numZ$ for all $z+L\in D_{\underline{L}}[c]$ implies that $0+L\in D_{\underline{L}}[c]^\bullet$, and hence \eqref{eq:reciprocity} holds for $x=0$.
\end{proof}

\begin{rema}
\label{rema:reciprocity}
Let $N$ be the level of $\underline{L}$. If $\gcd(c,N)=1$, then $D_{\underline{L}}[c]=\{0+L\}$ and $D_{\underline{L}}[c]^\bullet=D_{\underline{L}}$. Therefore, \eqref{eq:reciprocity} holds for all $x\in L^\sharp$. On the other hand, if $N\mid c$, then $cQ(z)=c/N\cdot NQ(z)\in\numZ$ for all $z+L\in D_{\underline{L}}$, and hence \eqref{eq:reciprocity} holds for $x=0$.
\end{rema}

\begin{coro}
\label{coro:GLbd0Ndivc}
Let $\underline{L}=(L,B)$ be an even integral lattice of signature $\sigma$ and level $N$. Let $c$ be a nonzero integer divisible by $N$. Then
\begin{equation*}
\mathfrak{G}_{\underline{L}}(1/c;0)=\etp{\frac{\sigma\sgn{c}}{8}}\cdot\sqrt{\abs{L^\sharp/L}}.
\end{equation*}
\end{coro}
\begin{proof}
By the above remark, \eqref{eq:reciprocity} is applicable for $x=0$, which gives
\begin{equation}
\label{eq:GL1c0temp}
\mathfrak{G}_{\underline{L}}(1/c;0)=\etp{\frac{\sigma\sgn{c}}{8}}\cdot\sqrt{\abs{D_{\underline{L}}[c]}}\mathscr{G}_{\underline{L}}(c,0)^{-1}.
\end{equation}
Since $N\mid c$, we have $D_{\underline{L}}[c]=D_{\underline{L}}$, and $\mathscr{G}_{\underline{L}}(c,0)=1$. Inserting these into \eqref{eq:GL1c0temp} gives the desired conclusion.
\end{proof}

\begin{coro}
\label{coro:GLbd0coprimetoN}
Let $\underline{L}=(L,B)$ be an even integral lattice of signature $\sigma$ and level $N$. Set $D=\abs{L^\sharp/L}$. Let $b$, $d$ be integers with $d\neq 0$, $\gcd(b,d)=1$ and $\gcd(N,d)=1$. Then
\begin{equation}
\label{eq:GLbd0coprimetoN}
\mathfrak{G}_{\underline{L}}(b/d;0)=\etp{\frac{\sigma\sgn{b}(\sgn{d_0}-d_0)}{8}}\cdot\legendre{\abs{b}^n D}{d}_K,
\end{equation}
where $d_0$ is the odd part of $d$ (with $\sgn{d_0}=\sgn{d}$), $n=\rank(L)$, and where we set $\sgn{0}=0$ as a convention.
\end{coro}
\begin{proof}
If $b=0$, then $d=\pm1$, and the conclusion holds trivially. Now assume that $b\neq0$. According to Lemma \ref{lemm:twoGaussSumRelation}, we have
\begin{align}
\mathfrak{G}_{\underline{L}}(b/d;0)&=\mathfrak{G}_{(L, \abs{b}\cdot B)}(\sgn{b}/d;0)\notag\\
&=\etp{\frac{\sigma\sgn{bd}}{8}}\sqrt{\abs{D_{(L,\abs{b}\cdot B)}[\sgn{b}d]}}\mathscr{G}_{(L,\abs{b}\cdot B)}(\sgn{b}d,0)^{-1},\label{eq:GLbd0temp}
\end{align}
where we have used the fact that the signatures of $\underline{L}$ and of $(L, \abs{b}B)$, respectively, are the same. Note that the level of $(L, \abs{b}B)$ is equal to $\abs{b}N$. By the conditions we have $\gcd(\abs{b}N,\sgn{b}d)=1$, and hence $\sqrt{\abs{D_{(L,\abs{b}\cdot B)}[\sgn{b}d]}}=1$. Inserting this and \eqref{eq:GLc0indefinite} into \eqref{eq:GLbd0temp} gives
\begin{equation*}
\mathfrak{G}_{\underline{L}}(b/d;0)=\etp{\frac{\sigma\sgn{bd}}{8}}\legendre{\abs{(b^{-1}L^\sharp)/L}}{\sgn{b}d}_K\etp{-\frac{\sigma\sgn{b}d_0}{8}}.
\end{equation*}
The above formula can be simplified to \eqref{eq:GLbd0coprimetoN} since $\abs{(b^{-1}L^\sharp)/L}=\abs{b}^nD$ and since $\legendre{x}{-y}_K=\legendre{x}{y}_K$ when $x>0$.
\end{proof}

\begin{lemm}
Let the notations and assumptions be as in Lemma \ref{lemm:GaussSumWeilRepr}.  Suppose that $\gcd(N,a)=1$. If $N\mid b$ and $d\neq0$, then we have
\begin{multline}
\label{eq:GacwxTemp}
\mathfrak{G}_{G}(a/c;\mathbf{w},\mathbf{x})=\rmi^{\frac{n\sgn{d}}{2}}c^{\frac{n}{2}}\etp{-\frac{a'a^2}{2c}\mathbf{x}^\tp G^{-1}\mathbf{x}}\\
\cdot\mathfrak{G}_{\underline{L}}(b/d;0)\sqrt{\abs{(G^{-1}\numZ^n\cap c^{-1}\numZ^n)/\numZ^n}}\mathscr{G}_{D_{\underline{L}}}(-ac,a(\mathbf{w}+G^{-1}\mathbf{x})),
\end{multline}
where $a'\in\numZ$ is any integer with $aa'\equiv1\bmod{cN}$.
\end{lemm}
\begin{proof}
Since $N\mid b$, we have $bc\cdot L^\sharp\subseteq L$, and hence Proposition \ref{prop:rhoA1yx2} is applicable. Substituting \eqref{eq:rhoA1yx2} (with $y=\mathbf{w}$ and $x=-aG^{-1}\mathbf{x}$) into \eqref{eq:GaussSumWeilRepr} gives
\begin{multline}
\label{eq:GacwxTemp2}
\mathfrak{G}_{G}(a/c;\mathbf{w},\mathbf{x})=\etp{-\frac{a^2d}{2c}\mathbf{x}^\tp G^{-1}\mathbf{x}}(\rmi c)^{\frac{n}{2}}\sqrt{\det(G)}\\
\cdot(-\rmi)^{\frac{n(1-\sgn{d})}{2}}\cdot\mathfrak{G}_{\underline{L}}(b/d;\mathbf{w})\sqrt{\frac{\abs{D_{\underline{L}}[c]}}{\abs{D_{\underline{L}}}}}\mathscr{G}_{D_{\underline{L}}}(-cd,\mathbf{w}+adG^{-1}\mathbf{x}).
\end{multline}
We have $D_{\underline{L}}=G^{-1}\numZ^n/\numZ^n$, and hence
\begin{equation*}
\abs{D_{\underline{L}}}=\det(G),\quad \abs{D_{\underline{L}}[c]}=\abs{(G^{-1}\numZ^n\cap c^{-1}\numZ^n)/\numZ^n}.
\end{equation*}
Since $bc\cdot L^\sharp\subseteq L$ and $\mathbf{w}\in G^{-1}\numZ^n=L^\sharp$, we have $bc\mathbf{w}\in L$, and hence Proposition \ref{prop:GwtoG0} is valid, that is, \eqref{eq:GwtoG0} holds. However $N\mid b$, so we can simplify \eqref{eq:GwtoG0} to
\begin{equation*}
\mathfrak{G}_{\underline{L}}(b/d;w)=\etp{a^2bdQ(w)}\mathfrak{G}_{\underline{L}}(b/d;0)=\mathfrak{G}_{\underline{L}}(b/d;0).
\end{equation*}
Furthermore, we have
\begin{equation*}
\etp{-\frac{a^2d}{2c}\mathbf{x}^\tp G^{-1}\mathbf{x}}=\etp{-\frac{a'ad}{2cN}(a\mathbf{x})^\tp NG^{-1}(a\mathbf{x})}=\etp{-\frac{a'a^2}{2c}\mathbf{x}^\tp G^{-1}\mathbf{x}}
\end{equation*}
for $ad=1+bc$ and $\frac{ba^2}{2}\mathbf{x}^\tp G^{-1}\mathbf{x}\in\numZ^n$. Now we consider $\mathscr{G}_{D_{\underline{L}}}(-cd,\mathbf{w}+adG^{-1}\mathbf{x})$. Applying the change of variables $D_{\underline{L}}\rightarrow D_{\underline{L}}$, $y\mapsto ay$, which is bijective since $N\mid b$, we find that
\begin{equation*}
\mathscr{G}_{D_{\underline{L}}}(-cd,\mathbf{w}+adG^{-1}\mathbf{x})=\mathscr{G}_{D_{\underline{L}}}(-a^2cd,a\mathbf{w}+a^2dG^{-1}\mathbf{x})=\mathscr{G}_{D_{\underline{L}}}(-ac,a(\mathbf{w}+G^{-1}\mathbf{x})),
\end{equation*}
where we have used the fact $D_{\underline{L}}[-cd]=D_{\underline{L}}[-a^2cd]=D_{\underline{L}}[-ac]$. Taking into account these facts, we can reduce \eqref{eq:GacwxTemp2} to \eqref{eq:GacwxTemp}, as desired.
\end{proof}

\subsection{Proof of Theorem \ref{thm:main}}
When $a=0$, then $c=1$ and Theorem \ref{thm:main} holds trivially. Assume that $a\neq0$ below. First we consider the subcase $\gcd(N,a)=1$.

We begin with \eqref{eq:GacwxTemp}, which is true for under the conditions of Theorem \ref{thm:main}, $\mathfrak{G}_{G}(a/c;\mathbf{w},\mathbf{x})$ is indeed integral-parametric, and we can indeed require $b$ and $d$ fulfilling the further requirement $N\mid b$ and $d\neq0$, for $\gcd(N,a)=1$. Since $N\mid b$, we have $\gcd(N,d)=1$, and hence \eqref{eq:GLbd0coprimetoN} holds with $\sigma=n$. It remains to calculate $\mathscr{G}_{D_{\underline{L}}}(-ac,a(\mathbf{w}+G^{-1}\mathbf{x}))$ in \eqref{eq:GacwxTemp}.

If $\gcd(N,c)=1$, then \eqref{eq:reciprocity}, with $c$ replaced by $-ac$, holds for all $x\in L^\sharp$ by Remark \ref{rema:reciprocity}, and hence for $x=a(\mathbf{w}+G^{-1}\mathbf{x})$. It follows that
\begin{equation}
\label{eq:proofMainTheoremTemp1}
\mathscr{G}_{D_{\underline{L}}}(-ac,a(\mathbf{w}+G^{-1}\mathbf{x}))=\etp{-\frac{n\sgn{a}}{8}}\mathfrak{G}_{\underline{L}}\left(-\frac{1}{ac};a(\mathbf{w}+G^{-1}\mathbf{x})\right)^{-1}.
\end{equation}
Let $a_1,b_1\in\numZ$ satisfy $\det\tbtmat{a_1}{1}{Nb_1}{-ac}=1$. Then $Nb_1a(\mathbf{w}+G^{-1}\mathbf{x})\in\numZ^n$ and hence, by Proposition \ref{prop:GwtoG0},
\begin{align}
\mathfrak{G}_{\underline{L}}\left(-\frac{1}{ac};a(\mathbf{w}+G^{-1}\mathbf{x})\right)&=\etp{-\frac{a_1^2a^3c}{2}(\mathbf{w}+G^{-1}\mathbf{x})^\tp G(\mathbf{w}+G^{-1}\mathbf{x})}\mathfrak{G}_{\underline{L}}\left(-\frac{1}{ac};0\right)\notag\\
&=\etp{-\frac{a'c'a^2}{2}(\mathbf{w}+G^{-1}\mathbf{x})^\tp G(\mathbf{w}+G^{-1}\mathbf{x})}\mathfrak{G}_{\underline{L}}\left(-\frac{1}{ac};0\right),\label{eq:proofMainTheoremTemp2}
\end{align}
where $c'\in\numZ$ satisfy $cc'\equiv1\bmod{aN}$. Recall that $aa'\equiv1\bmod{cN}$ as required in \eqref{eq:GacwxTemp}. Again by Corollary \ref{coro:GLbd0coprimetoN} we have
\begin{equation}
\label{eq:proofMainTheoremTemp3}
\mathfrak{G}_{\underline{L}}\left(-\frac{1}{ac};0\right)=\etp{\frac{n(\sgn{-a_0c_0}+a_0c_0)}{8}}\cdot\legendre{D}{-ac}_K,
\end{equation}
where $a_0$ and $c_0$ are the odd part of $a$ and $c$, respectively. Inserting \eqref{eq:proofMainTheoremTemp3} into \eqref{eq:proofMainTheoremTemp2}, and then inserting the obtained formula into \eqref{eq:proofMainTheoremTemp1}, we find that
\begin{multline*}
\mathscr{G}_{D_{\underline{L}}}(-ac,a(\mathbf{w}+G^{-1}\mathbf{x}))\\
=\etp{-\frac{n\sgn{a}}{8}}\etp{\frac{a'c'a^2}{2}(\mathbf{w}+G^{-1}\mathbf{x})^\tp G(\mathbf{w}+G^{-1}\mathbf{x})}\etp{-\frac{n(\sgn{-a_0c_0}+a_0c_0)}{8}}\cdot\legendre{D}{-ac}_K.
\end{multline*}
Inserting this and \eqref{eq:GLbd0coprimetoN} into \eqref{eq:GacwxTemp}, and noting that $\abs{(G^{-1}\numZ^n\cap c^{-1}\numZ^n)/\numZ^n}=1$ since $\gcd(N,c)=1$, we deduce that, in the case where $\gcd(N,a)=\gcd(N,c)=1$,
\begin{multline*}
\mathfrak{G}_{G}(a/c;\mathbf{w},\mathbf{x})=\rmi^{\frac{n\sgn{d}}{2}}c^{\frac{n}{2}}\etp{-\frac{a'a^2}{2c}\mathbf{x}^\tp G^{-1}\mathbf{x}}\etp{\frac{n\sgn{b}(\sgn{d_0}-d_0)}{8}}\cdot\legendre{\abs{b}^n D}{d}_K\\
\cdot\etp{-\frac{n\sgn{a}}{8}}\etp{\frac{a'c'a^2}{2}(\mathbf{w}+G^{-1}\mathbf{x})^\tp G(\mathbf{w}+G^{-1}\mathbf{x})}\etp{-\frac{n(\sgn{-a_0c_0}+a_0c_0)}{8}}\cdot\legendre{D}{-ac}_K.
\end{multline*}
When choosing $b$ and $d$ in \eqref{eq:GacwxTemp}, we can further require that $d>0$, $2\nmid d$ and $d$ is sufficiently large, so that $\sgn{a}=\sgn{bc}=\sgn{b}$ and $d_0=d$. Then the above formula can be reduced to
\begin{multline}
\label{eq:GGacwxTemp}
\mathfrak{G}_{G}(a/c;\mathbf{w},\mathbf{x})=c^{\frac{n}{2}}\etp{-\frac{a'a^2}{2c}\mathbf{x}^\tp G^{-1}\mathbf{x}+\frac{a'c'a^2}{2}(\mathbf{w}+G^{-1}\mathbf{x})^\tp G(\mathbf{w}+G^{-1}\mathbf{x})}\\
\cdot\etp{\frac{n(1-a_0c_0)}{8}}\legendre{D}{-ac}_K\etp{\frac{n\sgn{a}(1-d)}{8}}\legendre{\abs{b}^n D}{d}_K.
\end{multline}
Setting $\mathbf{w}=\mathbf{x}=0$ in the above we deduce that
\begin{equation}
\label{eq:GGac00Temp1}
\mathfrak{G}_{G}(a/c;0,0)=c^{\frac{n}{2}}\etp{\frac{n(1-a_0c_0)}{8}}\legendre{D}{-ac}_K\etp{\frac{n\sgn{a}(1-d)}{8}}\legendre{\abs{b}^n D}{d}_K.
\end{equation}
On the other hand, we have
\begin{equation}
\label{eq:GGac00Temp2}
\mathfrak{G}_{G}(a/c;0,0)=c^{\frac{n}{2}}\mathfrak{G}_{\underline{L}}(a/c;0)=c^{\frac{n}{2}}\etp{\frac{n\sgn{a}(1-c_0)}{8}}\cdot\legendre{\abs{a}^n D}{c}_K
\end{equation}
by Corollary \ref{coro:GLbd0coprimetoN}. Thus the right-hand sides of \eqref{eq:GGac00Temp1} and \eqref{eq:GGac00Temp2} are equal. Inserting this equality into \eqref{eq:GGacwxTemp}, we arrive at \eqref{eq:GGacwxWhenNccoprime}, as desired.

This concludes the proof in the case $\gcd(N,c)=\gcd(N,a)=1$. Now assume $\gcd(N,c)=1$ but $a$ is general. According to Proposition \ref{prop:GtoaG}, we have $\mathfrak{G}_{G}(a/c;\mathbf{w},\mathbf{x})=\mathfrak{G}_{\abs{a}G}(\sgn{a}/c;\mathbf{w},\abs{a}\mathbf{x})$. Since $\gcd(N,\sgn{a})=1$, we can apply what we have proved to $\mathfrak{G}_{\abs{a}G}(\sgn{a}/c;\mathbf{w},\abs{a}\mathbf{x})$. It should be noted that the level of $\abs{a}G$ is $\abs{a}N$ and $\det(\abs{a}G)=\abs{a}^nD$. It follows that
\begin{align*}
\mathfrak{G}_{G}(a/c;\mathbf{w},\mathbf{x})&=\mathfrak{G}_{\abs{a}G}(\sgn{a}/c;\mathbf{w},\abs{a}\mathbf{x})\\
&=c^{\frac{n}{2}}\etp{-\frac{\sgn{a}}{2c}(\abs{a}\mathbf{x})^\tp (\abs{a}G)^{-1}(\abs{a}\mathbf{x})+\frac{\sgn{a}c''}{2}(\mathbf{w}+G^{-1}\mathbf{x})^\tp \abs{a}G (\mathbf{w}+G^{-1}\mathbf{x})}\\
\cdot&\etp{\frac{n\sgn{a}(1-c_0)}{8}}\cdot\legendre{\abs{a}^n D}{c}_K,
\end{align*}
where $c''\in\numZ$ satisfies $cc''\equiv1\bmod{\sgn{a}\cdot\abs{a}N}$. This formula can be reduced to \eqref{eq:GGacwxWhenNccoprime} immediately.

Finally, we shall treat the case $\gcd(N,c)=N$, in which we must have $\gcd(N,a)=1$ so that \eqref{eq:GacwxTemp} still holds. Since $N\mid c$, we have $\abs{(G^{-1}\numZ^n\cap c^{-1}\numZ^n)/\numZ^n}=\abs{G^{-1}\numZ^n/\numZ^n}=D$. Therefore,
\begin{equation}
\label{eq:GacwxTempNmidc}
\mathfrak{G}_{G}(a/c;\mathbf{w},\mathbf{x})=\rmi^{\frac{n\sgn{d}}{2}}c^{\frac{n}{2}}\etp{-\frac{a'a^2}{2c}\mathbf{x}^\tp G^{-1}\mathbf{x}}\sqrt{D}\mathfrak{G}_{\underline{L}}(b/d;0)\mathscr{G}_{D_{\underline{L}}}(-ac,a(\mathbf{w}+G^{-1}\mathbf{x})).
\end{equation}
It follows from the orthogonality relation that $\mathscr{G}_{D_{\underline{L}}}(-ac,a(\mathbf{w}+G^{-1}\mathbf{x}))=\delta$ where $\delta$ is the one defined in Theorem \ref{thm:main}(2). It remains to express $\mathfrak{G}_{\underline{L}}(b/d;0)$ in terms of $a$ and $c$ in \eqref{eq:GacwxTempNmidc}. Recall that when choosing $b$ and $d$, we require $ad-bc=1$, $N\mid b$, and $d\neq0$. We now split the proof into four cases:
\begin{equation*}
(i)\,2\mid D\qquad(ii)\,2\nmid D\text{ and }2\mid c\qquad(iii)\,2\nmid D,2\nmid c\text{ and }2\nmid a\qquad(iv)\,2\nmid D,2\nmid c\text{ and }2\mid a.
\end{equation*}

For case (i), we further require that $b$ and $d$ satisfy $4D\mid b$ (which is possible for $N$ and $4D$ have the same set of prime factors), $d>0$ and $\sgn{a}=\sgn{b}$. Then we have $\legendre{\abs{b}}{d}_K=\legendre{\abs{b}}{a}_K$ for $ad\equiv1\bmod{\abs{b}}$ and $4\mid b$, and we have $\legendre{D}{d}_K=\legendre{D}{a}_K$ for $ad\equiv1\bmod{4D}$. It follows from these facts and \eqref{eq:GLbd0coprimetoN} that
\begin{equation}
\label{eq:GLbd0usingacTemp}
\mathfrak{G}_{\underline{L}}(b/d;0)=\etp{\frac{n\sgn{a}(1-a)}{8}}\legendre{\abs{b}}{a}_K^n\legendre{D}{a}_K
\end{equation}
It can be proved that $\legendre{\abs{b}}{a}_K=\legendre{-\sgn{a}c}{\abs{a}}_K$ using the facts $-bc\equiv1\bmod{a}$, $\sgn{b}=\sgn{a}$ and $2\nmid a$. Inserting this into \eqref{eq:GLbd0usingacTemp} we find that $\mathfrak{G}_{\underline{L}}(b/d;0)=\mu_G(a,c)$, where $\mu_G(a,c)$ is the one defined in Theorem \ref{thm:main}(2). Combining this and \eqref{eq:GacwxTempNmidc} we arrive at \eqref{eq:GGacwxWhenNdivc}, as desired.

For case (ii), we must have $2\mid n$ (cf., e.g. \cite[Remarks 14.3.23(b)]{CS17}). The deduction used in case (i) is also valid for this case.

For case (iii), we must have $2\mid n$ as in case (ii). The deduction is similar to case (ii). The different point is that in this case, we shall use the fact $2\mid n$ when proving
\begin{equation*}
\etp{\frac{n\sgn{b}(\sgn{d_0}-d_0)}{8}}=\etp{\frac{n\sgn{a}(1-a)}{8}}.
\end{equation*}
(If one chooses $b$ with $8D\mid b$ instead of $4D\mid b$, which is possible, then one does not need the fact $2\mid n$ to prove the above equality.)

For case (iv), we also have $2\mid n$ as above. We further require that $b$ and $d$ satisfy
\begin{equation*}
D\mid b,\quad d>0,\quad d\equiv1\bmod{8},\quad \sgn{b}=\sgn{a}.
\end{equation*}
Such $b$ and $d$ always exist by modulo $8$ arithmetic and by the fact $D$ and $N$ have the same set of prime factors. Since $2\mid a$ and $ad-bc=1$, we have $b\equiv c\equiv1\bmod{2}$. Therefore, \eqref{eq:GLbd0coprimetoN} implies
\begin{equation*}
\mathfrak{G}_{\underline{L}}(b/d;0)=\etp{\frac{n\sgn{b}(\sgn{d_0}-d_0)}{8}}\cdot\legendre{\abs{b}^n D}{d}_K=\legendre{D}{d}_K=\legendre{a}{D}_K.
\end{equation*}
Combining this and \eqref{eq:GacwxTempNmidc} we arrive at \eqref{eq:GGacwxWhenNdivc}, as desired. This concludes the whole proof.

\subsection{When $G$ is indefinite}
\label{subsec:when_g_is_indefinite}
If we replace the assumption ``$G$ is positive definite'' by ``$G$ is nonsingular'' in Theorem \ref{thm:main}, then there are similar conclusions. For instance, Corollaries \ref{coro:GLbd0Ndivc} and \ref{coro:GLbd0coprimetoN} give formulas for $\mathfrak{G}_{G}(r;0,0)$. These formulas involve the signature of the corresponding lattice and involve $\abs{L^\sharp/L}=\abs{\det(G)}$. Attention must be paid to that in Theorem \ref{thm:main} we define $D=\det(G)$, but in the above propositions where $\underline{L}$ or $G$ are allowed to be non-positive-definite, $D$ usually means $\abs{L^\sharp/L}=\abs{\det(G)}$.

The following extension of the $\gcd(N,c)=1$ case of Theorem \ref{thm:main} to indefinite $G$ is interesting, for the formula does not involve the signature. This formula is more efficient than those involving the signature, e.g., Corollary \ref{coro:GLbd0coprimetoN}, as to calculate the signature (indeed, only signature mod $8$ is needed), one need to first calculate all eigenvalues of $G$ or calculate the Jordan decomposition of $G^{-1}\numZ^n/\numZ^n$, both of which, when $n$ is large, are more arduous tasks than just calculating $\det(G)$ and then inserting it into the right-hand side of \eqref{eq:GGacwxWhenNccoprime}.
\begin{thm}
\label{thm:mainIndefinite}
Let the notations and assumptions be as in Theorem \ref{thm:main} except that $G$ is not required to be positive definite but required merely to be nonsingular. If $\gcd(N,c)=1$, then \eqref{eq:GGacwxWhenNccoprime} still holds. (We emphasize that $D=\det(G)$ may be negative now, and that $\gcd(N,c)=1\Longleftrightarrow\gcd(D,c)=1$.)
\end{thm}
\begin{proof}
Let $t$ be a positive integer and set $G_p=G+4N^2ct\cdot I_n$, where $I_n$ is the identity matrix. Then $G_p$ is even integral and symmetric. Let $t$ be sufficiently large, then $G_p$ is positive definite. Let $N_p$ be the level of $G_p$ and let $D_p=\det(G_p)$. Then
\begin{equation}
\label{eq:DpDmod4N2c}
D_p\equiv D\bmod{4N^2c},
\end{equation}
and hence
\begin{equation*}
\gcd(N,c)=1\Longleftrightarrow\gcd(D,c)=1\Longleftrightarrow\gcd(D_p,c)=1\Longleftrightarrow\gcd(N_p,c)=1.
\end{equation*}
Since $G\cdot\mathbf{w}\in\numZ^n$, we have $G_p\cdot\mathbf{w}\in\numZ^n$, so that $\mathfrak{G}_{G_p}(a/c;\mathbf{w},\mathbf{x})$ is integral-parametric. By the definition, we have
\begin{align*}
&\mathfrak{G}_{G_p}(a/c;\mathbf{w},\mathbf{x})\\
=&\sum_{\mathbf{v}\in\numZ^n/c\numZ^n}\etp{\frac{a}{c}\cdot\left(\frac{1}{2}(\mathbf{v}+\mathbf{w})^\tp\cdot G\cdot (\mathbf{v}+\mathbf{w})+(\mathbf{v}+\mathbf{w})^\tp\cdot\mathbf{x}\right)+2at(N\mathbf{v}+N\mathbf{w})^\tp(N\mathbf{v}+N\mathbf{w})}\\
=&\mathfrak{G}_{G}(a/c;\mathbf{w},\mathbf{x})
\end{align*}
since $2at(N\mathbf{v}+N\mathbf{w})^\tp(N\mathbf{v}+N\mathbf{w})\in\numZ$. Taking into account this and applying Theorem \ref{thm:main}(1) to $\mathfrak{G}_{G_p}(a/c;\mathbf{w},\mathbf{x})$, we obtain
\begin{multline*}
\mathfrak{G}_{G}(a/c;\mathbf{w},\mathbf{x})=c^{n/2}\cdot\legendre{\abs{a}}{c}_K^n\cdot\legendre{D_p}{c}_K\cdot\etp{\frac{n\cdot\sgn a\cdot(1-c_0)}{8}}\\
\times\etp{-\frac{a}{2c}\mathbf{x}^\tp G_p^{-1}\mathbf{x}+\frac{ac''}{2}(\mathbf{w}+G_p^{-1}\mathbf{x})^\tp G_p (\mathbf{w}+G_p^{-1}\mathbf{x})},
\end{multline*}
where $c''$ is any integer satisfying $cc''\equiv1\bmod{aN_p}$. To deduce \eqref{eq:GGacwxWhenNccoprime} from the above formula, it remains to prove
\begin{equation}
\label{eq:fromIndefiniteToDefinite1}
\legendre{D_p}{c}_K=\legendre{D}{c}_K,
\end{equation}
and
\begin{multline}
\label{eq:fromIndefiniteToDefinite2}
\etp{-\frac{a}{2c}\mathbf{x}^\tp G_p^{-1}\mathbf{x}+\frac{ac''}{2}(\mathbf{w}+G_p^{-1}\mathbf{x})^\tp G_p (\mathbf{w}+G_p^{-1}\mathbf{x})}\\
=\etp{-\frac{a}{2c}\mathbf{x}^\tp G^{-1}\mathbf{x}+\frac{ac'}{2}(\mathbf{w}+G^{-1}\mathbf{x})^\tp G (\mathbf{w}+G^{-1}\mathbf{x})},
\end{multline}
where $c'$ is any integer satisfying $cc'\equiv1\bmod{aN}$.

If $c$ is odd (it is already positive), then the function $x\mapsto\legendre{x}{c}_K$ is $c$-periodic, and hence \eqref{eq:fromIndefiniteToDefinite1} holds due to \eqref{eq:DpDmod4N2c}. Otherwise, if $c$ is even, then we can write $c=2^ec_0$ where $e\in\numgeq{Z}{1}$ and $c_0$ is odd. By \eqref{eq:DpDmod4N2c} we have $D_p\equiv D\bmod{8}$, and hence $\legendre{D_p}{2}_K=\legendre{D}{2}_K$. It follows that
\begin{equation*}
\legendre{D_p}{c}_K=\legendre{D_p}{2}_K^e\legendre{D_p}{c_0}_K=\legendre{D}{2}_K^e\legendre{D}{c_0}_K=\legendre{D}{c}_K.
\end{equation*}
This completes the proof of \eqref{eq:fromIndefiniteToDefinite1}.

Now we treat \eqref{eq:fromIndefiniteToDefinite2}. We choose a $c'''$ satisfying $cc'''\equiv1\bmod{aNN_p}$, which exists since $\gcd(aNN_p,c)=1$. Then the left-hand and right-hand sides of \eqref{eq:fromIndefiniteToDefinite2} remain unchanged if we set $c'=c'''$ and $c''=c'''$, respectively. After setting this value, we find \eqref{eq:fromIndefiniteToDefinite2} is equivalent to
\begin{multline}
\label{eq:fromIndefiniteToDefinite3}
\etp{\frac{ac'''}{2}\mathbf{w}^\tp G_p\mathbf{w}+ac'''\mathbf{w}^\tp\mathbf{x}+\frac{a(cc'''-1)}{2c}\mathbf{x}^\tp G_p^{-1}\mathbf{x}}\\
=\etp{\frac{ac'''}{2}\mathbf{w}^\tp G\mathbf{w}+ac'''\mathbf{w}^\tp\mathbf{x}+\frac{a(cc'''-1)}{2c}\mathbf{x}^\tp G^{-1}\mathbf{x}}.
\end{multline}
Since $G_p-G=4N^2ctI_n$, we have
\begin{equation}
\label{eq:fromIndefiniteToDefinite3_1}
\frac{ac'''}{2}\mathbf{w}^\tp G_p\mathbf{w}-\frac{ac'''}{2}\mathbf{w}^\tp G\mathbf{w}=2acc'''t(N\mathbf{w})^\tp(N\mathbf{w})\in\numZ.
\end{equation}
Note that $G_p^{-1}-G^{-1}=-4N^2ctG_p^{-1}G^{-1}$. Therefore, if we let $cc'''=1+aNN_pu$ with $u\in\numZ$, then
\begin{equation}
\label{eq:fromIndefiniteToDefinite3_2}
\frac{a(cc'''-1)}{2c}\mathbf{x}^\tp G_p^{-1}\mathbf{x}-\frac{a(cc'''-1)}{2c}\mathbf{x}^\tp G^{-1}\mathbf{x}=-2N^2ut(a\mathbf{x})^\tp(N_pG_p^{-1})(NG^{-1})(a\mathbf{x})\in\numZ.
\end{equation}
Combining \eqref{eq:fromIndefiniteToDefinite3_1} and \eqref{eq:fromIndefiniteToDefinite3_2}, we deduce \eqref{eq:fromIndefiniteToDefinite3}, and hence \eqref{eq:fromIndefiniteToDefinite2}. This concludes the proof.
\end{proof}

\section{Application I. A duality theorem and Gauss subsums}
\label{sec:application_i_a_duality_theorem_and_gauss_subsums}
In this section, we give our first application of Theorem \ref{thm:main} and Theorem \ref{thm:mainIndefinite}, namely, some duality formulas on Gauss subsums. The results are presented in two theorems, according to $\gcd(N,c)=1$ and $N$ respectively.

The conclusion for the case $\gcd(N,c)=1$ has been presented in Theorem \ref{thm:duality1}, whose proof will be given below in \S\ref{subsec:proofDuality1}. The conclusion for the case $\gcd(N,c)=N$, together with its proof, will be given in \S\ref{subsec:Duality2}. A lemma which is valid for all $c$ is provided in \S\ref{subsec:lemma_general_c}. It is required for both cases just mentioned.

\subsection{A lemma for general $c$}
\label{subsec:lemma_general_c}
\begin{lemm}
\label{lemm:GGUforGeneralc}
Let $G$ be a nonsingular even integral symmetric matrix of size $n\in\numgeq{Z}{1}$. Let $c$ be a positive integer and $a$ be a nonzero integer coprime to $c$. Let $H$ be a subgroup of $\numZ^n/c\numZ^n$ and let $\mathbf{w}\in G^{-1}\numZ^n$. Then we have
\begin{equation}
\label{eq:GGUforGeneralc}
\frac{1}{\sqrt{\abs{H}}}\cdot\mathfrak{G}_{G}^H(a/c;\mathbf{w})=\frac{1}{\sqrt{c^n}\sqrt{\abs{H^\bot}}}\cdot\sum_{\mathbf{y}\in H^\bot}\etp{\frac{\mathbf{w}^\tp\mathbf{y}}{c}}\mathfrak{G}_{G}(a/c;\mathbf{w},-a^{-1}\mathbf{y}).
\end{equation}
\end{lemm}
Before the proof, we remark three important facts. Firstly, $\mathfrak{G}_{G}^H(a/c;\mathbf{w})$ showing up in the left-hand side of \eqref{eq:GGUforGeneralc} is well-defined by Corollary \ref{coro:welldefinedGaussH}, since here $\mathfrak{G}_{G}(a/c;\mathbf{w})$ is integral-parametric. Secondly, each $\etp{\frac{\mathbf{w}^\tp\mathbf{y}}{c}}\mathfrak{G}_{G}(a/c;\mathbf{w},-a^{-1}\mathbf{y})$ showing up in the right-hand side of \eqref{eq:GGUforGeneralc} actually means $\etp{\frac{\mathbf{w}^\tp\widetilde{\mathbf{y}}}{c}}\mathfrak{G}_{G}(a/c;\mathbf{w},-a^{-1}\widetilde{\mathbf{y}})$ where $\widetilde{\mathbf{y}}$ is any element in the preimage of $\mathbf{y}$ under the projection $\numZ^n\rightarrow\numZ^n/c\numZ^n$. One can show immediately that this quantity is independent of the choice of $\widetilde{\mathbf{y}}$. However, the values of both factors $\etp{\frac{\mathbf{w}^\tp\widetilde{\mathbf{y}}}{c}}$ and $\mathfrak{G}_{G}(a/c;\mathbf{w},-a^{-1}\widetilde{\mathbf{y}})$ themselves may change if we choose different $\widetilde{\mathbf{y}}$. Finally, the sum $\mathfrak{G}_{G}(a/c;\mathbf{w},-a^{-1}\mathbf{y})$ is of-course integral-parametric.
\begin{proof}[Proof of Lemma \ref{lemm:GGUforGeneralc}]
Let $f$ be the function on $\numZ^n/c\numZ^n$ defined by $f(\mathbf{v})=\etp{\frac{a}{2c}(\mathbf{v}+\mathbf{w})^\tp G(\mathbf{v}+\mathbf{w})}$. For each $\mathbf{t}\in\numZ^n/c\numZ^n$, let $\chi_{\mathbf{t}}$ be the character of $\numZ^n/c\numZ^n$ defined by $\chi_{\mathbf{t}}(\mathbf{v})=\etp{\frac{1}{c}\mathbf{v}^\tp\mathbf{t}}$. It is immediate that the map from $\numZ^n/c\numZ^n$ to $\widehat{\numZ^n/c\numZ^n}$ that sends $\mathbf{t}$ to $\chi_{\mathbf{t}}$ is a group isomorphism. Applying Theorem \ref{thm:PoissonSummationFinite} to $M=\numZ^n/c\numZ^n$, $x=0$ and to $f$ just defined, we obtain
\begin{equation}
\label{eq:GGUforGeneralcTemp1}
\frac{1}{\abs{H}}\mathfrak{G}_{G}^H(a/c;\mathbf{w})=\frac{1}{\abs{H}}\sum_{\mathbf{y}\in H}f(\mathbf{y})=\frac{1}{c^n}\sum_{\chi\in\widehat{\numZ^n/c\numZ^n}^{H}}\widehat{f}(\chi)\overline{\chi(0)}.
\end{equation}
By the definitions and Lemma \ref{lemm:independentRepr}, for each $\mathbf{x}\in a^{-1}\numZ^n$ we have
\begin{equation}
\label{eq:GGUforGeneralcTemp2}
\mathfrak{G}_{G}(a/c;\mathbf{w},\mathbf{x})=\sum_{\mathbf{v}\in\numZ^n/c\numZ^n}f(\mathbf{v})\etp{\frac{\mathbf{v}^\tp(a\mathbf{x})}{c}}\etp{\frac{a\mathbf{w}^\tp\mathbf{x}}{c}}=\etp{\frac{a\mathbf{w}^\tp\mathbf{x}}{c}}\widehat{f}(\chi_{-a\mathbf{x}}).
\end{equation}
Let $\mathscr{R}\subseteq\numZ^n$ be a complete set of representatives of $\numZ^n/c\numZ^n$. Then when $\mathbf{x}$ runs over the elements in $-a^{-1}\cdot\mathscr{R}$, $\chi_{-a\mathbf{x}}$ precisely runs over all characters in $\widehat{\numZ^n/c\numZ^n}$. Combining this, \eqref{eq:GGUforGeneralcTemp1} and \eqref{eq:GGUforGeneralcTemp2}, we deduce that
\begin{align}
\frac{1}{\abs{H}}\mathfrak{G}_{G}^H(a/c;\mathbf{w})&=\frac{1}{c^n}\cdot\sum_{-a\mathbf{x}\in\mathscr{R},\,\chi_{-a\mathbf{x}}\in\widehat{\numZ^n/c\numZ^n}^{H}}\etp{-\frac{a\mathbf{w}^\tp\mathbf{x}}{c}}\mathfrak{G}_{G}(a/c;\mathbf{w},\mathbf{x})\notag\\
&=\frac{1}{c^n}\cdot\sum_{\mathbf{y}\in\mathscr{R},\,\chi_{\mathbf{y}}\in\widehat{\numZ^n/c\numZ^n}^{H}}\etp{\frac{\mathbf{w}^\tp\mathbf{y}}{c}}\mathfrak{G}_{G}(a/c;\mathbf{w},-a^{-1}\mathbf{y}).\label{eq:GGUforGeneralcTemp3}
\end{align}
According to \eqref{eq:defMdualH}, we have
\begin{align*}
\chi_{\mathbf{y}}\in\widehat{\numZ^n/c\numZ^n}^{H}&\Longleftrightarrow\text{for all }(v_1,v_2,\dots,v_n)+\numZ^n\in H\text{ we have }\sum v_iy_i\equiv0\bmod{c}\\
&\Longleftrightarrow\mathbf{y}+\numZ^n\in H^\bot.
\end{align*}
Inserting this into \eqref{eq:GGUforGeneralcTemp3}, and noting that $c^n=\sqrt{c^n}\sqrt{\abs{H}}\sqrt{\abs{H^\bot}}$ (see Corollary \ref{coro:HHbotequalcn}), we arrive at \eqref{eq:GGUforGeneralc}, as desired.
\end{proof}

\subsection{Proof of Theorem \ref{thm:duality1}}
\label{subsec:proofDuality1}
Note that both $\mathfrak{G}_{G}(a/c;\mathbf{w})$ and $\mathfrak{G}_{G^\bot}(a^\bot/c;-aG\mathbf{w})$ are integral-parametric, so $\mathfrak{G}_{G}^H(a/c;\mathbf{w})$ and $\mathfrak{G}_{G^\bot}^{H^\bot}(a^\bot/c;-aG\mathbf{w})$ are well-defined by Corollary \ref{coro:welldefinedGaussH}. In other words, each factor in \eqref{eq:duality1} makes sense. Since $\gcd(N,c)=1$, we can apply Theorem \ref{thm:mainIndefinite}. Therefore, we substitute \eqref{eq:GGacwxWhenNccoprime} into the right-hand side of \eqref{eq:GGUforGeneralc} and obtain
\begin{multline}
\label{eq:proofDuality1Temp}
\frac{1}{\sqrt{\abs{H}}}\cdot\mathfrak{G}_{G}^H(a/c;\mathbf{w})=\frac{1}{\sqrt{\abs{H^\bot}}}\sum_{\mathbf{y}\in H^\bot}\etp{\frac{\mathbf{w}^\tp\mathbf{y}}{c}}\legendre{\abs{a}}{c}_K^n\legendre{D}{c}_K\etp{\frac{n\cdot\sgn a\cdot(1-c_0)}{8}}\\
\times\etp{-\frac{1}{2ac}\mathbf{y}^\tp G^{-1}\mathbf{y}+\frac{ac'}{2}(\mathbf{w}-a^{-1}G^{-1}\mathbf{y})^\tp G (\mathbf{w}-a^{-1}G^{-1}\mathbf{y})},
\end{multline}
where $c_0$ is the odd part of $c$, and $c'$ is any solution of $cc'\equiv 1\bmod{aN}$. Note that
\begin{align*}
&\etp{\frac{\mathbf{w}^\tp\mathbf{y}}{c}}\etp{-\frac{1}{2ac}\mathbf{y}^\tp G^{-1}\mathbf{y}+\frac{ac'}{2}(\mathbf{w}-a^{-1}G^{-1}\mathbf{y})^\tp G (\mathbf{w}-a^{-1}G^{-1}\mathbf{y})}\\
=&\etp{\frac{a}{2c}\mathbf{w}^\tp G\mathbf{w}}\etp{\frac{cc'-1}{2ac}(\mathbf{y}-aG\mathbf{w})^\tp G^{-1}(\mathbf{y}-aG\mathbf{w})}\\
=&\etp{\frac{a}{2c}\mathbf{w}^\tp G\mathbf{w}}\etp{\frac{a^\bot}{2c}(\mathbf{y}-aG\mathbf{w})^\tp G^{\bot}(\mathbf{y}-aG\mathbf{w})}.
\end{align*}
Inserting this into the right-hand side of \eqref{eq:proofDuality1Temp}, we arrive at \eqref{eq:duality1}, which finishes the proof.

\subsection{The $\gcd(N,c)=N$ case}
\label{subsec:Duality2}
Let the notations and assumptions be the same as in Theorem \ref{thm:duality1}, but now we assume that $G$ is positive definite, and $N\mid c$. The coset $aG\mathbf{w}+G\numZ^n$ is contained in $\numZ^n$. Set
\begin{equation*}
(aG\mathbf{w}+G\numZ^n)/c\numZ^n:=\{\mathbf{x}+c\numZ^n\colon \mathbf{x}\in aG\mathbf{w}+G\numZ^n\},
\end{equation*}
which is a subset of $\numZ^n/c\numZ^n$. Then the $N\mid c$ analog of Theorem \ref{thm:duality1} is the following one.
\begin{thm}
\label{thm:duality2}
In the case $G$ is positive definite and $N\mid c$, we have\footnote{It may happen that $H^\bot\cap(aG\mathbf{w}+G\numZ^n)/c\numZ^n=\emptyset$, in which case $\mathfrak{G}_{G^\bot}^{H^\bot\cap(aG\mathbf{w}+G\numZ^n)/c\numZ^n}\left(\frac{-a'}{Nc},\mathbf{c};-aG\mathbf{w},0\right)=0$.}
\begin{equation}
\label{eq:duality2}
\frac{1}{\sqrt{\abs{H}}}\cdot\mathfrak{G}_{G}^H\left(\frac{a}{c};\mathbf{w}\right)=\gamma\cdot\frac{\sqrt{D}}{\sqrt{\abs{H^\bot}}}\cdot\mathfrak{G}_{G^\bot}^{H^\bot\cap(aG\mathbf{w}+G\numZ^n)/c\numZ^n}\left(\frac{-a'}{Nc},\mathbf{c};-aG\mathbf{w},0\right),
\end{equation}
where $a'$ is any integer with $aa'\equiv1\bmod{cN}$, $\mathbf{c}:=(c,c,\dots,c)\in\numZ^n$, and where
\begin{equation*}
\gamma=\rmi^{\frac{n}{2}}\cdot \mu_G(a,c)\cdot\etp{\frac{a}{2c}\mathbf{w}^\tp G\mathbf{w}}
\end{equation*}
with $\mu_G(a,c)$ given by \eqref{eq:mainB}.
\end{thm}
Before the proof, it is important to notice that $\mathfrak{G}_{G^\bot}\left(\frac{-a'}{Nc},\mathbf{c};-aG\mathbf{w},0\right)$ is not necessarily integral-parametric, since in general $\frac{-a'}{Nc}\cdot\mathbf{c}\not\in\numZ^n$. However, an analog of Lemma \ref{lemm:independentRepr} holds, which ensures $\mathfrak{G}_{G^\bot}^{H^\bot\cap(aG\mathbf{w}+G\numZ^n)/c\numZ^n}\left(\frac{-a'}{Nc},\mathbf{c};-aG\mathbf{w},0\right)$ to be well-defined.
\begin{lemm}
\label{lemm:nonintegralpara}
Let $p\colon\numZ^n\rightarrow \numZ^n/c\numZ^n$ be the natural quotient map. Let $S\subseteq\numZ^n$ be a subset such that $p\vert_{S_j}$ is a bijection from $S$ onto $H^\bot\cap(aG\mathbf{w}+G\numZ^n)/c\numZ^n$. Then the value of
\begin{equation}
\label{eq:sumvinS}
\sum_{\mathbf{v}\in S}\etp{\frac{-a'}{Nc}\cdot\left(\frac{1}{2}(\mathbf{v}-aG\mathbf{w})^\tp\cdot G^\bot\cdot (\mathbf{v}-aG\mathbf{w})\right)}
\end{equation}
is independent of the choice of $S$. In other words, $\mathfrak{G}_{G^\bot}^{H^\bot\cap(aG\mathbf{w}+G\numZ^n)/c\numZ^n}\left(\frac{-a'}{Nc},\mathbf{c};-aG\mathbf{w},0\right)$ is independent of the choice of representatives of the elements of $H^\bot\cap(aG\mathbf{w}+G\numZ^n)/c\numZ^n$, and hence is well-defined.
\end{lemm}
\begin{proof}
Let $\mathbf{v}_1, \mathbf{v}_2$ be vectors such that $p(\mathbf{v}_1)=p(\mathbf{v}_2)\in H^\bot\cap(aG\mathbf{w}+G\numZ^n)/c\numZ^n$. Then
\begin{equation*}
\mathbf{v}_1,\,\mathbf{v}_2\in aG\mathbf{w}+G\numZ^n+c\numZ^n=aG\mathbf{w}+G\numZ^n\, (\text{since } c\numZ^n\subseteq G\numZ^n),
\end{equation*}
and $\Delta\mathbf{v}:=\mathbf{v}_2-\mathbf{v}_1\in c\numZ^n$. It follows that
\begin{align*}
&\etp{\frac{-a'}{Nc}\cdot\left(\frac{1}{2}(\mathbf{v}_2-aG\mathbf{w})^\tp\cdot G^\bot\cdot (\mathbf{v}_2-aG\mathbf{w})\right)}\\
=&\etp{\frac{-a'}{Nc}\cdot\left(\frac{1}{2}(\mathbf{v}_1-aG\mathbf{w})^\tp\cdot G^\bot\cdot (\mathbf{v}_1-aG\mathbf{w})\right)}\etp{-\frac{a'}{c}\Delta\mathbf{v}^\tp G^{-1}(\mathbf{v}_1-aG\mathbf{w})}\etp{-\frac{a'}{2c}\Delta\mathbf{v}^\tp G^{-1}\Delta\mathbf{v}}\\
=&\etp{\frac{-a'}{Nc}\cdot\left(\frac{1}{2}(\mathbf{v}_1-aG\mathbf{w})^\tp\cdot G^\bot\cdot (\mathbf{v}_1-aG\mathbf{w})\right)},
\end{align*}
for $\etp{-\frac{a'}{2c}\Delta\mathbf{v}^\tp G^{-1}\Delta\mathbf{v}}=1$ since $-\frac{a'}{2c}\Delta\mathbf{v}^\tp G^{-1}\Delta\mathbf{v}=-a'\cdot\frac{1}{2}(c^{-1}\Delta\mathbf{v})^\tp cG^{-1}(c^{-1}\Delta\mathbf{v})\in\numZ$, and $\etp{-\frac{a'}{c}\Delta\mathbf{v}^\tp G^{-1}(\mathbf{v}_1-aG\mathbf{w})}=1$ since $\frac{\Delta\mathbf{v}}{c}\in\numZ^n$ and $G^{-1}(\mathbf{v}_1-aG\mathbf{w})\in\numZ^n$. That is to say, each term in \eqref{eq:sumvinS} is independent of the choice of the representative $\mathbf{v}$, so the whole sum \eqref{eq:sumvinS} is independent of the choice of $S$.
\end{proof}
Now we are ready to prove Theorem \ref{thm:duality2}.
\begin{proof}[Proof of Theorem \ref{thm:duality2}]
Inserting \eqref{eq:GGacwxWhenNdivc} into the right-hand side of \eqref{eq:GGUforGeneralc}, we obtain
\begin{equation}
\label{eq:duality2Temp1}
\frac{1}{\sqrt{\abs{H}}}\cdot\mathfrak{G}_{G}^H(a/c;\mathbf{w})=\rmi^{\frac{n}{2}}B\frac{\sqrt{D}}{\sqrt{\abs{H^\bot}}}\cdot\sum_{\mathbf{y}\in H^\bot}\delta_{a\mathbf{w}-G^{-1}\mathbf{y}+\numZ^n,0+\numZ^n}\cdot\etp{\frac{\mathbf{w}^\tp\mathbf{y}}{c}}\etp{-\frac{a'}{2c}\mathbf{y}^\tp G^{-1}\mathbf{y}},
\end{equation}
where $\delta_{a\mathbf{w}-G^{-1}\mathbf{y}+\numZ^n,0+\numZ^n}$ is the Kronecker delta. We have
\begin{equation*}
\etp{\frac{\mathbf{w}^\tp\mathbf{y}}{c}}\etp{-\frac{a'}{2c}\mathbf{y}^\tp G^{-1}\mathbf{y}}=\etp{-\frac{a'}{2Nc}(\mathbf{y}-aG\mathbf{w})^\tp\cdot G^\bot\cdot (\mathbf{y}-aG\mathbf{w})}\etp{\frac{a'a^2}{2c}\mathbf{w}^\tp G\mathbf{w}}.
\end{equation*}
Inserting this into the right-hand side of \eqref{eq:duality2Temp1}, and noting $a\mathbf{w}-G^{-1}\mathbf{y}+\numZ^n=0+\numZ^n$ if and only if $\mathbf{y}\in aG\mathbf{w}+G\numZ^n$, we deduce that
\begin{multline}
\label{eq:duality2Temp2}
\frac{1}{\sqrt{\abs{H}}}\cdot\mathfrak{G}_{G}^H(a/c;\mathbf{w})=\rmi^{\frac{n}{2}}B\etp{\frac{a'a^2}{2c}\mathbf{w}^\tp G\mathbf{w}}\frac{\sqrt{D}}{\sqrt{\abs{H^\bot}}}\\
\cdot\sum_{\mathbf{y}+c\numZ^n\in H^\bot,\,\mathbf{y}\in aG\mathbf{w}+G\numZ^n}\etp{-\frac{a'}{2Nc}(\mathbf{y}-aG\mathbf{w})^\tp\cdot G^\bot\cdot (\mathbf{y}-aG\mathbf{w})}.
\end{multline}
It is immediate that $\rmi^{\frac{n}{2}}B\etp{\frac{a'a^2}{2c}\mathbf{w}^\tp G\mathbf{w}}=\gamma$, since $aa'\equiv1\bmod{cN}$. Thus, \eqref{eq:duality2Temp2} can be reduced to the desired conclusion.
\end{proof}

\section{Intermission: Examples of the main theorem and the duality theorem}
\label{sec:intermission_examples_of_the_main_theorem_and_the_duality_theorem}
In this section, we give some concrete examples, immediate corollaries, and direct applications of Theorems \ref{thm:main}, \ref{thm:mainIndefinite}, \ref{thm:duality1} and \ref{thm:duality2}.

\subsection{The $A_n$ lattices}
\label{subsec:Anlattices}
Let $n\in\numgeq{Z}{1}$, and let $G=A_n$ be the matrix whose $(i,j)$-entry equals $2$ if $i=j$, equals $-1$ if $\abs{i-j}=1$, and equals $0$ otherwise. It is positive definite by, e.g., completing the squares. The lattices $(\numZ^n,(\mathbf{x},\mathbf{y})\mapsto\mathbf{x}^\tp A_n\mathbf{y})$ and $(\numZ^n,(\mathbf{x},\mathbf{y})\mapsto\mathbf{x}^\tp A_n^{-1}\mathbf{y})$ are similar to the weight lattice and the root lattice of the root system of type $A_n$, respectively; cf., e.g. \cite[p. 254 and 264]{TY05}.

Note that the entries of $A_n^{-1}$ can be determined explicitly as follows. If $i\leq j$, then the $(i,j)$-entry of $A_n^{-1}$ equals $\frac{i(n+1-j)}{n+1}$. Other entries are determined by the fact $A_n^{-1}$ is symmetric (since $A_n$ is). It follows that the level of $A_n$ is equal to $2n+2$ if $n$ is odd, and equal to $n+1$ if $n$ is even. Moreover, since $\det(A_n)=2\det(A_{n-1})-\det(A_{n-2})$ for $n\geq3$ and $\det(A_1)=2$, $\det(A_2)=3$, we have $\det(A_n)=n+1$ by induction. For instance, the matrices $A_5$ and $A_5^{-1}$ are provided below:
\begin{equation*}
A_5=\begin{pmatrix}
2 & -1 & 0 & 0 & 0\\
-1 & 2 & -1 & 0 & 0\\
0 & -1 & 2 & -1 & 0\\
0 & 0 & -1 & 2 & -1\\
0 & 0 & 0 & -1 & 2
\end{pmatrix},\quad
6A_5^{-1}=\begin{pmatrix}
5 & 4 & 3 & 2 & 1\\
4 & 8 & 6 & 4 & 2\\
3 & 6 & 9 & 6 & 3\\
2 & 4 & 6 & 8 & 4\\
1 & 2 & 3 & 4 & 5
\end{pmatrix}.
\end{equation*}

For $G=A_n$, the formula \eqref{eq:GGacwxWhenNccoprime} with $\mathbf{w}=0$ becomes
\begin{multline*}
\sum_{v_1,\dots,v_n\bmod{c}}\etp{\frac{a}{c}\left(\sum_{i=1}^nv_i^2-\sum_{i=1}^{n-1}v_iv_{i+1}+\sum_{i=1}^nv_ix_i\right)}=c^{n/2}\legendre{\abs{a}}{c}_K^n\legendre{n+1}{c}_K\\
\cdot\etp{\frac{n\cdot\sgn a\cdot(1-c_0)}{8}}\etp{\frac{a(cc'-1)}{2c(n+1)}\left(\sum_{i=1}^ni(n+1-i)x_i^2+2\sum_{1\leq i<j\leq n}i(n+1-j)x_ix_j\right)},
\end{multline*}
provided that $c>0$, $\gcd(c,n+1)=\gcd(c,a)=1$ and $ax_i\in\numZ$ for all $i$.

\subsection{A condition for the vanishing of Gauss subsums}
\label{subsec:a_condition_for_the_vanishing_of_gauss_subsums}
Here is an interesting application of Theorem \ref{thm:duality2}, which gives a criterion for whether $\mathfrak{G}_{G}^S\left(a/c;\mathbf{w}\right)=0$, where $S$ is a certain small subset of $\numZ^n/c\numZ^n$. This gives an insight into the fact that when $\mathfrak{G}_{G}\left(a/c;\mathbf{w}\right)=0$, this sum can be divided into many subsums, all of which vanish.
\begin{prop}
\label{prop:GzeroTower}
Let $G$ be a positive definite even integral symmetric matrix of size $n\in\numgeq{Z}{1}$, and let $N$ be its level. Let $a$ be a nonzero integer and $c$ be a positive integer with $\gcd(a,c)=1$ and $N\mid c$. Let $\mathbf{w}\in G^{-1}\numZ^n$ satisfying $a\mathbf{w}\not\in\numZ^n$.
\begin{enumerate}
	\item If $H$ is a subgroup of $\numZ^n/c\numZ^n$ containing $cG^{-1}\numZ^n/c\numZ^n$, then $\mathfrak{G}_{G}^H\left(a/c;\mathbf{w}\right)=0$.
	\item For all $\mathbf{x}\in\numZ^n$, we have $\mathfrak{G}_{G}^{(\mathbf{x}+cG^{-1}\numZ^n)/c\numZ^n}\left(a/c;\mathbf{w}\right)=0$.
\end{enumerate}
\end{prop}
\begin{proof}
(1) Note that
\begin{align*}
(G\numZ^n/c\numZ^n)^\bot&=\{\mathbf{x}+c\numZ^n\colon\mathbf{v}^\tp G\mathbf{x}\in c\numZ^n\text{ for all }\mathbf{v}\in\numZ^n\}=cG^{-1}\numZ^n/c\numZ^n,
\end{align*}
and that
\begin{equation*}
H\supseteq(G\numZ^n/c\numZ^n)^\bot\Longleftrightarrow H^\bot\subseteq G\numZ^n/c\numZ^n
\end{equation*}
by the second conclusion of Lemma \ref{lemm:HHbot}. Since $a\mathbf{w}\not\in\numZ^n$, we have
\begin{equation*}
H^\bot\cap(aG\mathbf{w}+G\numZ^n)/c\numZ^n \subseteq G\numZ^n/c\numZ^n\cap(aG\mathbf{w}+G\numZ^n)/c\numZ^n=\emptyset.
\end{equation*}
Thus, the right-hand side of \eqref{eq:duality2} vanishes. This, together with Theorem \ref{thm:duality2}, implies that $\mathfrak{G}_{G}^H\left(a/c;\mathbf{w}\right)=0$, as desired.

(2) Let $H=cG^{-1}\numZ^n/c\numZ^n$. Then $(\mathbf{x}+cG^{-1}\numZ^n)/c\numZ^n=(\mathbf{x}+c\numZ^n)+H$, which is an element of $(\numZ^n/c\numZ^n)/H$. Thus, we have
\begin{equation*}
\mathfrak{G}_{G}^{(\mathbf{x}+cG^{-1}\numZ^n)/c\numZ^n}\left(a/c;\mathbf{w}\right)=\mathfrak{G}_{G}^{(\mathbf{x}+c\numZ^n)+H}\left(a/c;\mathbf{w}\right)=\mathfrak{G}_{G}^{H}\left(a/c;\mathbf{x}+\mathbf{w}\right),
\end{equation*}
where the latter identity follows from the definition. Since $\mathbf{w}\in G^{-1}\numZ^n$ and $a\mathbf{w}\not\in\numZ^n$, we have as well $\mathbf{x}+\mathbf{w}\in G^{-1}\numZ^n$ and $a(\mathbf{x}+\mathbf{w})\not\in\numZ^n$. Therefore, the first conclusion we have just proved implies $\mathfrak{G}_{G}^{H}\left(a/c;\mathbf{x}+\mathbf{w}\right)=0$, i.e., $\mathfrak{G}_{G}^{(\mathbf{x}+cG^{-1}\numZ^n)/c\numZ^n}\left(a/c;\mathbf{w}\right)=0$.
\end{proof}

These are not strong conclusions, since one can also prove them in a rather direct manner, using the orthogonality relation of characters of $G^{-1}\numZ^n/\numZ^n$. However, they tell us that in the case $N\mid c$ and $\delta=0$ in \eqref{eq:GGacwxWhenNdivc}, what the pattern of cancellations in $\mathfrak{G}_{G}(a/c;\mathbf{w},\mathbf{x})$ looks like: this sum actually can be divided into $c^n/\det(G)$ sums, each of which equals $0$ due to the orthogonality relation mentioned above.

\subsection{Explicit formulas of Gauss subsums}
\label{subsec:explicit_formulas_of_gauss_subsums}
If $H^\bot$ in \eqref{eq:duality1} is cyclic, then the right-hand side is a classical Gauss sum, of which explicit formulas are known. Thus, we will have explicit formulas for $\mathfrak{G}_{G}^H(a/c;\mathbf{w})$ for such $H$. Note that for $\mathbf{v}_1,\mathbf{v}_2\in\numZ^n$ or $\numZ^n/c\numZ^n$, by $\mathbf{v}_1\perp\mathbf{v}_2$ we understand $\mathbf{v}_1^\tp\mathbf{v}_2\equiv0\bmod{c}$.
\begin{thm}
\label{thm:GaussSubsumExplicit}
Let $G$ be a nonsingular even integral symmetric matrix of size $n\in\numgeq{Z}{1}$, and let $D$ and $N$ be the determinant and level of $G$, respectively. Let $c\in\numgeq{Z}{1}$ and $a\in\numZ_{\neq0}$ with $\gcd(c,a)=\gcd(c,N)=1$. Let $\mathbf{w}\in G^{-1}\numZ^n$ and $\mathbf{h}=(h_1,\dots,h_n)\in\numZ^n$. Set $u=c/\gcd(c,h_1,\dots,h_n)$. Then we have
\begin{multline}
\label{eq:GaussSubsumExplicit}
\sum_{\mathbf{v}\in\numZ^n/c\numZ^n,\,\mathbf{v}\perp\mathbf{h}}\etp{\frac{a}{c}\cdot\left(\frac{1}{2}(\mathbf{v}+\mathbf{w})^\tp\cdot G\cdot (\mathbf{v}+\mathbf{w})\right)}\\
=c^{n/2}\cdot\legendre{\abs{a}}{c}_K^n\cdot\legendre{D}{c}_K\cdot\etp{\frac{n\cdot\sgn a\cdot(1-c_0)}{8}}\etp{\frac{ac'}{2}\mathbf{w}^\tp G\mathbf{w}}\cdot C,
\end{multline}
where $c_0$ is the odd part of $c$, and $c'$ is any integer satisfying $cc'\equiv1\bmod{aN}$, and where
\begin{equation*}
C=\begin{dcases}
1 &\text{if }\mathbf{h}^\tp G^\bot\mathbf{h}\equiv0\bmod{2c}\text{ and }\mathbf{h}\perp N\mathbf{w}\\
0 &\text{if }\mathbf{h}^\tp G^\bot\mathbf{h}\equiv0\bmod{2c}\text{ and }\mathbf{h}\not\perp N\mathbf{w}\\
u^{-1}\mathfrak{G}_{(2)}\left(\frac{a^\bot\mathbf{h}^\tp G^\bot\mathbf{h}}{2c},u;0,-\frac{2aN\mathbf{w}^\tp\mathbf{h}}{\mathbf{h}^\tp G^\bot\mathbf{h}}\right) &\text{if }\mathbf{h}^\tp G^\bot\mathbf{h}\not\equiv0\bmod{2c},
\end{dcases}
\end{equation*}
with $a^\bot$ an integer subject to the condition $a^\bot aN\equiv-1\bmod{c}$.
\end{thm}
Some remarks are in order before the proof. First, the integer $c'$ can be chosen such that merely $cc'\equiv1\bmod{N}$. The requirement $cc'\equiv1\bmod{aN}$ is presented here for being consistent with its usage in \eqref{eq:GGacwxWhenNccoprime}. In addition, it is immediate that $\mathbf{h}^\tp G^\bot\mathbf{h},\,N\mathbf{w}^\tp\mathbf{h}\in\numZ$, that $2c\mid u\mathbf{h}^\tp G^\bot\mathbf{h}$, and that $c\mid uN\mathbf{w}^\tp\mathbf{h}$. Therefore
\begin{equation}
\label{eq:GaussSuminC}
\mathfrak{G}_{(2)}\left(\frac{a^\bot\mathbf{h}^\tp G^\bot\mathbf{h}}{2c},u;0,-\frac{2aN\mathbf{w}^\tp\mathbf{h}}{\mathbf{h}^\tp G^\bot\mathbf{h}}\right)=\mathfrak{G}_{(2)}\left(\frac{a_1}{c_1},u;0,-\frac{2aN\mathbf{w}^\tp\mathbf{h}}{\mathbf{h}^\tp G^\bot\mathbf{h}}\right),
\end{equation}
where $a_1=\frac{a^\bot u\mathbf{h}^\tp G^\bot\mathbf{h}}{2c}/\gcd\left(u,\frac{u\mathbf{h}^\tp G^\bot\mathbf{h}}{2c}\right)$ and $c_1=u/\gcd\left(u,\frac{u\mathbf{h}^\tp G^\bot\mathbf{h}}{2c}\right)$ are coprime integers. Since $u\cdot a_1/c_1\in\numZ$, we deduce that $\mathfrak{G}_{(2)}\left(\frac{a^\bot\mathbf{h}^\tp G^\bot\mathbf{h}}{2c},u;0,-\frac{2aN\mathbf{w}^\tp\mathbf{h}}{\mathbf{h}^\tp G^\bot\mathbf{h}}\right)$ is integral-parametric if and only if $\gcd\left(u,\frac{u\mathbf{h}^\tp G^\bot\mathbf{h}}{2c}\right)\mid \frac{uN\mathbf{w}^\tp\mathbf{h}}{c}$. Suppose this is the case. If $c_1$ is odd or $4\mid c_1$, then \eqref{eq:GaussSuminC} can be calculated using \eqref{eq:GGacwxWhenNccoprime} or \eqref{eq:GGacwxWhenNdivc}, respectively, with $G=(2)$. Otherwise, if $c_1\equiv2\bmod{4}$ and $\mathbf{h}^\tp G^\bot\mathbf{h}\mid aN\mathbf{w}^\tp\mathbf{h}$, then \eqref{eq:GaussSuminC} equals $0$ by, e.g., \cite[Theorem 1.5.1]{BEW98}. If $c_1\equiv2\bmod{4}$ and $\mathbf{h}^\tp G^\bot\mathbf{h}\nmid aN\mathbf{w}^\tp\mathbf{h}$, then \eqref{eq:GaussSuminC} can be calculated efficiently using \cite[Theorem 1.2.2]{BEW98}. However, if $\gcd\left(u,\frac{u\mathbf{h}^\tp G^\bot\mathbf{h}}{2c}\right)\nmid \frac{uN\mathbf{w}^\tp\mathbf{h}}{c}$, i.e., \eqref{eq:GaussSuminC} is not integral-parametric, we do not know any explicit formula for it.
\begin{proof}
Set $H=\{\mathbf{v}\in\numZ^n/c\numZ^n\colon\mathbf{v}\perp\mathbf{h}\}$, which is a subgroup of $\numZ^n/c\numZ^n$. Then $H=\langle\mathbf{h}+c\numZ^n\rangle^\bot$, and hence by Lemma \ref{lemm:HHbot}, $H^\bot=\langle\mathbf{h}+c\numZ^n\rangle$, the subgroup generated by $\mathbf{h}+c\numZ^n$. It is immediate that $\abs{H^\bot}=\ord(\mathbf{h}+c\numZ^n)=u$, and consequently $\abs{H}=c^n/u$ by Lemma \ref{lemm:HHbot}.

By Theorem \ref{thm:duality1}, we have
\begin{multline}
\label{eq:GaussSubsumExplicitTemp}
\sum_{\mathbf{v}\in\numZ^n/c\numZ^n,\,\mathbf{v}\perp\mathbf{h}}\etp{\frac{a}{c}\cdot\left(\frac{1}{2}(\mathbf{v}+\mathbf{w})^\tp\cdot G\cdot (\mathbf{v}+\mathbf{w})\right)}=\mathfrak{G}_{G}^H(a/c;\mathbf{w})\\
=c^{n/2}\legendre{\abs{a}}{c}_K^n\legendre{D}{c}_K\etp{\frac{n\cdot\sgn a\cdot(1-c_0)}{8}}\etp{\frac{a}{2c}\mathbf{w}^\tp G\mathbf{w}}u^{-1}\mathfrak{G}_{G^\bot}^{H^\bot}(a^\bot/c;-aG\mathbf{w}).
\end{multline}
It remains to calculate $\mathfrak{G}_{G^\bot}^{H^\bot}(a^\bot/c;-aG\mathbf{w})$. Expanding the definition, we find that
\begin{equation}
\label{eq:GaussSubsumExplicitTemp2}
\mathfrak{G}_{G^\bot}^{H^\bot}(a^\bot/c;-aG\mathbf{w})=\etp{\frac{a^2a^\bot N\mathbf{w}^\tp G\mathbf{w}}{2c}}\sum_{v=0}^{u-1}\etp{\frac{a^\bot\mathbf{h}^\tp G^\bot\mathbf{h}}{2c}v^2-\frac{aa^\bot N\mathbf{w}^\tp\mathbf{h}}{c}v}.
\end{equation}
If $\mathbf{h}^\tp G^\bot\mathbf{h}\equiv0\bmod{2c}$, then
\begin{equation*}
\mathfrak{G}_{G^\bot}^{H^\bot}(a^\bot/c;-aG\mathbf{w})=\etp{\frac{a^2a^\bot N\mathbf{w}^\tp G\mathbf{w}}{2c}}\sum_{v=0}^{u-1}\etp{-\frac{aa^\bot N\mathbf{w}^\tp\mathbf{h}}{c}v}=\etp{\frac{a^2a^\bot N\mathbf{w}^\tp G\mathbf{w}}{2c}}uC.
\end{equation*}
(When $\mathbf{h}\not\perp N\mathbf{w}$, the last equality follows from the fact $c\mid uN\mathbf{w}^\tp\mathbf{h}$.) Inserting this into \eqref{eq:GaussSubsumExplicitTemp}, and noting
\begin{equation}
\label{eq:GaussSubsumExplicitTemp3}
\etp{\frac{a}{2c}\mathbf{w}^\tp G\mathbf{w}}\etp{\frac{a^2a^\bot N\mathbf{w}^\tp G\mathbf{w}}{2c}}=\etp{\frac{ac'}{2}\mathbf{w}^\tp G\mathbf{w}},
\end{equation}
we arrive at \eqref{eq:GaussSubsumExplicit} in the case $\mathbf{h}^\tp G^\bot\mathbf{h}\equiv0\bmod{2c}$. On the other hand, if $\mathbf{h}^\tp G^\bot\mathbf{h}\not\equiv0\bmod{2c}$, \eqref{eq:GaussSubsumExplicit} follows from \eqref{eq:GaussSubsumExplicitTemp}--\eqref{eq:GaussSubsumExplicitTemp3}, and
\begin{equation*}
\sum_{v=0}^{u-1}\etp{\frac{a^\bot\mathbf{h}^\tp G^\bot\mathbf{h}}{2c}v^2-\frac{aa^\bot N\mathbf{w}^\tp\mathbf{h}}{c}v}=\mathfrak{G}_{(2)}\left(\frac{a^\bot\mathbf{h}^\tp G^\bot\mathbf{h}}{2c},u;0,-\frac{2aN\mathbf{w}^\tp\mathbf{h}}{\mathbf{h}^\tp G^\bot\mathbf{h}}\right).
\end{equation*}
This concludes the proof.
\end{proof}

\subsection{Hecke Gauss sums}
\label{subsec:hecke_gauss_sums}
Throughout this subsection, let $K$ be a number field, $\mathcal{O}$ be the integral closure of $\numZ$ in $K$, and let $\mathfrak{d}$ and $\Delta$ be the different and discriminant of $K$, respectively. Let $(b_1,b_2,\dots,b_n)$ be an integral basis of $K$, i.e., $\mathcal{O}=\bigoplus_{j=1}^n\numZ b_j$ and $K=\bigoplus_{j=1}^n\numQ b_j$. Let $\tr\colon K\rightarrow\numQ$ and $\mathbf{N}\colon K\rightarrow\numQ$ be the (absolute) trace and norm of $K$, respectively. For a fractional ideal $\mathfrak{a}$, let $\mathbf{N}(\mathbf{\mathfrak{a}})$ be its norm. Note that $\mathbf{N}(\mathbf{\mathfrak{a}})=\abs{\mathcal{O}/\mathfrak{a}}$ if $\mathfrak{a}$ is integral, that $\mathbf{N}(\omega)=\mathbf{N}(\mathcal{O}\omega)$ if $\omega\in K^\times=K\setminus\{0\}$, and that $\mathbf{N}(\mathfrak{d})=\abs{\Delta}$; cf., e.g. \cite[Prop. 22 in p. 26 and Prop. 14 in p. 66]{Lan94}. Set $n=[K\colon\numQ]$.

Let $\omega\in K^\times$. Then there are unique coprime integral ideals $\mathfrak{a}$ and $\mathfrak{b}$ such that $\omega\mathfrak{d}=\mathfrak{b}\mathfrak{a}^{-1}$. The \emph{Hecke Gauss sum} is defined by the formula
\begin{equation*}
C(\omega)=C_K(\omega):=\sum_{\mu\in\mathcal{O}/\mathfrak{a}}\etp{\tr(\mu^2\omega)}.
\end{equation*}
See \cite[Section 1]{BS10} for a good historical overview. As was mentioned by Boylan and Skoruppa in this paper, no explicit formulas were known for $C_K(\omega)$ if $K\neq\numQ$ until they derived one for $K$ being quadratic fields. See also \cite{BS13}, where the authors provided an elegant proof of Hecke’s reciprocity law for $C(\omega)$, based on Milgram's formula.

We shall show that for all $K$ and all $\omega$, $C(\omega)$ is equal to some $\mathfrak{G}_{G}(r,\mathbf{t};\mathbf{0},\mathbf{0})$, up to a simple factor, where $G$ is a nonsingular even integral symmetric matrix of size $n$ depending on $K$ and $\omega$. Then, in the case of quadratic fields, we derive an explicit formula (different from Boylan and Skoruppa's) when Theorem \ref{thm:mainIndefinite} can be applied. Finally, we provide an example where $K$ is a cyclotomic field.

\begin{prop}
\label{prop:HeckeGaussToMatrixGauss}
Let $c_1$ be a positive integer such that $c_1\mathcal{O}\subseteq\mathfrak{a}$. Let $a_1$ be the greatest common divisor of $\tr(c_1\omega b_j^2)$ ($j=1,2,\dots,n$) and $2\tr(c_1\omega b_ib_j)$ ($1\leq i<j\leq n$). Set $c=c_1/\gcd(a_1,c_1)$, $a=a_1/\gcd(a_1,c_1)$, and $G=\left(2a^{-1}c\tr(\omega b_ib_j)\right)_{1\leq i,j\leq n}$. Then $G$ is a nonsingular even integral symmetric matrix (independent of the choice of $c_1$, but depending on the basis $(b_j)$), and
\begin{equation}
\label{eq:HeckeGaussToMatrixGauss}
C_K(\omega)=\mathbf{N}(\mathfrak{a})c^{-n}\mathfrak{G}_{G}(a/c;\mathbf{0}).
\end{equation}
\end{prop}
\begin{proof}
(1) $c_1$ exists since $\mathcal{O}$ and $\mathfrak{a}$ are both lattices in $K$($\cong_\numQ\numQ^n$).

(2) We now show $\tr(c_1\omega b_ib_j)\in\numZ$, from which it follows that $a_1$ is well-defined. The map $(x,y)\mapsto \tr(xy)$ is a $\numQ$-bilinear map on $K\times K$. It is strongly nondegenerate (cf., e.g., \cite[Theorem 8.2.2]{Rom06}). Moreover, $\mathfrak{d}^{-1}$, by the definition, is the dual lattice of $\mathcal{O}$ in $(K,(x,y)\mapsto \tr(xy))$. Thus,
\begin{equation*}
\tr(c_1\omega b_ib_j)\in\tr(\omega c_1\mathcal{O})\subseteq\tr(\omega \mathfrak{a})=\tr(\mathfrak{b}\mathfrak{d}^{-1})\subseteq\numZ.
\end{equation*}

(3) It is immediate that $G$ is even integral and symmetric. Since $K=\bigoplus_{j=1}^n\numQ b_j=\bigoplus_{j=1}^n\numQ \omega b_j$, the nondegeneracy of $G$ follows from the fact $(x,y)\mapsto \tr(xy)$ is strongly nondegenerate. Moreover, a different choice of $c_1$ leads to the same $a$ and $c$, and hence the same $G$.

(4) We now show \eqref{eq:HeckeGaussToMatrixGauss}. We consider the sum $\sum_{\mu\in\mathcal{O}/c_1\mathcal{O}}\etp{\tr(\mu^2\omega)}$, which is well-defined since for $\mu_1,\mu_2\in\mathcal{O}$, the condition $\mu_1+c_1\mathcal{O}=\mu_2+c_1\mathcal{O}$ would imply that
\begin{equation*}
\etp{\tr(\mu_1^2\omega)}=\etp{\tr(\mu_2^2\omega)+\tr(2\mu_2(\mu_1-\mu_2)\omega)+\tr((\mu_1-\mu_2)^2\omega)}=\etp{\tr(\mu_2^2\omega)},
\end{equation*}
as $2\mu_2(\mu_1-\mu_2)\omega,\,(\mu_1-\mu_2)^2\omega\in\omega c_1\mathcal{O}\subseteq\mathfrak{b}\mathfrak{d}^{-1}$. We have
\begin{align*}
\sum_{\mu\in\mathcal{O}/c_1\mathcal{O}}\etp{\tr(\mu^2\omega)}&=\sum_{\mu_1\in\mathcal{O}/\mathfrak{a},\,\mu_2\in\mathfrak{a}/c_1\mathcal{O}}\etp{\tr((\mu_1+\mu_2)^2\omega)}\\
&=\sum_{\mu_1\in\mathcal{O}/\mathfrak{a}}\etp{\tr(\mu_1^2\omega)}\sum_{\mu_2\in\mathfrak{a}/c_1\mathcal{O}}\etp{\tr(2\mu_1\mu_2\omega)+\tr(\mu_2^2\omega)}\\
&=\abs{\mathfrak{a}/c_1\mathcal{O}}\cdot C_K(\omega),
\end{align*}
since $\etp{\tr(2\mu_1\mu_2\omega)+\tr(\mu_2^2\omega)}=1$ as $\mu_1\in\mathcal{O}$ and $\mu_2\in\mathfrak{a}$. It follows that
\begin{align*}
C_K(\omega)&=\abs{\mathfrak{a}/c_1\mathcal{O}}^{-1}\sum_{\mu\in\mathcal{O}/c_1\mathcal{O}}\etp{\tr(\mu^2\omega)}\\
&=\frac{\abs{\mathcal{O}/\mathfrak{a}}}{\abs{\mathcal{O}/c_1\mathcal{O}}}\sum_{\mathbf{v}\in\numZ^n/c_1\numZ^n}\etp{\tr((v_1b_1+v_2b_2+\dots v_nb_n)^2\omega)}\\
&=\mathbf{N}(\mathfrak{a})c_1^{-n}\mathfrak{G}_{G}(a/c,(c_1,c_1,\dots,c_1);\mathbf{0},\mathbf{0}).
\end{align*}
Inserting \eqref{eq:ttoc} with $\mathbf{t}=(c_1,c_1,\dots,c_1)$ into the above formula, we obtain \eqref{eq:HeckeGaussToMatrixGauss} as desired.
\end{proof}
\begin{rema}
We can choose $c_1$ with merely $c_1\omega\in\mathcal{O}$, which is easier to handle. Then it follows that $c_1\omega\mathfrak{d}=c_1\mathfrak{b}\mathfrak{a}^{-1}$ is integral. Consequently, $\mathfrak{a}\mid c_1\mathfrak{b}$, and hence $\mathfrak{a}\mid c_1\mathcal{O}$, i.e., $c_1\mathcal{O}\subseteq\mathfrak{a}$, since $\mathfrak{a}$ and $\mathfrak{b}$ are coprime by the assumption. Therefore, Proposition \ref{prop:HeckeGaussToMatrixGauss} is valid for this $c_1$.
\end{rema}
\begin{rema}
\label{rema:NaFormula}
To compute the right-hand side of \eqref{eq:HeckeGaussToMatrixGauss}, we need to first work out $\mathfrak{a}$, since $\mathbf{N}(\mathfrak{a})$ appears. For certain $K$ and $\omega$, we can know the value of $\mathbf{N}(\mathfrak{a})$ without knowing $\mathfrak{a}$. Assume that each prime factor $p$ of $\Delta$ does not split in $\mathcal{O}$, i.e., there is only one prime $\mathcal{O}$-ideal lying above $p\numZ$. Let $\omega=\alpha/c_1$ with $c_1\in\numgeq{Z}{1}$ and $\alpha\in\mathcal{O}$. We further assume that $\gcd(c_1,\mathbf{N}(\alpha))=1$. Then we would have
\begin{equation}
\label{eq:Naformula}
\mathbf{N}(\mathfrak{a})=\frac{c_1^n}{\gcd(c_1^n,\Delta)}.
\end{equation}
To see this, suppose to the contrary that there is a rational prime $p$ dividing $\mathbf{N}(\mathfrak{a})$ and $\mathbf{N}(\mathfrak{b})$. Then there are prime $\mathcal{O}$-ideals $\mathfrak{p}_1$ and $\mathfrak{p}_2$ lying above $p$ such that $\mathfrak{p}_1\mid\mathfrak{a}$ and $\mathfrak{p}_2\mid\mathfrak{b}$. Since $\mathfrak{a}$ and $\mathfrak{b}$ are coprime, we have $\mathfrak{p}_1\neq\mathfrak{p}_2$. Since $\omega\mathfrak{d}=\mathfrak{b}\mathfrak{a}^{-1}$, i.e., $\alpha\mathfrak{d}(c_1\mathcal{O})^{-1}=\mathfrak{b}\mathfrak{a}^{-1}$, we have $\mathfrak{p}_1\mid c_1\mathcal{O}$ and $\mathfrak{p}_2\mid \alpha\mathfrak{d}$. The assumption $\gcd(c_1,\mathbf{N}(\alpha))=1$ implies $\mathfrak{p}_2\nmid \alpha\mathcal{O}$, so $\mathfrak{p}_2\mid \mathfrak{d}$. Taking the norms we find $p\mid \Delta$, which contradicts the assumption on $\Delta$. We have thus proved $\gcd(\mathbf{N}(\mathfrak{a}),\mathbf{N}(\mathfrak{b}))=1$. Taking the norms of $\alpha\mathfrak{d}(c_1\mathcal{O})^{-1}=\mathfrak{b}\mathfrak{a}^{-1}$, we have $\mathbf{N}(\alpha)\abs{\Delta}c_1^{-n}=\mathbf{N}(\mathfrak{b})\mathbf{N}(\mathfrak{a})^{-1}$, from which \eqref{eq:Naformula} follows.
\end{rema}
\begin{rema}
\label{rema:TKandBW}
Set $T_K(\omega)=\left(\tr(\omega b_ib_j)\right)_{1\leq i,j\leq n}$, then $G=2a^{-1}cT_K(\omega)$. To compute the right-hand side of \eqref{eq:HeckeGaussToMatrixGauss} in practice, we need to work out $T_K(\omega)$. Let $\theta_1,\theta_2,\dots,\theta_n$ be all field embeddings from $K$ to $\numC$ that fix $\numQ$. Then $\tr(x)=\sum_{j=1}^n\theta_j(x)$ for $x\in K$. Set $B=(\theta_j(b_i))_{1\leq i,j\leq n}$ and $W=\mathop{\mathrm{diag}}(\theta_1(\omega),\dots,\theta_n(\omega))$. It is immediate that
\begin{equation*}
T_K(\omega)=BWB^\tp.
\end{equation*}
\end{rema}
For quadratic fields, Proposition \ref{prop:HeckeGaussToMatrixGauss} specializes to the following one.
\begin{prop}
\label{prop:HeckeGaussQuadratic}
Let $d\neq0,1$ be a square-free integer, $K=\numQ(\sqrt{d})$, and let $\omega=\frac{v_0+v_1\sqrt{d}}{c_1}$ where $c_1\in\numgeq{Z}{1}$ and $v_0, v_1\in\numZ$ such that $\gcd(c_1,v_0,v_1)=1$. Set
\begin{equation}
\label{eq:quadratic_a1}
a_1=\begin{dcases}
\gcd(2v_0, 4v_1d) &\text{if }d\not\equiv1\bmod{4}\\
\gcd(2v_0, 2v_1d, v_0+v_1d) &\text{if }d\equiv1\bmod{4},
\end{dcases}
\end{equation}
\begin{align*}
G&=\tbtMat{4v_0a_1^{-1}}{4v_1da_1^{-1}}{4v_1da_1^{-1}}{4v_0da_1^{-1}}\text{ if }d\not\equiv1\bmod{4},\\
G&=\tbtMat{4v_0a_1^{-1}}{2(v_0+v_1d)a_1^{-1}}{2(v_0+v_1d)a_1^{-1}}{(v_0d+v_0+2v_1d)a_1^{-1}}\text{ if }d\equiv1\bmod{4}.
\end{align*}
Set $a=a_1/\gcd(a_1,c_1),\,c=c_1/\gcd(a_1,c_1)$. Then $G$ is nonsingular, even integral, and symmetric, and we have
\begin{equation}
\label{eq:HeckeGaussQuadratic}
C_{\numQ(\sqrt{d})}\left(\frac{v_0+v_1\sqrt{d}}{c_1}\right)=\frac{\mathbf{N}(\mathfrak{a})}{c^2}\mathfrak{G}_{G}\left(\frac{a}{c};\mathbf{0}\right).
\end{equation}
In particular, if $\gcd(c_1,v_0^2-v_1^2d)=1$, then
\begin{equation}
\label{eq:HeckeGaussQuadraticNa}
C_{\numQ(\sqrt{d})}\left(\frac{v_0+v_1\sqrt{d}}{c_1}\right)=\frac{(c_1,a_1)^2}{(c_1^2,\Delta)}\mathfrak{G}_{G}\left(\frac{a}{c};\mathbf{0}\right),
\end{equation}
where $\Delta$ is the discriminant of $\numQ(\sqrt{d})$ ($\Delta=d$ if $d\equiv1\bmod{4}$, and $\Delta=4d$ otherwise).
\end{prop}
\begin{proof}
We take the integral basis $b_1=1$, $b_2=\delta$, where $\delta=(1+\sqrt{d})/2$ if $d\equiv1\bmod{4}$ and $\delta=\sqrt{d}$ otherwise. Let $T_K(\omega)$, $B$ and $W$ be as in Remark \ref{rema:TKandBW}. For $K=\numQ(\sqrt{d})$, $\theta_1$ is the identity map and $\theta_2$ maps $r_1+r_2\sqrt{d}$ to $\overline{r_1+r_2\sqrt{d}}=r_1-r_2\sqrt{d}$, where $r_1,r_2\in\numQ$. Then $B=\tbtmat{1}{1}{\delta}{\overline{\delta}}$, $W=c_1^{-1}\tbtmat{v_0+v_1\sqrt{d}}{0}{0}{v_0-v_1\sqrt{d}}$, and hence $T_{\numQ(\sqrt{d})}(\omega)=c_1^{-1}\tbtmat{2v_0}{2v_1d}{2v_1d}{2v_0d}$ if $d\not\equiv1\bmod{4}$, and $T_{\numQ(\sqrt{d})}(\omega)=c_1^{-1}\tbtmat{2v_0}{v_0+v_1d}{v_0+v_1d}{v_0d/2+v_0/2+v_1d}$ if $d\equiv1\bmod{4}$. The quantity $a_1$ defined by \eqref{eq:quadratic_a1} is the same as the one defined in Proposition \ref{prop:HeckeGaussToMatrixGauss}, and thus $G$ defined here is the same as in Proposition \ref{prop:HeckeGaussToMatrixGauss}. Therefore, the conclusion on $G$ and \eqref{eq:HeckeGaussQuadratic} follow from Proposition \ref{prop:HeckeGaussToMatrixGauss}. Now suppose that $\gcd(c_1,v_0^2-v_1^2d)=1$, that is, $\gcd(c_1,\mathbf{N}(v_0+v_1\sqrt{d}))=1$. Since each prime factor $p$ of $\Delta$ does not split in $\mathcal{O}$ for $K=\numQ(\sqrt{d})$, we conclude by Remark \ref{rema:NaFormula} and \eqref{eq:HeckeGaussQuadratic} that \eqref{eq:HeckeGaussQuadraticNa} holds.
\end{proof}

\begin{thm}
\label{thm:explicitHeckeGaussQuadratic}
Let the notations and assumptions be as in Proposition \ref{prop:HeckeGaussQuadratic}. Suppose that $$\gcd(c, 4\Delta(v_0^2-v_1^2d)/a_1^2)=1.$$ Then we have
\begin{equation}
\label{eq:explicitHeckeGaussQuadratic}
C_{\numQ(\sqrt{d})}\left(\frac{v_0+v_1\sqrt{d}}{c_1}\right)=\frac{\mathbf{N}(\mathfrak{a})}{c}\cdot(-1)^{\frac{1-c_0}{2}}\cdot\legendre{4\Delta(v_0^2-v_1^2d)/a_1^2}{c}_K,
\end{equation}
where $c_0$ is the odd part of $c$. If in addition $\gcd(c_1,v_0^2-v_1^2d)=1$, then
\begin{equation}
\label{eq:explicitHeckeGaussQuadraticNa}
C_{\numQ(\sqrt{d})}\left(\frac{v_0+v_1\sqrt{d}}{c_1}\right)=\frac{c_1(c_1,a_1)}{(c_1^2,\Delta)}\cdot(-1)^{\frac{1-c_0}{2}}\cdot\legendre{4\Delta(v_0^2-v_1^2d)/a_1^2}{c}_K.
\end{equation}
\end{thm}
\begin{proof}
We have $\det(G)=4\Delta(v_0^2-v_1^2d)/a_1^2$. Thus the assumption $\gcd(c, 4\Delta(v_0^2-v_1^2d)/a_1^2)=1$ means $\gcd(c, \det(G))=1$. Since the determinant and level of $G$ have the same set of prime factors, Theorem \ref{thm:mainIndefinite} can be applied to $\mathfrak{G}_{G}(a/c;\mathbf{0})$. Therefore, we insert \eqref{eq:GGacwxWhenNccoprime} into the right-hand sides of \eqref{eq:HeckeGaussQuadratic} and \eqref{eq:HeckeGaussQuadraticNa}, and obtain \eqref{eq:explicitHeckeGaussQuadratic} and \eqref{eq:explicitHeckeGaussQuadraticNa}, respectively.
\end{proof}

To conclude this subsection, we present a specialization of Proposition \ref{prop:HeckeGaussToMatrixGauss} and an example showing that one can compute the exact value of $C_K(\omega)$, where $K$ is a cyclotomic field, by combining Theorem \ref{thm:mainIndefinite} and this specialization.

Let $p$ be an odd rational prime, $\zeta_p=\etp{1/p}$, and $K=\numQ(\zeta_p)$. Then $n=[K\colon\numQ]=p-1$, $\mathcal{O}=\numZ[\zeta_p]$, $p\mathcal{O}=\mathfrak{p}^{p-1}$, $\mathfrak{d}=\mathfrak{p}^{p-2}$, and $\Delta=\legendre{-1}{p}_K\cdot p^{p-2}$, where $\mathfrak{p}$ is the unique prime $\mathcal{O}$-ideal lying above $p\numZ$ (indeed, $\mathfrak{p}$ is the principal ideal generated by $1-\zeta_p$); cf. \cite[Chap. 4]{Lan94}. We shall use the following integral basis
\begin{equation*}
(b_1,b_2,\dots,b_{p-1})=(\zeta_p^1,\zeta_p^2,\dots,\zeta_p^{p-1}).
\end{equation*}
For each $1\leq k\leq p-1$, define a $(p-1)\times(p-1)$ matrix $T_k$ by the formula
\begin{equation*}
\text{the }(i_1,i_2)\text{ entry}=T_k[i_1,i_2]=\begin{dcases}
p-1 &\text{if } i_1+i_2\equiv-k\bmod{p}\\
-1 &\text{otherwise}.
\end{dcases}
\end{equation*}
\begin{prop}
\label{prop:HeckeGaussCyclotomic}
In Proposition \ref{prop:HeckeGaussToMatrixGauss}, if $K=\numQ(\zeta_p)$ and $\omega=c_1^{-1}\sum_{k=1}^{p-1}v_k\zeta_p^k$, where $c_1\in\numgeq{Z}{1}$ and $v_k\in\numZ$ satisfying $\gcd(c_1,v_1,\dots,v_{p-1})=1$, then $a_1$ is equal to the greatest common divisor of
\begin{equation*}
\sum_{k=1}^{p-1}v_kT_k[j,j],\text{ where }j=1,2,\dots,p-1,\text{ and }\sum_{k=1}^{p-1}2v_kT_k[i,j],\text{ where }1\leq i<j\leq p-1,
\end{equation*}
and we have
\begin{equation*}
G=2a_1^{-1}\sum_{k=1}^{p-1}v_kT_k.
\end{equation*}
Moreover, if we assume additionally $\gcd(c_1,\mathbf{N}(\sum v_k\zeta_p^k))=1$, then \eqref{eq:HeckeGaussToMatrixGauss} becomes
\begin{equation}
\label{eq:HeckeGaussCyclotomicgcd1}
C_{\numQ(\zeta_p)}\left(\frac{\sum_{k=1}^{p-1}v_k\zeta_p^k}{c_1}\right)=\frac{(c_1,a_1)^{p-1}}{(c_1^{p-1},p^{p-2})}\mathfrak{G}_{G}\left(\frac{a}{c};\mathbf{0}\right).
\end{equation}
\end{prop}
\begin{proof}
Let $B$, $W$, and $T_K(\omega)$ be the matrices defined in Remark \ref{rema:TKandBW}. The Galois group of $K=\numQ(\zeta_p)$ over $\numQ$ consists of $p-1$ elements: $\theta_j$, $j=1,2,\dots,p-1$, where $\theta_j(\zeta_p)=\zeta_p^j$; cf. \cite[p. 71]{Lan94}. It follows that $B=\left(\etp{ij/p}\right)_{1\leq i,j\leq p-1}$. A direct calculation shows that
$$
T_{\numQ(\zeta_p)}(\omega)=\sum_{k=1}^{p-1}c_1^{-1}v_kB\mathop{\mathrm{diag}}(\zeta_p^k,\zeta_p^{2k},\dots\zeta_p^{(p-1)k})B^\tp=c_1^{-1}\sum_{k=1}^{p-1}v_kT_k.
$$
The assertions on $a_1$ and $G$ thus follow from $G=2a^{-1}cT_K(\omega)$. The last assertion follows from Remark \ref{rema:NaFormula} and the fact $p$ does not split in $\numZ[\zeta_p]$.
\end{proof}
\begin{examp}
\label{examp:Q13HeckeGaussSum}
For $p=13$ and $\omega=37^{-1}\sum_{k=1}^{12}k\zeta_{13}^k\in\numQ(\zeta_{13})$, we have $a_1=13$, $c_1=37$, $a=13$, $c=37$, and
\begin{equation*}
G=\begin{pmatrix}
10 & 8 & 6 & 4 & 2 & 0 & -2 & -4 & -6 & -8 & -10 & -12 \\
8 & 6 & 4 & 2 & 0 & -2 & -4 & -6 & -8 & -10 & -12 & 12 \\
6 & 4 & 2 & 0 & -2 & -4 & -6 & -8 & -10 & -12 & 12 & 10 \\
4 & 2 & 0 & -2 & -4 & -6 & -8 & -10 & -12 & 12 & 10 & 8 \\
2 & 0 & -2 & -4 & -6 & -8 & -10 & -12 & 12 & 10 & 8 & 6 \\
0 & -2 & -4 & -6 & -8 & -10 & -12 & 12 & 10 & 8 & 6 & 4 \\
-2 & -4 & -6 & -8 & -10 & -12 & 12 & 10 & 8 & 6 & 4 & 2 \\
-4 & -6 & -8 & -10 & -12 & 12 & 10 & 8 & 6 & 4 & 2 & 0 \\
-6 & -8 & -10 & -12 & 12 & 10 & 8 & 6 & 4 & 2 & 0 & -2 \\
-8 & -10 & -12 & 12 & 10 & 8 & 6 & 4 & 2 & 0 & -2 & -4 \\
-10 & -12 & 12 & 10 & 8 & 6 & 4 & 2 & 0 & -2 & -4 & -6 \\
-12 & 12 & 10 & 8 & 6 & 4 & 2 & 0 & -2 & -4 & -6 & -8
\end{pmatrix}.
\end{equation*}
Since $\det(G)=2^{12}13^{10}$ and $\mathbf{N}(\sum_{k=1}^{12}k\zeta_{13}^k)=13^{11}$, we have $\gcd(c,\det(G))=1$ and $\gcd(c_1,\mathbf{N}(\sum k\zeta_{13}^k))=1$, respectively, and hence we can substitute \eqref{eq:GGacwxWhenNccoprime} into the right-hand side of \eqref{eq:HeckeGaussCyclotomicgcd1}, which is ensured by Theorem \ref{thm:mainIndefinite}. We thus obtain
\begin{equation*}
C_{\numQ(\zeta_{13})}\left(\frac{\sum_{k=1}^{12}k\zeta_{13}^k}{37}\right)=\mathfrak{G}_{G}\left(\frac{13}{37};\mathbf{0}\right)=37^6.
\end{equation*}
\end{examp}

\section{Application II. Explicit formulas for coefficients of Weil representations of $\slZ$}
\label{sec:application_ii_explicit_formulas_for_coefficients_of_weil_representations_of_slz_}
In this section, we first prove Theorem \ref{thm:WeilCoef}, then present an explicit formula in the case $N\mid c$, and finally translate these formulas into the language of theta series of lattice index. For the definitions and basic properties, see \S \ref{subsec:lattices_and_weil_representations}.

\subsection{Proof of Theorem \ref{thm:WeilCoef}}
\label{subsec:proof_of_theorem_ref_thm_weilcoef}
Choose any $\numZ$-linear isomorphism $\sigma\colon\numZ^n\rightarrow L$, and set $G=\left(B(\sigma(e_i),\sigma(e_j)\right)_{1\leq i,j\leq n}$, where $e_i$ is the $i$th column of the identity matrix. Then $G$ is a positive definite even integral symmetric matrix, and $\sigma$ is an isometry from $(\numZ^n,(\mathbf{v},\mathbf{w})\mapsto\mathbf{v}^\tp G\mathbf{w})$ onto $(L,B)$. Thus, without loss of generality we assume $L=\numZ^n$ and $B(\mathbf{v},\mathbf{w})=\mathbf{v}^\tp G\mathbf{w}$. Note that $\det(G)=\abs{L^\sharp/L}$. Let $\mathbf{w}_1,\mathbf{v}_1\in L^\sharp$ be any representatives of $w,v\in L^\sharp/L$, respectively. Setting in \eqref{eq:GaussSumWeilRepr} $\mathbf{w}=\mathbf{w}_1$ and $\mathbf{x}=-a^{-1}G\mathbf{v}_1$, we derive that
\begin{equation}
\label{eq:proofWeilCoefTemp}
\rho_{\underline{L}}\left(A,1\right)_{w,v}=\rho_{\underline{L}}\left(A,1\right)_{\mathbf{w}_1,\mathbf{v}_1}=D^{-\frac{1}{2}}\rmi^{-\frac{n}{2}}c^{-\frac{n}{2}}\etp{\frac{d}{2c}\mathbf{v}_1^\tp G\mathbf{v}_1}\mathfrak{G}_{G}(a/c;\mathbf{w}_1,-a^{-1}G\mathbf{v}_1).
\end{equation}
Note that $\mathfrak{G}_{G}(a/c;\mathbf{w}_1,-a^{-1}G\mathbf{v}_1)$ is integral-parametric, and that $\gcd(N,c)=1$. ($N$ is defined as the level of $\underline{L}$, so is also that of $G$.) This means Theorem \ref{thm:main}(1) is valid for $\mathfrak{G}_{G}(a/c;\mathbf{w}_1,-a^{-1}G\mathbf{v}_1)$. Inserting \eqref{eq:GGacwxWhenNccoprime} (with $\mathbf{w}=\mathbf{w}_1$ and $\mathbf{x}=-a^{-1}G\mathbf{v}_1$) into the right-hand side of \eqref{eq:proofWeilCoefTemp}, we obtain
\begin{multline}
\label{eq:proofWeilCoefTemp2}
\rho_{\underline{L}}\left(A,1\right)_{w,v}=\frac{1}{\sqrt{D}}\cdot(-\rmi)^{\frac{n}{2}}\cdot\legendre{\abs{a}}{c}_K^n\cdot\legendre{D}{c}_K\cdot\etp{\frac{n\cdot\sgn a\cdot(1-c_0)}{8}}\\
\times\etp{\frac{d}{2c}\mathbf{v}_1^\tp G\mathbf{v}_1}\etp{-\frac{1}{2ac}\mathbf{v}_1^\tp G\mathbf{v}_1+\frac{ac'}{2}(\mathbf{w}_1-a^{-1}\mathbf{v}_1)^\tp G(\mathbf{w}_1-a^{-1}\mathbf{v}_1)}.
\end{multline}
A direct calculation shows that (using $ad-bc=1$)
\begin{multline*}
\etp{\frac{d}{2c}\mathbf{v}_1^\tp G\mathbf{v}_1}\etp{-\frac{1}{2ac}\mathbf{v}_1^\tp G\mathbf{v}_1+\frac{ac'}{2}(\mathbf{w}_1-a^{-1}\mathbf{v}_1)^\tp G(\mathbf{w}_1-a^{-1}\mathbf{v}_1)}\\
=\etp{\frac{b+c'}{a}Q(\mathbf{v}_1)-c'B(\mathbf{w}_1,\mathbf{v}_1)+ac'Q(\mathbf{w}_1)},
\end{multline*}
from which and \eqref{eq:proofWeilCoefTemp2} the desired formula follows.

\begin{rema}
The root of unity $\etp{\frac{b+c'}{a}Q(\mathbf{v}_1)-c'B(\mathbf{w}_1,\mathbf{v}_1)+ac'Q(\mathbf{w}_1)}$ is independent of the choices of representatives $\mathbf{w}_1$ and $\mathbf{v}_1$, so we are safe to write just $\etp{\frac{b+c'}{a}Q(v)-c'B(w,v)+ac'Q(w)}$. This can be seen from the above proof, or alternatively, from the expression itself, since we have $\frac{b+c'}{a}\in\numZ$, which follows from $a\mid bc+cc'$.
\end{rema}

\subsection{The formula in the case $N\mid c$}
\label{subsec:the_formula_in_the_case_nmidc}
Using \eqref{eq:GGacwxWhenNdivc} instead of \eqref{eq:GGacwxWhenNccoprime} we have the following analog of Theorem \ref{thm:WeilCoef} in the case $N\mid c$. This is not a new result; see Remark \ref{rema:WeilCoeffNmidc}.
\begin{thm}
Let $\underline{L}=(L,B)$ be a positive definite even integral lattice of rank $n\in\numgeq{Z}{1}$ and set $Q(x)=\frac{1}{2}B(x,x)$. Let $A=\tbtmat{a}{b}{c}{d}\in\slZ$ with $N\mid c>0$, where $N$ is the level of $\underline{L}$. Then for $w,v\in L^\sharp/L$ we have
\begin{equation}
\label{eq:WeilCoeffNmidc}
\rho_{\underline{L}}\left(A,1\right)_{w,v}=\delta_{aw,v}\cdot\mu_G(a,c)\cdot\etp{abQ(w)},
\end{equation}
where $\mu_G(a,c)$ is defined by \eqref{eq:mainB} with $D=\abs{L^\sharp/L}$, and where $\delta_{aw,v}$ is the Kronecker $\delta$.
\end{thm}
\begin{proof}
If $a=0$, then $c=1$ and hence $N=1$, $L^\sharp=L$. Then $8\mid n$ by Theorem \ref{thm:MilgramExtension2}, and consequently \eqref{eq:WeilCoeffNmidc} holds by \eqref{eq:WeilReprShintani}. Now assume $a\neq0$. As in the last subsection, \eqref{eq:proofWeilCoefTemp} still holds in the present case. Inserting \eqref{eq:GGacwxWhenNdivc} (with $\mathbf{w}=\mathbf{w}_1$ and $\mathbf{x}=-a^{-1}G\mathbf{v}_1$) into the right-hand side of \eqref{eq:proofWeilCoefTemp}, we obtain
\begin{equation*}
\rho_{\underline{L}}\left(A,1\right)_{w,v}=\delta_{aw,v}\cdot\mu_G(a,c)\cdot\etp{\frac{d-a'}{c}Q(v)},
\end{equation*}
where $a'$ satisfies $aa'\equiv 1\bmod{cN}$. The desired formula \eqref{eq:WeilCoeffNmidc} follows from this and the fact when $v=aw$ we have $\etp{\frac{d-a'}{c}Q(v)}=\etp{abQ(w)}$.
\end{proof}
\begin{rema}
\label{rema:WeilCoeffNmidc}
An explicit formula for $\rho_{\underline{L}}\left(A,1\right)_{w,v}$ in the case $N\mid c$, to our best knowledge, was first obtained, in the form of a transformation equation of certain theta series (see the next subsection), by Schoeneberg \cite[Eq. (16)]{Sch39}, using a method developed by Hecke. Schoeneberg considered the case $2\mid n$, and the case $2\nmid n$ was treated by Pfezter \cite[Eq. (22)]{Pfe53}. For more recent reference, see \cite[Theorem 14.3.11]{CS17}. Note that our formula are different from the above three formulas mentioned (e.g., $\mu_G(a,c)$ depends on $a$ and $c$, while the formula in \cite[Theorem 14.3.11]{CS17} expresses $\mu_G(a,c)$ in another form depending on $c$ and $d$), and the proof of the equivalence of these formulas is not very obvious. Finally, note that Richter's results \cite{Ric00_2,Ric00_1} on symplectic theta series associated with possibly indefinite lattices can also be applied to obtain a formula for $\rho_{\underline{L}}\left(A,1\right)_{w,v}$.
\end{rema}

\subsection{Jacobi theta series of lattice index}
\label{subsec:theta_series_of_lattice_index}
Let $\underline{L}=(L,B)$, where $L\subseteq\numR^n$, be a positive definite even integral lattice. Set $Q(x)=\frac{1}{2}B(x,x)$ and extend $B$ to $\numC^n$ by bilinearity. We consider the following theta series
\begin{equation}
\label{eq:deffJacobiThetaLatticeIndex}
\vartheta_{\underline{L}, t}(\tau,z)=\sum_{v \in t+L}\etp{\tau Q(v)+B(v,z)},
\end{equation}
where $\tau\in\uhp:=\{\tau\in\numC\colon\Im\tau>0\}$, $z\in\numC^n$, and $t\in L^\sharp$. It is well-known that this series converges absolutely and uniformly on compact subsets, and hence is holomorphic. Moreover, $\vartheta_{\underline{L}, t}$ depends only on $t+L$, not $t$ itself, so we may assume $t\in L^\sharp/L$. A good reference is \cite[\S 3.5]{Boy15}, where Boylan developed the transformation equations of $\vartheta_{\underline{L}, t}$ for $\mathcal{O}$-lattice $\underline{L}$, based on \cite[Chap. 5]{Ebe13}. By setting $\mathcal{O}=\numZ$, the functions considered by Boylan become \eqref{eq:deffJacobiThetaLatticeIndex}. See \cite[\S 2.3]{Ajo15} as well for another good reference. The modular transformation equations of $\vartheta_{\underline{L}, t}$ are well-known:
\begin{equation}
\label{eq:thetaTransformation}
\varepsilon^{-n}(c\tau+d)^{-n/2}\etp{-\frac{c}{c\tau+d}Q(z)}\vartheta_{\underline{L}, t}\left(\frac{a\tau+b}{c\tau+d},\frac{z}{c\tau+d}\right)=\sum_{x\in L^\sharp/L}\rho_{\underline{L}}\left(\tbtmat{a}{b}{c}{d},\varepsilon\right)_{t,x}\vartheta_{\underline{L}, x}(\tau,z),
\end{equation}
where $\tbtmat{a}{b}{c}{d}\in\slZ$ and $\varepsilon\in\{\pm1\}$. See for instance \cite[Theorem 2.3.4]{Ajo15}. Therefore, Theorem \ref{thm:WeilCoef} has an equivalent form, which provides explicitly the Fourier expansion of $\vartheta_{\underline{L}, t}\left(\frac{a\tau+b}{c\tau+d},\frac{z}{c\tau+d}\right)$.
\begin{thm}
\label{thm:thetaTransformation}
If $c>0$ is coprime to $N$, the level of $\underline{L}$, and $a\neq0$, then
\begin{equation}
\label{eq:thetaabcd}
(c\tau+d)^{-n/2}\etp{-\frac{c}{c\tau+d}Q(z)}\vartheta_{\underline{L}, t}\left(\frac{a\tau+b}{c\tau+d},\frac{z}{c\tau+d}\right)=\gamma\cdot\sum_{v\in L^\sharp}c(v,t)q^{Q(v)}\zeta^v,
\end{equation}
where $q^{Q(v)}:=\etp{\tau Q(v)}$, $\zeta^v:=\etp{B(v,z)}$, and where
\begin{align*}
\gamma&=\frac{1}{\sqrt{D}}\cdot(-\rmi)^{\frac{n}{2}}\cdot\legendre{\abs{a}}{c}_K^n\cdot\legendre{D}{c}_K\cdot\etp{\frac{n\cdot\sgn a\cdot(1-c_0)}{8}},\\
c(v,t)&=\etp{\frac{b+c'}{a}Q(v)-c'B(t,v)+ac'Q(t)}.
\end{align*}
(For $D$, $c'$ and $c_0$, see Theorem \ref{thm:WeilCoef}.)
\end{thm}
\begin{proof}
This follows from combining Theorem \ref{thm:WeilCoef} and \eqref{eq:thetaTransformation}.
\end{proof}
\begin{rema}
If we set $\vartheta_{\underline{L}^\sharp}(\tau,z)=\sum_{v \in L^\sharp}\etp{\tau Q(v)+B(v,z)}$, or equivalently, $\vartheta_{\underline{L}^\sharp}=\sum_{t\in L^\sharp/L}\vartheta_{\underline{L}, t}$, then the right-hand side of \eqref{eq:thetaabcd} can be rewritten as
\begin{equation*}
\gamma\cdot\etp{ac'Q(t)}\cdot\vartheta_{\underline{L}^\sharp}\left(\tau+\frac{b+c'}{a},z-c't\right).
\end{equation*}
The formula \eqref{eq:thetaabcd}, with the right-hand side replaced by the above expression, shows a phenomenon that the action of a modular transformation $\tbtmat{a}{b}{c}{d}$ on $\vartheta_{\underline{L}, t}$ is equivalent to the action of the translation $(\tau,z)\mapsto(\tau+(b+c')/a,z-c't)$ on $\vartheta_{\underline{L}^\sharp}$, up to a factor in $\numQ(\sqrt{D},\zeta_N,\zeta_8)$. It is an interesting question, at least to us, to study whether, for any modular (or Jacobi) form, there is a function that plays the same role as $\vartheta_{\underline{L}^\sharp}$ plays for $\vartheta_{\underline{L}, t}$.
\end{rema}

\section{Application III. Number of points of affine quadratic hypersurfaces over $\mathbb{F}_p$}
\label{sec:application_iii_number_of_points_of_arbitrary_affine_quadratic_hypersurfaces_over_prime_finite_fields}
In this section, we study the congruence \eqref{eq:affineQuadratic}, in particular, its solution number $r_{G,\mathbf{v}}(m)$. In \S \ref{subsec:a_formula_for_arbitrary_modulus}, we provide a formula for $p$ being an arbitrary modulus, not necessarily a prime. The proof is based on Theorem \ref{thm:mainIndefinite} and finite Fourier analysis. Then in \S\ref{subsec:proof_of_theorem_ref_thm_rgvm}, we prove Theorem \ref{thm:rGvm}, as a direct consequence of the above general formula.

\subsection{A formula for arbitrary modulus $c$}
\label{subsec:a_formula_for_arbitrary_modulus}
\begin{thm}
\label{thm:affineQuadraticmodulusc}
Let $G$ be a nonsingular even integral symmetric matrix of size $n\in\numgeq{Z}{1}$. Let $\mathbf{v}\in\numZ^n$, $m$ be an integer, and $c$ be a positive integer. If $\gcd(D,c)=1$, where $D=\det(G)$, then the solution number $r_{G,\mathbf{v}}(m)$ of the congruence 
\begin{equation*}
\frac{1}{2}\mathbf{x}^\tp G\mathbf{x}+\mathbf{v}^\tp\mathbf{x}\equiv m\bmod{c},\quad\mathbf{x}\in\numZ^n/c\numZ^n
\end{equation*}
is given by the formula
\begin{equation}
\label{eq:rGvmGeneralc}
r_{G,\mathbf{v}}(m)=c^{\frac{n}{2}-1}\legendre{D}{c}_K\cdot\sum_{a=0}^{c-1}\etp{\frac{n(1-\widetilde{c}_0)}{8}}(a,c)^\frac{n}{2}\legendre{\widetilde{a}}{\widetilde{c}}^n\legendre{D}{(a,c)}\etp{\frac{a}{c}\left(\frac{\widetilde{c}\cdot\widetilde{c}'-1}{2}\mathbf{v}^\tp G^{-1}\mathbf{v}-m\right)},
\end{equation}
where $\widetilde{a}=a/(a,c)$, $\widetilde{c}=c/(a,c)$, $\widetilde{c}_0$ is the odd part of $\widetilde{c}$, and $\widetilde{c}'$ is an integer satisfies $\widetilde{c}\cdot\widetilde{c}'\equiv1\bmod{\widetilde{a}N}$ ($N$ being the level of $G$).
\end{thm}
Note that for fixed $G$ and $c$, the abbreviations $\widetilde{a}$, $\widetilde{c}$, $\widetilde{c}_0$, and $\widetilde{c}'$ all depend on $a$, so we can not place them outside the sum over $a$.
\begin{proof}
We define a map
\begin{align*}
f_{G,\mathbf{v}}\colon\numZ^n/c\numZ^n&\rightarrow\numZ/c\numZ,\\
\mathbf{x}+c\numZ^n&\mapsto\frac{1}{2}\mathbf{x}^\tp G\mathbf{x}+\mathbf{v}^\tp\mathbf{x}+c\numZ,
\end{align*}
which is well-defined as can be checked directly. Moreover, we regard $r_{G,\mathbf{v}}$, bu abuse of notations, as a map from $\numZ/c\numZ$ to $\numgeq{Z}{0}$. They are related by the following formula:
\begin{equation}
\label{eq:fGvandrGvRelation}
\sum_{\mathbf{x}\in\numZ^n/c\numZ^n}\etp{\frac{af_{G,\mathbf{v}}(\mathbf{x})}{c}}=\sum_{m\in\numZ/c\numZ}r_{G,\mathbf{v}}(m)\etp{\frac{am}{c}},
\end{equation}
where $a$ is an arbitrary integer. This can be seen by regrouping the terms of the left-hand side according to the value of $f_{G,\mathbf{v}}(\mathbf{x})$. We carry over the notations in \S\ref{subsec:fourier_analysis_on_finite_abelian_groups} with $M=\numZ/c\numZ$. Then $r_{G,\mathbf{v}}\in L^2(\numZ/c\numZ)$. For each $a+c\numZ\in\numZ/c\numZ$, let $\chi_{a}$ be the character of $\numZ/c\numZ$ defined by $\chi_{a}(v)=\etp{\frac{av}{c}}$ (The $n=1$ case of the characters introduced in the proof of Lemma \ref{lemm:GGUforGeneralc}). It is immediate that the map from $\numZ/c\numZ$ to $\widehat{\numZ/c\numZ}$ that sends $a+c\numZ$ to $\chi_{a}$ is a group isomorphism. In these notations, \eqref{eq:fGvandrGvRelation} amounts to
\begin{equation*}
\widehat{r_{G,\mathbf{v}}}(\chi_{-a})=\sum_{\mathbf{x}\in\numZ^n/c\numZ^n}\etp{\frac{af_{G,\mathbf{v}}(\mathbf{x})}{c}}.
\end{equation*}
Thus, expanding $r_{G,\mathbf{v}}$ as a linear combination of $\chi_{-a}$, $a\in\numZ/c\numZ$ (the finite Fourier expansion), we find that
\begin{align*}
r_{G,\mathbf{v}}(m)&=\frac{1}{c}\sum_{a\in\numZ/c\numZ}\widehat{r_{G,\mathbf{v}}}(\chi_{-a})\chi_{-a}(m)\\
&=\frac{1}{c}\sum_{a\in\numZ/c\numZ}\mathfrak{G}_{G}(\widetilde{a}/\widetilde{c},\mathbf{c};\mathbf{0},\mathbf{v})\etp{-\frac{am}{c}},
\end{align*}
where $\mathbf{c}=(c,c,\dots,c)$. Since $\gcd(D,c)=1$ we have $\gcd(D,\widetilde{c})=1$ for all $a$. Thus, inserting \eqref{eq:ttoc} and then the conclusion of Theorem \ref{thm:mainIndefinite} into the above formula, we obtain \eqref{eq:rGvmGeneralc} as desired. (The term with $a=0$ should be considered separately since Theorem \ref{thm:mainIndefinite} requires $a\neq0$, or we can just choose another representative, e.g., $a=c$.)
\end{proof}

As an immediate corollary, we have
\begin{equation*}
\abs{r_{G,\mathbf{v}}(m)-c^{n-1}}\leq c^{\frac{n}{2}-1}\sum_{d\mid c,\,d\neq c}d^{\frac{n}{2}}\phi(c/d),
\end{equation*}
where $\phi$ is the Euler totient function.

\begin{rema}
The number $r_{G,\mathbf{v}}(m)$, with $\mathbf{v}=\mathbf{0}$ and $c$ being a prime power, plays an important role in the deduction of the well-known (classical) Siegel Theorem on solution numbers of integral quadratic forms; cf. \cite{Sie36}, or \cite[Chap. II, Section 9]{MH73}. Our formula can be used to derive explicit expressions of the local densities that appear in Siegel's Theorem. For instance, \cite[Chap. II, Lemma 9.8]{MH73} can be seen as a special case of \eqref{eq:rGvmGeneralc} after an elementary simplification. For local densities in more general context, see, for instance, \cite{Yan98}. Finally, it seems that our formula for $r_{G,\mathbf{v}}(m)$, when $\mathbf{v}\neq\mathbf{0}$, can be used to derive explicit formulas for local densities of inhomogeneous quadratic forms (cf., e.g., \cite{Shi04}), which we deem important for future work.
\end{rema}

\subsection{Proof of Theorem \ref{thm:rGvm}}
\label{subsec:proof_of_theorem_ref_thm_rgvm}
Since $p\nmid D$, we can set $c=p$ in \eqref{eq:rGvmGeneralc}, and thus obtain
\begin{equation}
\label{eq:rGvmPrimep}
r_{G,\mathbf{v}}(m)=p^{\frac{n}{2}-1}\legendre{D}{p}_K\cdot\left(p^{\frac{n}{2}}\legendre{D}{p}_K+\sum_{a=1}^{p-1}\etp{\frac{n(1-p_0)}{8}}\legendre{a}{p}_K^n\etp{\frac{a}{p}\left(\frac{pp'-1}{2}\mathbf{v}^\tp G^{-1}\mathbf{v}-m\right)}\right),
\end{equation}
where $p_0=p$ if $p$ is odd and $p_0=1$ if $p=2$, and where $p'$ (depending an both $p$ and $a$) is an integer satisfying $pp'\equiv1\bmod{aN}$. We choose $p'$ such that $pp'\equiv1\bmod{(p-1)!\cdot N}$. Then $pp'\equiv1\bmod{aN}$ for all $1\leq a\leq p-1$, namely, this choice of $p'$ is the same for all $a$. Note that $\frac{pp'-1}{2}\mathbf{v}^\tp G^{-1}\mathbf{v}-m\in\numZ$.

\textbf{Case (1)} $2\mid n$ and $p>2$. The sum $\sum_{a=1}^{p-1}$ in \eqref{eq:rGvmPrimep} is then a geometric sum. Therefore, if $\frac{pp'-1}{2}\mathbf{v}^\tp G^{-1}\mathbf{v}\equiv m\bmod{p}$, then $r_{G,\mathbf{v}}(m)=p^{n-1}+\legendre{D}{p}_K\cdot\etp{\frac{n(1-p)}{8}}\cdot p^{\frac{n}{2}-1}(p-1)$. On the other hand, if $\frac{pp'-1}{2}\mathbf{v}^\tp G^{-1}\mathbf{v}\not\equiv m\bmod{p}$, then $r_{G,\mathbf{v}}(m)=p^{n-1}-\legendre{D}{p}_K\cdot\etp{\frac{n(1-p)}{8}}\cdot p^{\frac{n}{2}-1}$. It is immediate that
\begin{equation*}
\frac{pp'-1}{2}\mathbf{v}^\tp G^{-1}\mathbf{v}\equiv m\bmod{p}\Longleftrightarrow Nm+\frac{1}{2}\mathbf{v}^\tp G^\bot\mathbf{v}\equiv0\bmod{p},
\end{equation*}
from which Theorem \ref{thm:rGvm}(1) follows.

\textbf{Case (2)} $2\mid n$ and $p=2$. Then \eqref{eq:rGvmPrimep} becomes
\begin{equation*}
r_{G,\mathbf{v}}(m)=2^{n-1}+\legendre{D}{2}_K\cdot(-1)^{\frac{pp'-1}{2}\mathbf{v}^\tp G^{-1}\mathbf{v}-m}\cdot2^{\frac{n}{2}-1}.
\end{equation*}
Since $p=2$, we have $2\nmid N$, and hence
\begin{equation*}
(-1)^{\frac{pp'-1}{2}\mathbf{v}^\tp G^{-1}\mathbf{v}-m}=(-1)^{Nm+\frac{1}{2}\mathbf{v}^\tp G^\bot\mathbf{v}}
\end{equation*}
as in the previous case. Therefore, Theorem \ref{thm:rGvm}(2) holds.

\textbf{Case (3)} $2\nmid n$. Then $2\mid N$ (see \cite[Remarks 14.3.23]{CS17}), and hence $p>2$. Now \eqref{eq:rGvmPrimep} becomes
\begin{equation*}
r_{G,\mathbf{v}}(m)=p^{\frac{n}{2}-1}\legendre{D}{p}_K\cdot\left(p^{\frac{n}{2}}\legendre{D}{p}_K+\etp{\frac{n(1-p)}{8}}\sum_{a=1}^{p-1}\legendre{a}{p}_K\etp{\frac{at_m'}{p}}\right),
\end{equation*}
where $t_m'=\frac{pp'-1}{2}\mathbf{v}^\tp G^{-1}\mathbf{v}-m$. As in the previous cases, we have $t_m'\equiv -t_m\bmod{p}$. By the classical formula on Gauss character sums (cf., e.g., \cite[Prop. 3.4.5(a) and 3.4.10(b)]{CS17}), we have
\begin{equation*}
r_{G,\mathbf{v}}(m)=p^{n-1}+p^{\frac{n}{2}-1}\legendre{D}{p}_K\etp{\frac{n(1-p)}{8}}\sqrt{p}\legendre{-2t_m}{p}_K\etp{\frac{1-p}{8}},
\end{equation*}
from which Theorem \ref{thm:rGvm}(3) follows.

\section{Application IV. Number of solutions of generalized Markoff equations over $\mathbb{F}_p$}
\label{sec:application_iv_number_of_solutions_of_generalized_markoff_equations_over_prime_finite_fields}
Here we carry over the notations introduced in \S\ref{subsec:generalized_markoff_equations} and Theorem \ref{thm:solnumbergMarkoff}. The aim of this section is to prove Theorem \ref{thm:solnumbergMarkoff}. Besides $A_{11}$, $A_{22}$ and $A_{33}$, we further set
\begin{equation}
\label{eq:defA12A13A23}
A_{12}=-2a_{12}a_{33}+a_{13}a_{23},\quad A_{13}=a_{12}a_{23}-2a_{22}a_{13},\quad A_{23}=-2a_{11}a_{23}+a_{12}a_{13},
\end{equation}
so that $G\cdot A=A\cdot G=D\cdot I_3$, where $A=(A_{ij})_{1\leq i,j\leq3}$ with $A_{ij}=A_{ji}$.
\begin{lemm}
We have
\begin{align}
A_{22}A_{33}-A_{23}^2&=D\cdot2a_{11},\qquad&-A_{12}A_{33}+A_{23}A_{13}&=D\cdot a_{12},\label{eq:Adet1}\\
A_{11}A_{33}-A_{13}^2&=D\cdot2a_{22},\qquad &A_{12}A_{23}-A_{22}A_{13}&=D\cdot a_{13},\label{eq:Adet2}\\
A_{11}A_{22}-A_{12}^2&=D\cdot2a_{33},\qquad&-A_{11}A_{23}+A_{12}A_{13}&=D\cdot a_{23},\label{eq:Adet3}
\end{align}
and
\begin{align}
D&=2a_{11}A_{11}+a_{12}A_{12}+a_{13}A_{13},\label{eq:Dexpansion1}\\
D&=a_{12}A_{12}+2a_{22}A_{22}+a_{23}A_{23},\label{eq:Dexpansion2}\\
D&=a_{13}A_{13}+a_{23}A_{23}+2a_{33}A_{33}.\label{eq:Dexpansion3}
\end{align}
\end{lemm}
\begin{proof}
A straightforward calculation.
\end{proof}

The proof for the case $p=2$ is relatively easy:
\begin{lemm}
\label{lemm:solnumbergMarkoffp2}
Theorem \ref{thm:solnumbergMarkoff} holds in the case $p=2$.
\end{lemm}
\begin{proof}
Since we have assumed $p\nmid a_{11}a_{22}a_{33}d$, we can assume $a_{11}=a_{22}=a_{33}=d=1$ for $p=2$ without loss of generality. There are 8 possibilities of the value of $(a_{12}\bmod{2},a_{13}\bmod{2},a_{23}\bmod{2})$. For each value, we solve \eqref{eq:gMarkoff} modulo $2$ by checking whether each value of $(x,y,z)\in\numZ^3/2\numZ^3$ satisfies it directly. In this way we find that Theorem \ref{thm:solnumbergMarkoff}(4), i.e., the case $p=2$ of Theorem \ref{thm:solnumbergMarkoff}, holds.
\end{proof}

To prove Theorem \ref{thm:solnumbergMarkoff} for $p>2$, we partition the affine variety determined by \eqref{eq:gMarkoff} into the sections $z=j$, where $j=0,1,\dots,p-1$, and then show at most two of these sections are singular affine quadratic varieties, and the rest are nonsingular ones.
\begin{lemm}
\label{lemm:a11pa22p}
Let $p$ be an odd prime with $p\nmid a_{11}a_{22}a_{33}d$, and let $t_1,t_2$ be nonnegative integers. Set $a_{11}'=a_{11}+pt_1$, $a_{22}'=a_{22}+pt_2$. For $z=0,1,\dots,p-1$, set
\begin{align*}
G_z&=\tbtMat{2a_{11}'}{a_{12}-dz}{a_{12}-dz}{2a_{22}'},\quad&\mathbf{v}_z&=(a_{13}z,a_{23}z),\\
D_z&=\det(G_z)=4a_{11}'a_{22}'-(a_{12}-dz)^2,\quad &N_z&=\text{ level of }G_z.
\end{align*}
Then there exist $t_1$ and $t_2$ such that:
\begin{enumerate}
	\item $G_z$ is positive definite for all $0\leq z\leq p-1$,
	\item $N_z=D_z$ for all $0\leq z\leq p-1$,
	\item The solution number of \eqref{eq:gMarkoff} modulo $p$ is equal to
	\begin{equation}
	\label{eq:partitionSections}
	\sum_{z=0}^{p-1}r_{G_z,\mathbf{v}_z}(-a_{33}z^2),
	\end{equation}
	where $r_{G_z,\mathbf{v}_z}(-a_{33}z^2)$ is defined in Theorem \ref{thm:rGvm}.
\end{enumerate}
\end{lemm}
\begin{proof}
If $t_1$ and $t_2$ are sufficiently large, then $G_z$ is positive definite for all $0\leq z\leq p-1$. Among these large values we choose $t_1$ and $t_2$ such that $\gcd(a_{11}',a_{22}')=1$, which is possible by Dirichlet's theorem on primes in arithmetic progressions, or by more elementary \cite[Lemma 6.1.1]{CS17}. We have
\begin{equation*}
G_z^{-1}=D_z^{-1}\tbtMat{2a_{22}'}{-(a_{12}-dz)}{-(a_{12}-dz)}{2a_{11}'}.
\end{equation*}
Therefore $D_zG_z^{-1}$ is even integral, and hence $N_z\mid D_z$. Since $a_{11}'$ and $a_{22}'$ are coprime, $mG_z^{-1}$ is not even integral for any proper positive divisor $m$ of $D_z$, so $N_z=D_z$.

Note that the solution number of \eqref{eq:gMarkoff} modulo $p$ is equal to the solution number of it with $a_{11}$, $a_{22}$ replaced by $a_{11}'$, $a_{22}'$, respectively, since $a_{jj}\equiv a_{jj}'\bmod{p}$. The latter equation can be rewritten as
\begin{equation*}
\frac{1}{2}\begin{pmatrix}x&y\end{pmatrix}G_z\begin{pmatrix}x\\y\end{pmatrix}+\mathbf{v}_{z}^\tp\begin{pmatrix}x\\y\end{pmatrix}\equiv-a_{33}z^2\bmod{p},
\end{equation*}
from which \eqref{eq:partitionSections} follows. (This formula holds for whatever $t_1$ and $t_2$.)
\end{proof}

From here on, $a_{11}'$ and $a_{22}'$ are always chosen, and $G_z$, $\mathbf{v}_z$, $D_z$, and $N_z$ are always defined as in Lemma \ref{lemm:a11pa22p}. The sections $z=j$ considered above are divided into three disjoint classes:
\begin{align}
S_1&=\left\{z\in\{0,1,\dots,p-1\}\colon N_z\equiv0\bmod{p}\right\},\label{eq:zS1}\\
S_2&=\left\{z\in\{0,1,\dots,p-1\}\colon N_z\not\equiv0\bmod{p},\quad-N_za_{33}z^2+\frac{1}{2}\mathbf{v}_z^\tp N_zG_z^{-1}\mathbf{v}_z\equiv0\bmod{p}\right\},\label{eq:zS2}\\
S_3&=\left\{z\in\{0,1,\dots,p-1\}\colon N_z\not\equiv0\bmod{p},\quad-N_za_{33}z^2+\frac{1}{2}\mathbf{v}_z^\tp N_zG_z^{-1}\mathbf{v}_z\not\equiv0\bmod{p}\right\}.\label{eq:zS3}
\end{align}
For $z\in S_2$ or $z\in S_3$, $r_{G_z,\mathbf{v}_z}(-a_{33}z^2)$ can be evaluated\footnote{This will be the only place where we apply our main theorem. After applying it, we will obtain some intermediate formulas for the solution numbers as one will see in the proof below (e.g. \eqref{eq:proofThm1.6Temp2}). However, to acquire a formula in a relatively beautiful form like the one given in Theorem \ref{thm:solnumbergMarkoff}, we need more effort. This extra effort is a lengthy, tedious, and elementary reasoning, which is split into fifteen cases.} using Theorem \ref{thm:rGvm}(1), where $\delta$ equals $1$ or $0$, respectively. For $z\in S_1$, the following elementary fact will be used:
\begin{lemm}
\label{lemm:rGvSingularQuad}
Let $p$ be an odd prime, let $G=\tbtmat{2g_{11}}{g_{12}}{g_{12}}{2g_{22}}$, where $g_{11},g_{12},g_{22}\in\numZ$, let $\mathbf{v}=(v_1,v_2)\in\numZ^2$, and let $m\in\numZ$. Suppose that $p\mid\det(G)$, and $p\nmid g_{11}$. Then
\begin{enumerate}
	\item If $p\nmid g_{12}v_1-2g_{11}v_2$, then $r_{G,\mathbf{v}}(m)=p$.
	\item If $p\mid g_{12}v_1-2g_{11}v_2$, then $$r_{G,\mathbf{v}}(m)=p\cdot\left(1+\legendre{4g_{11}m+v_1^2}{p}_K\right).$$
\end{enumerate}
\end{lemm}
\begin{proof}
For $G$, $\mathbf{v}$ given above, the congruence \eqref{eq:affineQuadratic} becomes
\begin{equation*}
\frac{1}{2}\begin{pmatrix}x&y\end{pmatrix}\tbtMat{2g_{11}}{g_{12}}{g_{12}}{2g_{22}}\begin{pmatrix}x\\y\end{pmatrix}+\begin{pmatrix}v_1&v_2\end{pmatrix}\begin{pmatrix}x\\y\end{pmatrix}\equiv m\bmod{p}.
\end{equation*}
Multiplying $4g_{11}$ on both sides, and noting that $4g_{11}g_{22}y^2\equiv g_{12}^2y^2\bmod{p}$ (since $p\mid\det(G)$), we find this congruence is equivalent to
\begin{equation}
\label{eq:singularQuadraticCongTemp}
(2g_{11}x+g_{12}y+v_1)^2\equiv2(g_{12}v_1-2g_{11}v_2)y+4g_{11}m+v_1^2\bmod{p}.
\end{equation}

If $p\nmid g_{12}v_1-2g_{11}v_2$, then there are exactly $(p-1)/2$ values for $y$ in $\numZ/p\numZ$ such that the right-hand side of \eqref{eq:singularQuadraticCongTemp} is a quadratic residue mod $p$, and if $y$ assumes each such value, there are exactly two choices of $x$ such that \eqref{eq:singularQuadraticCongTemp} holds. On the other hand, there are exact one value for $y$ such that the right-hand side of \eqref{eq:singularQuadraticCongTemp} is congruent to $0\bmod{p}$, and if $y$ assumes this value, there is exactly one choice of $x$ such that \eqref{eq:singularQuadraticCongTemp} holds. Therefore,
\begin{equation*}
r_{G,\mathbf{v}}(m)=\frac{p-1}{2}\cdot2+1\cdot1=p,
\end{equation*}
which proves the first assertion.

If $p\mid g_{12}v_1-2g_{11}v_2$, then \eqref{eq:singularQuadraticCongTemp} is equivalent to
\begin{equation*}
(2g_{11}x+g_{12}y+v_1)^2\equiv4g_{11}m+v_1^2\bmod{p}.
\end{equation*}
The rest of the proof is then immediate.
\end{proof}

The congruence appearing in \eqref{eq:zS2} is connected to $A_{11}$ and $A_{22}$ as the following lemma shows:
\begin{lemm}
\label{lemm:S2cong}
Let $p$ be an odd prime with $p\nmid a_{11}a_{22}a_{33}d$. Then $-N_za_{33}z^2+\frac{1}{2}\mathbf{v}_z^\tp N_zG_z^{-1}\mathbf{v}_z\equiv0\bmod{p}$ holds if and only if $z\equiv0\bmod{p}$ or
\begin{equation}
\label{eq:S2congzneq0}
(a_{13}a_{23}-2a_{33}(a_{12}-dz))^2\equiv A_{11}A_{22}\bmod{p}.
\end{equation}
\end{lemm}
\begin{proof}
First using Lemma \ref{lemm:a11pa22p}(2) to replace $N_z$ by $D_z$, and then a straightforward calculation.
\end{proof}

The following classical formula will be used when we evaluate the sum \eqref{eq:partitionSections}:
\begin{lemm}
\label{lemm:legendreSum}
Let $p$ be an odd prime, and let $a,b,c\in\numZ$ with $p\nmid a$. Then
\begin{equation*}
\sum_{x=0}^{p-1}\legendre{ax^2+bx+c}{p}_K=\begin{dcases}
-\legendre{a}{p}_K &\text{ if }b^2-4ac\not\equiv0\bmod{p}\\
(p-1)\cdot\legendre{a}{p}_K &\text{ if }b^2-4ac\equiv0\bmod{p}.
\end{dcases}
\end{equation*}
\end{lemm}
\begin{proof}
See \cite[Theorem 2.1.2]{BEW98} and the example following it.
\end{proof}

The proof of Theorem \ref{thm:solnumbergMarkoff} (for $p>2$) will be split into the following cases (all congruences are modulo $p$)
\begin{align*}
\text{\textbf{Case (1)}  }&\legendre{A_{11}A_{22}}{p}_K=-1.\\
\text{\textbf{Case (2)}  }&\legendre{A_{11}A_{22}}{p}_K=0,\quad A_{12}\equiv0.\\
\text{\textbf{Case (3)}  }&\legendre{A_{11}A_{22}}{p}_K=0,\quad A_{12}\not\equiv0. \\
\text{\textbf{Case (4)}  }&\legendre{A_{11}A_{22}}{p}_K=1,\quad a_{11}A_{11}\not\equiv a_{22}A_{22},\quad A_{11}A_{22}-A_{12}^2\equiv0.\\
\text{\textbf{Case (5)}  }&\legendre{A_{11}A_{22}}{p}_K=1,\quad a_{11}A_{11}\not\equiv a_{22}A_{22},\quad A_{11}A_{22}-A_{12}^2\not\equiv0.\\
\text{\textbf{Case (6)}  }&\legendre{A_{11}A_{22}}{p}_K=1,\quad a_{11}A_{11}\equiv a_{22}A_{22},\quad A_{11}A_{22}-A_{12}^2\equiv0,\quad a_{13}\not\equiv0,\quad A_{23}\not\equiv0.\\
\text{\textbf{Case (7)}  }&\legendre{A_{11}A_{22}}{p}_K=1,\quad a_{11}A_{11}\equiv a_{22}A_{22},\quad A_{11}A_{22}-A_{12}^2\equiv0,\quad a_{13}\equiv0,\quad A_{23}\not\equiv0.\\
\text{\textbf{Case (8)}  }&\legendre{A_{11}A_{22}}{p}_K=1,\quad a_{11}A_{11}\equiv a_{22}A_{22},\quad A_{11}A_{22}-A_{12}^2\equiv0,\quad a_{13}\not\equiv0,\quad A_{23}\equiv0.\\
\text{\textbf{Case (9)}  }&\legendre{A_{11}A_{22}}{p}_K=1,\quad a_{11}A_{11}\equiv a_{22}A_{22},\quad A_{11}A_{22}-A_{12}^2\equiv0,\quad a_{13}\equiv0,\quad A_{23}\equiv0,\\
&A_{33}\equiv0.\\
\text{\textbf{Case (10)}  }&\legendre{A_{11}A_{22}}{p}_K=1,\quad a_{11}A_{11}\equiv a_{22}A_{22},\quad A_{11}A_{22}-A_{12}^2\equiv0,\quad a_{13}\equiv0,\quad A_{23}\equiv0,\\
&A_{33}\not\equiv0.\\
\text{\textbf{Case (11)}  }&\legendre{A_{11}A_{22}}{p}_K=1,\quad a_{11}A_{11}\equiv a_{22}A_{22},\quad A_{11}A_{22}-A_{12}^2\not\equiv0,\quad a_{13}\not\equiv0,\quad A_{23}\not\equiv0.\\
\text{\textbf{Case (12)}  }&\legendre{A_{11}A_{22}}{p}_K=1,\quad a_{11}A_{11}\equiv a_{22}A_{22},\quad A_{11}A_{22}-A_{12}^2\not\equiv0,\quad a_{13}\equiv0,\quad A_{23}\not\equiv0.\\
\text{\textbf{Case (13)}  }&\legendre{A_{11}A_{22}}{p}_K=1,\quad a_{11}A_{11}\equiv a_{22}A_{22},\quad A_{11}A_{22}-A_{12}^2\not\equiv0,\quad a_{13}\not\equiv0,\quad A_{23}\equiv0.\\
\text{\textbf{Case (14)}  }&\legendre{A_{11}A_{22}}{p}_K=1,\quad a_{11}A_{11}\equiv a_{22}A_{22},\quad A_{11}A_{22}-A_{12}^2\not\equiv0,\quad a_{13}\equiv0,\quad A_{23}\equiv0,\\
&A_{33}\equiv0.\\
\text{\textbf{Case (15)}  }&\legendre{A_{11}A_{22}}{p}_K=1,\quad a_{11}A_{11}\equiv a_{22}A_{22},\quad A_{11}A_{22}-A_{12}^2\not\equiv0,\quad a_{13}\equiv0,\quad A_{23}\equiv0,\\
&A_{33}\not\equiv0.
\end{align*}
\begin{lemm}
\label{lemm:Dmodp}
Let $p$ be an odd prime with $p\nmid a_{11}a_{22}a_{33}d$. The above fifteen cases are mutually exclusive, and cover all possibilities. Moreover, Cases (7), (10), (12), (13), (14) cannot occur. Finally, $p\mid D$ if and only if one of the cases (2), (4), (6), (8), (9) occurs, and $p\nmid D$ if and only if one of the cases (1), (3), (5), (11), (15) occurs.
\end{lemm}
\begin{proof}
It is immediate that these cases are mutually exclusive, and cover all possibilities.

To prove Case (7) is impossible, we assume its conditions hold, and we shall deduce a contradiction. Subtracting \eqref{eq:Dexpansion2} from \eqref{eq:Dexpansion1}, and noting $a_{11}A_{11}\equiv a_{22}A_{22}$, we find that $a_{13}A_{13}\equiv a_{23}A_{23}$. Since $a_{13}\equiv0$ and $A_{23}\not\equiv0$, we then have $a_{23}\equiv0$. Thus, by the third equation in \eqref{eq:defA12A13A23}, we have $A_{23}\equiv0$, a contradiction, which shows Case (7) is impossible.

For Case (10), assume by contradiction that it holds. Then $D\equiv0$ by $A_{11}A_{22}-A_{12}^2\equiv0$ and the first formula of \eqref{eq:Adet3}. By the first formula of \eqref{eq:Adet1} and $A_{23}\equiv0$ we know that $A_{22}A_{33}\equiv0$, which contradicts the assumptions $\legendre{A_{11}A_{22}}{p}_K=1$ and $A_{33}\not\equiv0$. This shows Case (10) never occurs.

For Case (12), assume by contradiction that it holds. Since $a_{13}\equiv0$, we have $A_{22}\equiv4a_{11}a_{33}$ by the definition of $A_{22}$. Thus $a_{11}A_{11}\equiv a_{22}A_{22}\equiv4a_{11}a_{22}a_{33}$, which implies $A_{11}\equiv4a_{22}a_{33}$. By this and the definition of $A_{11}$, we have $a_{23}\equiv0$. It follows from this, the assumption $a_{13}\equiv0$, and the third equation of \eqref{eq:defA12A13A23} that $A_{23}\equiv0$, which contradicts one of the assumptions. This shows Case (12) never occurs.

For Case (13), assume by contradiction that it holds. As in Case (7), we have $a_{13}A_{13}\equiv a_{23}A_{23}$. This, together with the assumptions $a_{13}\not\equiv0$ and $A_{23}\equiv0$, implies $A_{13}\equiv0$. By the second equation of \eqref{eq:Adet2}, we then have $Da_{13}\equiv0$. However, $D\not\equiv0$ (since $A_{11}A_{22}-A_{12}^2\not\equiv0$ and the first equation of \eqref{eq:Adet3}), so $a_{13}\equiv0$, which contradicts one of the assumptions. This shows Case (13) never occurs.

For Case (14), assume by contradiction that it holds. Since $A_{11}A_{22}-A_{12}^2\not\equiv0$ we have $D\not\equiv0$ by the first equation of \eqref{eq:Adet3}. By the assumptions $a_{13}\equiv0$, $A_{23}\equiv0$, and $A_{33}\equiv0$, we have $D=a_{13}A_{13}+a_{23}A_{23}+2a_{33}A_{33}\equiv0$ (according to \eqref{eq:Dexpansion3}), contradicts $D\not\equiv0$ just deduced. This shows Case (14) never occurs.

We now examine other cases to see whether $p\mid D$ or not. In Case (1), it follows from the first equation of \eqref{eq:Adet3} that $p\nmid D$. In Cases (2) and (3), it follows again from the first equation of \eqref{eq:Adet3} that $p\mid D$ and $p\nmid D$, respectively. In Cases (4), (5), (6), (8), (9), (11), (15), there are assumptions on $A_{11}A_{22}-A_{12}^2\bmod{p}$. By the first equation of \eqref{eq:Adet3}, we have $p\mid D$ if and only if $p\mid A_{11}A_{22}-A_{12}^2$. Thus, in these cases, whether $p\mid D$ or not is clear.
\end{proof}

In Cases (1)--(5), the following observation will significantly simplify the proof of Theorem \ref{thm:solnumbergMarkoff} by reducing the number of subcases.
\begin{lemm}
\label{lemm:cases1-5rz}
Let $p$ be an odd prime with $p\nmid a_{11}a_{22}a_{33}d$, and assume we are in one of the cases (1)--(5). Then each element $z\in S_1$ (see \eqref{eq:zS1}), if exists, satisfies
\begin{equation*}
p\nmid -da_{13}z^2+A_{23}z\quad\text{ or }\quad p\mid A_{22}z^2,
\end{equation*}
and (as a consequence) we have $r_{G_z,\mathbf{v}_z}(-a_{33}z^2)=p$.
\end{lemm}
\begin{proof}
Assume that $p\mid -da_{13}z^2+A_{23}z$, we need to prove $p\mid A_{22}z^2$. Since $z\in S_1$, we have $(a_{12}-dz)^2\equiv4a_{11}a_{22}\bmod{p}$ (see Lemma \ref{lemm:a11pa22p}(2)).

If $z=0$, then obviously $p\mid A_{22}z^2$.

If $z\neq0$, it remains to prove $p\mid A_{22}$. In this case $da_{13}z\equiv A_{23}\bmod{p}$. Hence we have
\begin{equation*}
4a_{11}a_{22}a_{13}^2\equiv (a_{12}-dz)^2a_{13}^2\equiv(a_{12}a_{13}-A_{23})^2\bmod{p}.
\end{equation*}
Inserting the third equation of \eqref{eq:defA12A13A23} into above, we obtain $4a_{11}a_{22}a_{13}^2\equiv4a_{11}^2a_{23}^2\bmod{p}$, which is equivalent to $A_{11}a_{13}^2\equiv A_{22}a_{23}^2\bmod{p}$ by the definitions of $A_{11}$ and $A_{22}$. Assume by contradiction that $p\nmid A_{22}$, then $\legendre{A_{11}A_{22}}{p}_K\legendre{a_{13}^2}{p}_K=\legendre{a_{23}^2}{p}_K$. If we are in Case (1), then $-\legendre{a_{13}^2}{p}_K=\legendre{a_{23}^2}{p}_K$, so $a_{13}\equiv a_{23}\equiv0\bmod{p}$. If follows from this and the definitions of $A_{11}$ and $A_{22}$ that
$$\legendre{A_{11}A_{22}}{p}_K=\legendre{4a_{22}a_{33}\cdot4a_{11}a_{33}}{p}_K=\legendre{a_{11}a_{22}}{p}_K=1\text{ or }0,$$
which contradicts the assumption of Case (1). If we are in Cases (2) or (3), then $p\mid A_{11}$, and hence $p\mid a_{23}$ by $A_{11}a_{13}^2\equiv A_{22}a_{23}^2\bmod{p}$. Since $A_{11}=4a_{22}a_{33}-a_{23}^2$, we deduce $p\mid4a_{22}a_{33}$, a contradiction. Finally, if we are in Cases (4) or (5), then
\begin{align*}
A_{11}a_{13}^2\equiv A_{22}a_{23}^2\bmod{p}&\Longleftrightarrow A_{11}(4a_{11}a_{33}-A_{22})\equiv A_{22}(4a_{22}a_{33}-A_{11})\bmod{p}\\
&\Longleftrightarrow a_{11}A_{11}\equiv a_{22}A_{22}\bmod{p},
\end{align*}
which contradicts one of the assumptions of Cases (4) and (5). We have thus shown that in each of the cases (1)--(5), there is a contradiction. Therefore we must have $p\mid A_{22}$.

We now prove the last assertion. Setting $G=G_z$, $\mathbf{v}=\mathbf{v}_z$ in Lemma \ref{lemm:rGvSingularQuad}, we find that $g_{12}v_1-2g_{11}v_2\equiv-da_{13}z^2+A_{23}z\bmod{p}$. If $-da_{13}z^2+A_{23}z\not\equiv0\bmod{p}$, then $r_{G_z,\mathbf{v}_z}(-a_{33}z^2)=p$ by Lemma \ref{lemm:rGvSingularQuad}(1). On the other hand, if $-da_{13}z^2+A_{23}z\equiv0\bmod{p}$, then $p\mid A_{22}z^2$ by what we have just proved. By Lemma \ref{lemm:rGvSingularQuad}(2), we have
\begin{equation*}
r_{G_z,\mathbf{v}_z}(-a_{33}z^2)=p\cdot\left(1+\legendre{4g_{11}m+v_1^2}{p}_K\right)=p\cdot\left(1+\legendre{-A_{22}z^2}{p}_K\right)=p.
\end{equation*}
This concludes the whole proof.
\end{proof}

We are now in a position to achieve our goal:
\begin{proof}[Proof of Theorem \ref{thm:solnumbergMarkoff}]
The case $p=2$ has been proved in Lemma \ref{lemm:solnumbergMarkoffp2}. For $p>2$, let $R$ denote the number of solutions of \eqref{eq:gMarkoff} over $\numZ/p\numZ$. By Lemma \ref{lemm:a11pa22p}(3) we have
\begin{equation*}
R=\sum_{z\in S_1}r_{G_z,\mathbf{v}_z}(-a_{33}z^2)+\sum_{z\in S_2}r_{G_z,\mathbf{v}_z}(-a_{33}z^2)+\sum_{z\in S_3}r_{G_z,\mathbf{v}_z}(-a_{33}z^2),
\end{equation*}
where $S_1$, $S_2$, and $S_3$ are the sets defined in \eqref{eq:zS1}, \eqref{eq:zS2}, and \eqref{eq:zS3}, respectively. Set $r_z=r_{G_z,\mathbf{v}_z}(-a_{33}z^2)$ for simplicity. Inserting the formula given in Theorem \ref{thm:rGvm}(1) into $r_{G_z,\mathbf{v}_z}(-a_{33}z^2)$ for $z\in S_2\cup S_3$, we obtain
\begin{align}
R&=\sum_{z\in S_1}r_z+\sum_{z\in S_2}\left(p+\legendre{D_z}{p}_K(-1)^{\frac{1-p}{2}}(p-1)\right)+\sum_{z\in S_3}\left(p-\legendre{D_z}{p}_K(-1)^{\frac{1-p}{2}}\right)\notag\\
&=p^2-(-1)^{\frac{1-p}{2}}\sum_{z=0}^{p-1}\legendre{D_z}{p}_K+(-1)^{\frac{1-p}{2}}p\sum_{z\in S_2}\legendre{D_z}{p}_K+\sum_{z\in S_1}\left(r_z-p+(-1)^{\frac{1-p}{2}}\legendre{D_z}{p}_K\right)\notag\\
&=p^2-(-1)^{\frac{1-p}{2}}\sum_{z=0}^{p-1}\legendre{D_z}{p}_K+p\sum_{z\in S_2}\legendre{-D_z}{p}_K+\sum_{z\in S_1}\left(r_z-p\right),\label{eq:proofThm1.6Temp1}
\end{align}
where the last equality follows from the fact $D_z=N_z$ (see Lemma \ref{lemm:a11pa22p}(2)).
Note that
\begin{equation*}
\sum_{z=0}^{p-1}\legendre{D_z}{p}_K=\sum_{z=0}^{p-1}\legendre{4a_{11}a_{22}-(a_{12}-dz)^2}{p}_K=\sum_{z=0}^{p-1}\legendre{4a_{11}a_{22}-z^2}{p}_K=-\legendre{-1}{p}_K,
\end{equation*}
where the last equality follows from Lemma \ref{lemm:legendreSum}. Inserting this into \eqref{eq:proofThm1.6Temp1} gives
\begin{equation}
\label{eq:proofThm1.6Temp2}
R=p^2+1+p\sum_{z\in S_2}\legendre{-D_z}{p}_K+\sum_{z\in S_1}\left(r_z-p\right).
\end{equation}
The rest of the proof is split into Cases (1)--(15) introduced above, in each of which we shall evaluate the two sums in the right-hand side of \eqref{eq:proofThm1.6Temp2} in terms of $\legendre{-A_{jj}}{p}_K\cdot p$, $j=1,2,3$.

Note that in Cases (1)--(5), \eqref{eq:proofThm1.6Temp2} can be further simplified to
\begin{equation}
\label{eq:proofThm1.6Temp3}
R=p^2+1+p\sum_{z\in S_2}\legendre{-D_z}{p}_K
\end{equation}
by Lemma \ref{lemm:cases1-5rz}.

\textbf{Case (1).} We have $p\nmid D$ by Lemma \ref{lemm:Dmodp}. Comparing \eqref{eq:proofThm1.6Temp3} and Theorem \ref{thm:solnumbergMarkoff}(1), we need to prove
\begin{equation}
\label{eq:proofThm1.6case1Todo}
\sum_{z\in S_2}\legendre{-D_z}{p}_K=\legendre{-A_{11}}{p}_K+\legendre{-A_{22}}{p}_K+\legendre{-A_{33}}{p}_K.
\end{equation}
Since $\legendre{A_{11}A_{22}}{p}_K=-1$, \eqref{eq:S2congzneq0} is unsolvable, and hence
\begin{equation}
\label{eq:proofThm1.6Temp4}
S_2=\left\{z\in\{0,1,\dots,p-1\}\colon 4a_{11}a_{22}-(a_{12}-dz)^2\not\equiv0\bmod{p}\right\}\cap\{0\}.
\end{equation}
If $A_{33}\not\equiv0\bmod{p}$, then $S_2=\{0\}$, and hence $\sum_{z\in S_2}\legendre{-D_z}{p}_K=\legendre{-D_0}{p}_K=\legendre{-A_{33}}{p}_K$, from which \eqref{eq:proofThm1.6case1Todo} follows. If $A_{33}\equiv0\bmod{p}$, then $S_2=\emptyset$, in which case \eqref{eq:proofThm1.6case1Todo} also holds.

\textbf{Case (2).} We have $p\mid D$ by Lemma \ref{lemm:Dmodp}. The two assumptions of this case imply \eqref{eq:S2congzneq0} has exactly one solution $z\equiv0\bmod{p}$, and hence $S_2$ is again given by \eqref{eq:proofThm1.6Temp4}. If $A_{33}\equiv0\bmod{p}$, then at least two of $A_{11}, A_{22}, A_{33}$ are divisible by $p$, so we are in the case of Theorem\ref{thm:solnumbergMarkoff}(2), whose conclusion holds by \eqref{eq:proofThm1.6Temp3} since here $S_2=\emptyset$. On the other hand, suppose $A_{33}\not\equiv0\bmod{p}$. Then $S_2=\{0\}$ and hence $R=p^2+1+p\legendre{-A_{33}}{p}_K$. The situation $A_{11}\equiv A_{22}\equiv0\bmod{p}$ cannot happen, for if it happens, then $A_{13}\equiv A_{23}\equiv0\bmod{p}$ by the first equations of \eqref{eq:Adet1} and \eqref{eq:Adet2}, and hence $A_{33}\equiv0\bmod{p}$ by \eqref{eq:Dexpansion3}, which contradicts the assumption $A_{33}\not\equiv0\bmod{p}$. Therefore, at most one of $A_{11}, A_{22}, A_{33}$ is divisible by $p$, that is, we are in the case of Theorem\ref{thm:solnumbergMarkoff}(3). Without loss of generality, we assume that $A_{11}\equiv0\bmod{p}$, so $A_{22}\not\equiv0\bmod{p}$. By \eqref{eq:Adet1}, we find that $\legendre{-A_{22}}{p}_K=\legendre{-A_{33}}{p}_K$, and hence the conclusion of Theorem\ref{thm:solnumbergMarkoff}(3) holds with $j=2$ or $3$.

\textbf{Case (3).} We have $p\nmid D$ by Lemma \ref{lemm:Dmodp}. Comparing \eqref{eq:proofThm1.6Temp3} and Theorem \ref{thm:solnumbergMarkoff}(1), we need to prove \eqref{eq:proofThm1.6case1Todo}. Let $0\leq z_1\leq p-1$ be the unique solution of \eqref{eq:S2congzneq0} (the uniqueness follows from $\legendre{A_{11}A_{22}}{p}_K=0$). Then $z_1\not\equiv0\bmod{p}$ by the assumptions of Case (3), and hence
\begin{equation}
\label{eq:proofThm1.6Case3S2}
S_2=\left\{z\in\{0,1,\dots,p-1\}\colon 4a_{11}a_{22}-(a_{12}-dz)^2\not\equiv0\bmod{p}\right\}\cap\{0,z_1\}.
\end{equation}
It follows that
\begin{equation}
\label{eq:proofThm1.6Temp5}
\sum_{z\in S_2}\legendre{-D_z}{p}_K=\legendre{-D_0}{p}_K+\legendre{-D_{z_1}}{p}_K=\legendre{-A_{33}}{p}_K+\legendre{-D_{z_1}}{p}_K.
\end{equation}
By \eqref{eq:S2congzneq0}, $a_{13}a_{23}\equiv2a_{33}(a_{12}-dz_1)\bmod{p}$, and hence
\begin{equation}
\label{eq:proofThm1.6Temp6}
-4a_{33}^2D_{z_1}\equiv4a_{33}^2((a_{12}-dz_1)^2-4a_{11}a_{22})\equiv a_{13}^2a_{23}^2-16a_{11}a_{22}a_{33}^2\bmod{p}.
\end{equation}
Since $\legendre{A_{11}A_{22}}{p}_K=0$, we assume, without loss of generality, $A_{11}\equiv0\bmod{p}$, that is, $a_{23}^2\equiv4a_{22}a_{33}\bmod{p}$. Then
\begin{equation}
\label{eq:proofThm1.6Temp7}
a_{13}^2a_{23}^2-16a_{11}a_{22}a_{33}^2\equiv a_{13}^2\cdot4a_{22}a_{33}-16a_{11}a_{22}a_{33}^2\equiv-a_{23}^2A_{22}\bmod{p}.
\end{equation}
Combining \eqref{eq:proofThm1.6Temp6} and \eqref{eq:proofThm1.6Temp7} we find that
\begin{equation}
\label{eq:proofThm1.6Temp8}
\legendre{-D_{z_1}}{p}_K=\legendre{-4a_{33}^2D_{z_1}}{p}_K=\legendre{a_{23}^2}{p}_K\legendre{-A_{22}}{p}_K=\legendre{-A_{22}}{p}_K.
\end{equation}
The last equality follows from the fact $a_{23}\not\equiv0\bmod{p}$, since $p\mid A_{11}$ but $p\nmid4a_{22}a_{33}$. Inserting \eqref{eq:proofThm1.6Temp8} into \eqref{eq:proofThm1.6Temp5}, and noting that $\legendre{-A_{11}}{p}_K=0$, we arrive at \eqref{eq:proofThm1.6case1Todo}, as desired.

\textbf{Case (4).} We have $p\mid D$ by Lemma \ref{lemm:Dmodp}. Since $\legendre{A_{11}A_{22}}{p}_K=1$, \eqref{eq:S2congzneq0} has two different solutions $z\equiv z_0,z_1\bmod{p}$. (Assume that $0\leq z_0<z_1\leq p-1$.) Since $A_{11}A_{22}-A_{12}^2\equiv0\bmod{p}$, we have $z_0=0$. Thus, $S_2$ is still given by \eqref{eq:proofThm1.6Case3S2}, and hence \eqref{eq:proofThm1.6Temp5} holds as well in this case. Since $z=z_1$ satisfies \eqref{eq:S2congzneq0}, and since $A_{11}A_{22}\equiv A_{12}^2\bmod{p}$, we have $a_{13}a_{23}-2a_{33}(a_{12}-dz_1)\equiv \pm A_{12}\bmod{p}$. The plus sign cannot occur, since $z_1\not\equiv0\bmod{p}$. Therefore, $a_{33}dz_1\equiv-A_{12}\bmod{p}$, and hence
\begin{equation}
\label{eq:proofThm1.6Temp8.5}
\legendre{-D_{z_1}}{p}_K=\legendre{-a_{33}^2D_{z_1}}{p}_K=\legendre{(a_{13}a_{23}-a_{12}a_{33})^2-4a_{11}a_{22}a_{33}^2}{p}_K.
\end{equation}
We have
\begin{align}
4A_{33}((a_{13}a_{23}-a_{12}a_{33})^2-4a_{11}a_{22}a_{33}^2)&=4A_{33}(-A_{33}a_{33}^2+a_{13}a_{23}A_{12})\notag\\
&=-4(a_{33}A_{33})^2+4a_{13}a_{23}(A_{13}A_{23}-Da_{12})\notag\\
&\equiv-(2a_{33}A_{33})^2+4(a_{13}A_{13})(a_{23}A_{23})\bmod{p}\label{eq:proofThm1.6Temp9},
\end{align}
where we have used the second equation of \eqref{eq:Adet1} and the fact $p\mid D$ in this case. By \eqref{eq:Dexpansion3} and $p\mid D$ we find that $(2a_{33}A_{33})^2\equiv(a_{13}A_{13}+a_{23}A_{23})^2$, from which and \eqref{eq:proofThm1.6Temp9} we deduce
\begin{equation*}
4A_{33}((a_{13}a_{23}-a_{12}a_{33})^2-4a_{11}a_{22}a_{33}^2)\equiv-(a_{13}A_{13}-a_{23}A_{23})^2\bmod{p}.
\end{equation*}
It follows that
\begin{equation}
\label{eq:proofThm1.6Temp10}
\legendre{-A_{33}}{p}_K\legendre{(a_{13}a_{23}-a_{12}a_{33})^2-4a_{11}a_{22}a_{33}^2}{p}_K=\legendre{(a_{13}A_{13}-a_{23}A_{23})^2}{p}_K.
\end{equation}
Taking the difference between \eqref{eq:Dexpansion1} and \eqref{eq:Dexpansion2}, and noting that $a_{11}A_{11}\not\equiv a_{22}A_{22}\bmod{p}$ (one of the assumptions of Case (4)), we obtain $a_{13}A_{13}-a_{23}A_{23}\not\equiv0\bmod{p}$. This, together with \eqref{eq:proofThm1.6Temp10}, implies $\legendre{-A_{33}}{p}_K=\legendre{(a_{13}a_{23}-a_{12}a_{33})^2-4a_{11}a_{22}a_{33}^2}{p}_K\neq0$. Combining this, \eqref{eq:proofThm1.6Temp8.5}, and \eqref{eq:proofThm1.6Temp5} (which we have shown is true as well in this case), we find that
\begin{equation}
\label{eq:proofThm1.6Temp11}
\sum_{z\in S_2}\legendre{-D_z}{p}_K=2\cdot\legendre{-A_{33}}{p}_K.
\end{equation}
By the assumption $\legendre{A_{11}A_{22}}{p}_K=1$, we know that $\legendre{-A_{11}}{p}_K=\legendre{-A_{22}}{p}_K$. Since $\legendre{-A_{33}}{p}_K\neq0$ and $p\mid D$, we conclude from the first equation of \eqref{eq:Adet1} that $\legendre{A_{22}A_{33}}{p}_K=1$, so $\legendre{A_{jj}}{p}_K$ are the same for $j=1,2,3$. Thus, we are in the case of Theorem \ref{thm:solnumbergMarkoff}(3). The desired conclusion follows from inserting \eqref{eq:proofThm1.6Temp11} into \eqref{eq:proofThm1.6Temp3}, and taking into account the fact $\legendre{A_{jj}}{p}_K$ are the same.

\textbf{Case (5).} We have $p\nmid D$ by Lemma \ref{lemm:Dmodp}. As in Case (4), \eqref{eq:S2congzneq0} has two different solutions $z\equiv z_0,z_1\bmod{p}$. (Again we assume that $0\leq z_0<z_1\leq p-1$.) The difference is that in the present case we have $z_0,z_1\not\equiv0\bmod{p}$ since $A_{11}A_{22}-A_{12}^2\not\equiv0\bmod{p}$. Therefore
\begin{equation}
\label{eq:proofThm1.6Case5S2}
S_2=\left\{z\in\{0,1,\dots,p-1\}\colon 4a_{11}a_{22}-(a_{12}-dz)^2\not\equiv0\bmod{p}\right\}\cap\{0,z_0,z_1\},
\end{equation}
and hence
\begin{equation}
\label{eq:proofThm1.6Temp12}
\sum_{z\in S_2}\legendre{-D_z}{p}_K=\legendre{-A_{33}}{p}_K+\legendre{-D_{z_0}}{p}_K+\legendre{-D_{z_1}}{p}_K.
\end{equation}
We shall prove
\begin{equation}
\label{eq:proofThm1.6Temp13}
\legendre{-D_{z_0}}{p}_K=\legendre{-D_{z_1}}{p}_K=\legendre{-A_{11}}{p}_K=\legendre{-A_{22}}{p}_K\neq0.
\end{equation}
Set
\begin{equation}
\label{eq:defZ0Z1}
Z_j=4a_{33}^2(a_{12}-dz_j)^2-16a_{11}a_{22}a_{33}^2\equiv-4a_{33}^2D_{z_j}\bmod{p},\quad j=0,1.
\end{equation}
Since $z\equiv z_0,z_1\bmod{p}$ are solutions of \eqref{eq:S2congzneq0}, we have
\begin{equation}
\label{eq:proofThm1.6Temp14}
Z_j\equiv-2a_{13}a_{23}(a_{13}a_{23}-2a_{33}(a_{12}-dz_j))+A_{11}A_{22}+a_{13}^2a_{23}^2-16a_{11}a_{22}a_{33}^2,\quad j=0,1.
\end{equation}
Multiplying $Z_0$ and $Z_1$, and using the fact $a_{13}a_{23}-2a_{33}(a_{12}-dz_0)\equiv-(a_{13}a_{23}-2a_{33}(a_{12}-dz_1))\bmod{p}$, we have
\begin{equation}
\label{eq:proofThm1.6Temp14.5}
Z_0Z_1\equiv16a_{33}^2(a_{13}^2a_{22}-a_{23}^2a_{11})^2\bmod{p}.
\end{equation}
By the assumption $a_{11}A_{11}\not\equiv a_{22}A_{22}\bmod{p}$, we have $p\nmid a_{13}^2a_{22}-a_{23}^2a_{11}$. It follows that $\legendre{Z_0}{p}_K=\legendre{Z_1}{p}_K\neq0$, and hence $\legendre{-D_{z_0}}{p}_K=\legendre{-D_{z_1}}{p}_K\neq0$. The assumption $\legendre{A_{11}A_{22}}{p}_K=1$ implies $\legendre{-A_{11}}{p}_K=\legendre{-A_{22}}{p}_K\neq0$. Therefore, to obtain \eqref{eq:proofThm1.6Temp13}, it remains to prove $\legendre{Z_1}{p}_K=\legendre{-A_{22}}{p}_K$. We have
\begin{align*}
&(a_{23}A_{22}+a_{13}(a_{13}a_{23}-2a_{33}(a_{12}-dz_1)))^2\\
\equiv&a_{23}^2A_{22}^2+a_{13}^2A_{11}A_{22}+2a_{13}a_{23}A_{22}(a_{13}a_{23}-2a_{33}(a_{12}-dz_1))\\
\equiv&-A_{22}Z_1\bmod{p},
\end{align*}
where the first congruence follows from the fact $z\equiv z_1\bmod{p}$ is a solution of \eqref{eq:S2congzneq0}, and the second congruence follows from \eqref{eq:proofThm1.6Temp14}. It follows that $\legendre{Z_1}{p}_K=\legendre{-A_{22}}{p}_K$. Therefore, \eqref{eq:proofThm1.6Temp13} holds. Since $p\nmid D$, we are in the case of Theorem \ref{thm:solnumbergMarkoff}(1). The conclusion follows by combining \eqref{eq:proofThm1.6Temp3}, \eqref{eq:proofThm1.6Temp12}, and \eqref{eq:proofThm1.6Temp13}.

In the remaining cases (Cases (6)--(15)), we should use \eqref{eq:proofThm1.6Temp2} instead of \eqref{eq:proofThm1.6Temp3}.

\textbf{Case (6).} We have $p\mid D$ by Lemma \ref{lemm:Dmodp}. Repeating\footnote{This is reasonable since the beginning part of the deduction of Case (4) only uses the assumptions $\legendre{A_{11}A_{22}}{p}_K=1$ and $A_{11}A_{22}-A_{12}^2\equiv0\bmod{p}$, which are also assumptions of Case (6).} the beginning part of the deduction of Case (4), we have \eqref{eq:proofThm1.6Temp5} and \eqref{eq:proofThm1.6Temp8.5}--\eqref{eq:proofThm1.6Temp10}. In particular,
\begin{equation}
\label{eq:proofThm1.6Temp15}
\sum_{z\in S_2}\legendre{-D_z}{p}_K=\legendre{-A_{33}}{p}_K+\legendre{(a_{13}a_{23}-a_{12}a_{33})^2-4a_{11}a_{22}a_{33}^2}{p}_K.
\end{equation}
Since $A_{23}\not\equiv0\bmod{p}$ (one of the assumptions) and $D\equiv0\bmod{p}$, we have $A_{33}\not\equiv0\bmod{p}$ by the first equation of \eqref{eq:Adet1}. Since $a_{11}A_{11}\equiv a_{22}A_{22}\bmod{p}$ (one of the assumptions), the right-hand side of \eqref{eq:proofThm1.6Temp10} is equal to $0$, and hence $\legendre{(a_{13}a_{23}-a_{12}a_{33})^2-4a_{11}a_{22}a_{33}^2}{p}_K=0$. Therefore, \eqref{eq:proofThm1.6Temp15} becomes
\begin{equation}
\label{eq:proofThm1.6Temp16}
\sum_{z\in S_2}\legendre{-D_z}{p}_K=\legendre{-A_{33}}{p}_K.
\end{equation}
Now we consider the sum $\sum_{z\in S_1}\left(r_z-p\right)$ in \eqref{eq:proofThm1.6Temp2}. We must have $S_1\neq\emptyset$, for $\legendre{A_{11}A_{22}}{p}_K=1$ and $a_{11}A_{11}\equiv a_{22}A_{22}\bmod{p}$ by the assumptions. Let $z\in S_1$. We have just proved $A_{33}\not\equiv0\bmod{p}$, so $z\not\equiv0\bmod{p}$. Setting $G=G_z$, $\mathbf{v}=\mathbf{v}_z$ in Lemma \ref{lemm:rGvSingularQuad}, we find that $g_{12}v_1-2g_{11}v_2\equiv-da_{13}z^2+A_{23}z\bmod{p}$. If $p\nmid-da_{13}z^2+A_{23}z\bmod{p}$, then $r_z-p=0$ by Lemma \ref{lemm:rGvSingularQuad}(1). Otherwise, if $p\mid-da_{13}z^2+A_{23}z\bmod{p}$, then $da_{13}z\equiv A_{23}\bmod{p}$. Since $a_{13}\not\equiv0,\, A_{23}\not\equiv0\bmod{p}$, we have $r_z-p=p\legendre{-A_{22}}{p}_K$ by Lemma \ref{lemm:rGvSingularQuad}(2). It follows that
\begin{equation}
\label{eq:proofThm1.6Temp17}
\sum_{z\in S_1}\left(r_z-p\right)=p\cdot\legendre{-A_{22}}{p}_K.
\end{equation}
Inserting \eqref{eq:proofThm1.6Temp16} and \eqref{eq:proofThm1.6Temp17} into \eqref{eq:proofThm1.6Temp2}, we obtain the conclusion of Theorem \ref{thm:solnumbergMarkoff}(3) with $j=1$. However, the conclusion with $j=2,3$ holds as well, since in Case (6) $\legendre{-A_{jj}}{p}_K$ are the same for all $j$. (For instance, $\legendre{-A_{22}}{p}_K=\legendre{-A_{33}}{p}_K$ follows from the first equation of \eqref{eq:Adet1}.)

\textbf{Case (7).} This never occurs by Lemma \ref{lemm:Dmodp}.

\textbf{Case (8).} We have $p\mid D$ by Lemma \ref{lemm:Dmodp}. The only difference between this case and Case (6) is the assumption on $A_{23}$. Consequently, some part of the deduction of Case (6) is valid here: \eqref{eq:proofThm1.6Temp15} holds, and $S_1\neq\emptyset$. We have $A_{33}\equiv0\bmod{p}$ by the first equation of \eqref{eq:Adet1}. Furthermore, we have $A_{13}\equiv0$ and $A_{12}\not\equiv0$ mod $p$, by the first equations of \eqref{eq:Adet2} and \eqref{eq:Adet3}. Moreover, it follows from \eqref{eq:Dexpansion2} that $a_{12}\not\equiv0\bmod{p}$. Now we have
\begin{align*}
\legendre{(a_{13}a_{23}-a_{12}a_{33})^2-4a_{11}a_{22}a_{33}^2}{p}_K&=\legendre{-a_{33}^2A_{33}+a_{13}a_{23}A_{12}}{p}_K\\
&=\legendre{a_{13}a_{23}A_{12}}{p}_K=\legendre{a_{12}a_{13}a_{23}\cdot a_{12}A_{12}}{p}_K\\
&=\legendre{a_{12}a_{13}a_{23}\cdot (-2a_{11}A_{11})}{p}_K\quad\text{ (by \eqref{eq:Dexpansion1})}\\
&=\legendre{-A_{11}(a_{12}a_{13})^2}{p}_K.
\end{align*}
Since $a_{13}\not\equiv0\bmod{p}$ (the first time we use this assumption in this case!), we have
\begin{equation*}
\legendre{(a_{13}a_{23}-a_{12}a_{33})^2-4a_{11}a_{22}a_{33}^2}{p}_K=\legendre{-A_{11}(a_{12}a_{13})^2}{p}_K=\legendre{-A_{11}}{p}_K.
\end{equation*}
Inserting this into \eqref{eq:proofThm1.6Temp15}, we obtain
\begin{equation}
\label{eq:proofThm1.6Temp18}
\sum_{z\in S_2}\legendre{-D_z}{p}_K=\legendre{-A_{11}}{p}_K.
\end{equation}
For the sum $\sum_{z\in S_1}\left(r_z-p\right)$, we use a reasoning similar to Case (6), whose details are omitted here. We obtain that
\begin{equation}
\label{eq:proofThm1.6Temp19}
\sum_{z\in S_1}\left(r_z-p\right)=0.
\end{equation}
Inserting \eqref{eq:proofThm1.6Temp18} and \eqref{eq:proofThm1.6Temp19} into \eqref{eq:proofThm1.6Temp2}, and noting that $\legendre{-A_{33}}{p}_K=0$, we obtain the conclusion of Theorem \ref{thm:solnumbergMarkoff}(3) with $j=2$. However, the conclusion with $j=1$ also holds since $\legendre{A_{11}A_{22}}{p}_K=1$.

\textbf{Case (9).} We have $p\mid D$ by Lemma \ref{lemm:Dmodp}. Most of the assumptions are the same to Case (8), so we can repeat the deduction there until we need the assumption on $a_{13}$. Thus,
\begin{align}
\sum_{z\in S_2}\legendre{-D_z}{p}_K&=\legendre{-A_{33}}{p}_K+\legendre{(a_{13}a_{23}-a_{12}a_{33})^2-4a_{11}a_{22}a_{33}^2}{p}_K\notag\\
&=\legendre{(a_{13}a_{23}-a_{12}a_{33})^2-4a_{11}a_{22}a_{33}^2}{p}_K=\legendre{a_{13}a_{23}A_{12}}{p}_K\notag\\
&=0.\label{eq:proofThm1.6Temp20}
\end{align}
In the last equality, we have used the assumption $a_{13}\equiv0\bmod{p}$. Now we consider the sum $\sum_{z\in S_1}\left(r_z-p\right)$ in \eqref{eq:proofThm1.6Temp2}. We must have $\abs{S_1}=2$, for $\legendre{A_{11}A_{22}}{p}_K=1$ and $a_{11}A_{11}\equiv a_{22}A_{22}\bmod{p}$ by the assumptions. Since $A_{33}\equiv0\bmod{p}$ (one of the assumptions), we have $0\in S_1$. Assume that $S_1=\{0,z_0\}$, where $z_0\not\equiv0\bmod{p}$. Then $a_{12}-dz_0\equiv-(a_{12}-d\cdot0)\bmod{p}$, and hence $dz_0\equiv2a_{12}\bmod{p}$. Setting $G=G_{z_0}$, $\mathbf{v}=\mathbf{v}_{z_0}$ in Lemma \ref{lemm:rGvSingularQuad}, we find that
$$g_{12}v_1-2g_{11}v_2\equiv-da_{13}z_0^2+A_{23}z_0\equiv 0\bmod{p},$$
where the last congruence follows from the assumptions $a_{13}\equiv A_{23}\equiv0\bmod{p}$. Therefore, we deduce from Lemma \ref{lemm:rGvSingularQuad}(2) that $r_{z_0}-p=p\legendre{-A_{22}}{p}_K$. Similarly, we have $r_{0}-p=p\legendre{0}{p}_K=0$. Thus,
\begin{equation}
\label{eq:proofThm1.6Temp21}
\sum_{z\in S_1}\left(r_z-p\right)=(r_0-p)+(r_{z_0}-p)=p\cdot\legendre{-A_{22}}{p}_K.
\end{equation}
Inserting \eqref{eq:proofThm1.6Temp20} and \eqref{eq:proofThm1.6Temp21} into \eqref{eq:proofThm1.6Temp2}, and noting that $\legendre{A_{11}A_{22}}{p}_K=1$ and $\legendre{A_{33}}{p}_K=0$, we obtain the conclusion of Theorem \ref{thm:solnumbergMarkoff}(3).

\textbf{Case (10).} This never occurs by Lemma \ref{lemm:Dmodp}.

\textbf{Case (11).} We have $p\nmid D$ by Lemma \ref{lemm:Dmodp}. The congruence \eqref{eq:S2congzneq0} has two different solutions $z\equiv z_0,z_1\bmod{p}$ and we assume that $0\leq z_0<z_1\leq p-1$. Since $A_{11}A_{22}-A_{12}^2\not\equiv0\bmod{p}$ we have $z_0,z_1\not\equiv0\bmod{p}$. Therefore, \eqref{eq:proofThm1.6Case5S2}, and hence \eqref{eq:proofThm1.6Temp12} hold. We define $Z_0$, $Z_1$ by \eqref{eq:defZ0Z1}. Then \eqref{eq:proofThm1.6Temp14} and \eqref{eq:proofThm1.6Temp14.5} hold in the present case as well. Since $a_{11}A_{11}\equiv a_{22}A_{22}\bmod{p}$ (one of the assumptions), we have $a_{13}^2a_{22}-a_{23}^2a_{11}\equiv0\bmod{p}$. This, together with \eqref{eq:proofThm1.6Temp14.5}, implies that $Z_0Z_1\equiv0\bmod{p}$. Without loss of generality, we assume that $Z_1\equiv0\bmod{p}$. Set $t=a_{13}a_{23}-2a_{33}(a_{12}-dz_0)$, then $t^2\equiv A_{11}A_{22}\bmod{p}$ and $a_{13}a_{23}-2a_{33}(a_{12}-dz_1)\equiv-t\bmod{p}$. By \eqref{eq:proofThm1.6Temp14}, we have $Z_1-Z_0\equiv4a_{13}a_{23}t\bmod{p}$. However, (congruences below are modulo $p$)
\begin{align*}
2a_{13}a_{23}t&\equiv-A_{11}A_{22}-a_{13}^2a_{23}^2+16a_{11}a_{22}a_{33}^2\quad(\text{since }Z_1\equiv0 \text{ and \eqref{eq:proofThm1.6Temp14}})\\
&=4a_{33}(a_{11}A_{11}+a_{22}A_{22})-2A_{11}A_{22}\\
&\equiv8a_{33}a_{11}A_{11}-2A_{11}A_{22}\quad(\text{since }a_{11}A_{11}\equiv a_{22}A_{22})\\
&=2a_{13}^2A_{11}.
\end{align*}
It follows that
\begin{equation*}
\legendre{Z_0}{p}_K=\legendre{-(Z_1-Z_0)}{p}_K=\legendre{-4a_{13}^2A_{11}}{p}_K=\legendre{-A_{11}}{p}_K,
\end{equation*}
where we have used (the first time in this case) the assumption $a_{13}\not\equiv0\bmod{p}$. Therefore,
\begin{align}
\sum_{z\in S_2}\legendre{-D_z}{p}_K&=\legendre{-A_{33}}{p}_K+\legendre{-D_{z_0}}{p}_K+\legendre{-D_{z_1}}{p}_K\notag\\
&=\legendre{-A_{33}}{p}_K+\legendre{Z_0}{p}_K+\legendre{Z_1}{p}_K\notag\\
&=\legendre{-A_{11}}{p}_K+\legendre{-A_{33}}{p}_K.\label{eq:proofThm1.6Temp22}
\end{align}
For $\sum_{z\in S_1}\left(r_z-p\right)$, \eqref{eq:proofThm1.6Temp17} holds as well for the present case. This can be proved using an argument similar to that of Case (6), and we omit it. (Note that $A_{33}$ can be either zero or nonzero modulo $p$ in this case.) Inserting \eqref{eq:proofThm1.6Temp22} and \eqref{eq:proofThm1.6Temp17} into \eqref{eq:proofThm1.6Temp2}, we obtain the conclusion of Theorem \ref{thm:solnumbergMarkoff}(1) in the present case.

\textbf{Cases (12)--(14).} These never occur by Lemma \ref{lemm:Dmodp}.

\textbf{Case (15).} We have $p\nmid D$ by Lemma \ref{lemm:Dmodp}. The first three assumptions are the same to Case (11), so we can repeat the deduction there until we need the assumption on $a_{13}$. In this way we obtain $Z_0Z_1\equiv0\bmod{p}$ and
\begin{equation*}
Z_1-Z_0\equiv4a_{13}a_{23}t\equiv0\bmod{p}.\quad(\text{since }a_{13}\equiv0\bmod{p})
\end{equation*}
Therefore,
\begin{align}
\sum_{z\in S_2}\legendre{-D_z}{p}_K&=\legendre{-A_{33}}{p}_K+\legendre{-D_{z_0}}{p}_K+\legendre{-D_{z_1}}{p}_K\notag\\
&=\legendre{-A_{33}}{p}_K+\legendre{Z_0}{p}_K+\legendre{Z_1}{p}_K\notag\\
&=\legendre{-A_{33}}{p}_K.\label{eq:proofThm1.6Temp23}
\end{align}
Now we consider the sum $\sum_{z\in S_1}\left(r_z-p\right)$ in \eqref{eq:proofThm1.6Temp2}. We must have $\abs{S_1}=2$, for $\legendre{A_{11}A_{22}}{p}_K=1$ and $a_{11}A_{11}\equiv a_{22}A_{22}\bmod{p}$ by the assumptions. As $A_{33}\not\equiv0\bmod{p}$ (one of the assumptions), we have $0\not\in S_1$. Assume that $S_1=\{z_2,z_3\}$, where $z_2,z_3\not\equiv0\bmod{p}$. Setting $G=G_{z_2}$, $\mathbf{v}=\mathbf{v}_{z_2}$ in Lemma \ref{lemm:rGvSingularQuad}, we find that
$$g_{12}v_1-2g_{11}v_2\equiv-da_{13}z_2^2+A_{23}z_2\equiv 0\bmod{p},$$
where the last congruence follows from the assumptions $a_{13}\equiv A_{23}\equiv0\bmod{p}$. Therefore, we deduce from Lemma \ref{lemm:rGvSingularQuad}(2) that $r_{z_2}-p=p\legendre{-A_{22}z_2^2}{p}_K=p\legendre{-A_{22}}{p}_K$. Similarly, we have $r_{z_3}-p=\legendre{-A_{22}}{p}_K$. However, $\legendre{-A_{11}}{p}_K=\legendre{-A_{22}}{p}_K$ by the first assumption, so
\begin{equation*}
\sum_{z\in S_1}\left(r_z-p\right)=(r_{z_2}-p)+(r_{z_3}-p)=p\cdot\left(\legendre{-A_{11}}{p}_K+\legendre{-A_{22}}{p}_K\right).
\end{equation*}
Inserting this and \eqref{eq:proofThm1.6Temp23} into \eqref{eq:proofThm1.6Temp2}, we obtain the conclusion of Theorem \ref{thm:solnumbergMarkoff}(1), which concludes the whole proof.
\end{proof}

\begin{rema}
Of course, one can produce an independent proof using the idea of \cite[\S V.1]{Bar91}. In our opinion, such a proof might not be simpler than ours, since it as well need to compute four $r_{G,\mathbf{v}}(m)$, where $\det(G)$ may be or may not be congruence to $0$ modulo $p$. Thus, one still need to split the proof into many cases.
\end{rema}

\section{Miscellaneous observations and open questions}
\label{sec:miscellaneous_observations_and_open_questions}
\textbf{Formulas for $\mathfrak{G}_{G}(a/c;\mathbf{w},\mathbf{x})$ for arbitrary $c$.} Theorem \ref{thm:main} gives formulas in the case $\gcd(N,c)=1$ or $N$. As we have mentioned, by combining Lemma \ref{lemm:GaussSumWeilRepr} and \cite[Remark 6.8]{Str13}, one can deduce a formula for all positive integers $c$. However, this formula involves quantities like $\xi(a,c)$ (cf. \cite[Definition 6.1]{Str13}), which requires firstly decomposing $G^{-1}\numZ^n/\numZ^n$ into Jordan components. It is natural further work to derive a formula for all $c$ that does not involve any local data. Such formulas exist for $n=1$ (Gauss) and $n=2$ (Skoruppa and Zagier \cite[Theorem 3]{SZ89}).

\textbf{On the positive definiteness restriction.} Theorem \ref{thm:main} requires $G$ to be positive definite. We have removed this restriction in Theorem \ref{thm:mainIndefinite} for $\gcd(N,c)=1$. It seems that we can merely assume $G$ is nonsingular in all results (except Theorem \ref{thm:thetaTransformation}). However, in some places in these generalized formulas, the rank $n$ should be replaced by the signature; see e.g. Remark \ref{rema:MilgramExtension2indef}. (Although in Theorem \ref{thm:mainIndefinite}, the formula remains unchanged compared to the positive definite case.) It is our further work to remove this restriction.

\textbf{Non-integral-parametric sums.} Theorem \ref{thm:main} is concerned with $\mathfrak{G}_{G}(a/c;\mathbf{w},\mathbf{x})$ of integral-parametric (see Definition \ref{deff:intPara}). We do not know any explicit formula if it is non-integral-parametric. Here we pose an open question: derive an explicit formula for $\mathfrak{G}_{G}(a/c;\mathbf{w},\mathbf{x})$ where $G\cdot\mathbf{w}\not\in\numZ^n$ or $a\cdot\mathbf{x}\not\in\numZ^n$.

\textbf{Gauss sums in finite fields.} As we have shown in \eqref{eq:GaussSumFq}, Gauss sums in finite fields can be reduced to $\mathfrak{G}_{G}(a/c;\mathbf{w},\mathbf{x})$ of integral-parametric. It is an interesting problem to derive, using Theorem \ref{thm:main}, explicit formulas for the following sum
\begin{equation*}
\sum_{x\in \mathbb{F}_q}\etp{\frac{\tr(\beta_1 x^2+\beta_2x)}{p}}.\qquad (\beta_1,\beta_2\in\mathbb{F}_q)
\end{equation*}
This has potential applications to the subject of, e.g., determining the zeta functions of varieties of characteristic $p$.

\textbf{Hecke Gauss sums.} We have discussed this topic in \S\ref{subsec:hecke_gauss_sums} and have shown that Hecke Gauss sums can always be reduced to $\mathfrak{G}_{G}(a/c;\mathbf{w},\mathbf{x})$ of integral-parametric (see Proposition \ref{prop:HeckeGaussToMatrixGauss}).Then we discussed the cases of quadratic fields and cyclotomic fields in more details. A natural direction of further work is to deduce more explicit formulas for Hecke Gauss sums in other number fields.

\textbf{Applications of the duality theorem.} Theorem \ref{thm:GaussSubsumExplicit} is an application of Theorem \ref{thm:duality1}. A problem: find an explicit formula for
\begin{equation*}
\sum_{\mathbf{v}\in\numZ^n/c\numZ^n,\,\mathbf{v}\perp\mathbf{h_1},\mathbf{h_2}}\etp{\frac{a}{c}\cdot\left(\frac{1}{2}(\mathbf{v}+\mathbf{w})^\tp\cdot G\cdot (\mathbf{v}+\mathbf{w})\right)}.
\end{equation*}
This could be done as well by applying Theorem \ref{thm:duality1}.

\textbf{Affine quadratic hypersurfaces.} A problem: generalize Theorem \ref{thm:rGvm} to the modulus $p^r$. This might be done by applying Theorem \ref{thm:affineQuadraticmodulusc} with $c$ set to $p^r$. Note that Li and Ouyang \cite{LO18} have achieved this, among other beautiful results, for diagonal $G$. Another related problem is to generalize Theorem \ref{thm:rGvm} to the corresponding Diophantine equation over  $\mathbb{F}_{p^r}$.

\textbf{Generalized Markoff equations over $\mathbb{F}_q$.} Theorem \ref{thm:solnumbergMarkoff} is concerned with the equation \eqref{eq:gMarkoff} over $\mathbb{F}_p$. It might be more involved to prove a formula for the number of solutions of \eqref{eq:gMarkoff} over $\mathbb{F}_q$, where $q=p^r$, a prime power. (But for a fixed equation with $a_{ij}$ and $d$ known, this can be done directly by the method of Baragar \cite[\S V.1]{Bar91}.) Motivated by numerical experiments, we pose the following conjecture.
\begin{conj}
Suppose that $p$ is a prime with $p\nmid a_{11}a_{22}a_{33}d$ and $r$ is a positive integer. Set $q=p^r$.

(1) If $p\nmid D$, then $p>2$, and the number of solutions of \eqref{eq:gMarkoff}, where $x,y,z\in\mathbb{F}_q$, equals
\begin{equation*}
q^2+\left(\legendre{-A_{11}}{q}_K+\legendre{-A_{22}}{q}_K+\legendre{-A_{33}}{q}_K\right)\cdot q+1.
\end{equation*}

(2) If $2<p\mid D$ and at least two of $A_{11}, A_{22}, A_{33}$ are divisible by $p$, then the solution number over $\mathbb{F}_q$ is
\begin{equation*}
q^2+1.
\end{equation*}

(3) If $2<p\mid D$ and at most one of $A_{11}, A_{22}, A_{33}$ is divisible by $p$, then the solution number over $\mathbb{F}_q$ is
\begin{equation*}
q^2+\left(\legendre{-A_{11}}{q}_K+\legendre{-A_{22}}{q}_K+\legendre{-A_{33}}{q}_K-\legendre{-A_{jj}}{q}_K\right)\cdot q+1,
\end{equation*}
where $1\leq j\leq3$ is any index such that $p\nmid A_{jj}$.

(4) If $2=p\mid D$, then the solution number over $\mathbb{F}_q$ is
\begin{equation*}
\begin{dcases}
q^2+1 &\text{ if }a_{12}\equiv a_{13}\equiv a_{23}\equiv0\bmod{2}\\
q^2+(-1)^r2q+1 &\text{ if }a_{12}\equiv a_{13}\equiv a_{23}\equiv1\bmod{2}\\
q^2+(-1)^rq+1 &\text{ otherwise.}
\end{dcases}
\end{equation*}
\end{conj}

\bibliographystyle{plain}
\bibliography{ref}

@book {BEW98,
    AUTHOR = {Berndt, Bruce C. and Evans, Ronald J. and Williams, Kenneth
              S.},
     TITLE = {Gauss and {J}acobi sums},
    SERIES = {Canadian Mathematical Society Series of Monographs and
              Advanced Texts},
      NOTE = {A Wiley-Interscience Publication},
 PUBLISHER = {John Wiley \& Sons, Inc., New York},
      YEAR = {1998},
     PAGES = {xii+583},
      ISBN = {0-471-12807-4},
   MRCLASS = {11L05 (11A15 11L10 11T22 11T24)},
  MRNUMBER = {1625181},
MRREVIEWER = {Charles\ Helou},
}

@book {IR90,
    AUTHOR = {Ireland, Kenneth and Rosen, Michael},
     TITLE = {A classical introduction to modern number theory},
    SERIES = {Graduate Texts in Mathematics},
    VOLUME = {84},
   EDITION = {Second},
 PUBLISHER = {Springer-Verlag, New York},
      YEAR = {1990},
     PAGES = {xiv+389},
      ISBN = {0-387-97329-X},
   MRCLASS = {11-01 (11-02)},
  MRNUMBER = {1070716},
MRREVIEWER = {Glenn\ Stevens},
       DOI = {10.1007/978-1-4757-2103-4},
       URL = {https://doi.org/10.1007/978-1-4757-2103-4},
}

@article {Kur04,
    AUTHOR = {Kuroda, Mitsuru},
     TITLE = {Quadratic gauss sums on matrices},
   JOURNAL = {Linear Algebra Appl.},
  FJOURNAL = {Linear Algebra and its Applications},
    VOLUME = {384},
      YEAR = {2004},
     PAGES = {187--198},
      ISSN = {0024-3795,1873-1856},
   MRCLASS = {11T24 (11L03)},
  MRNUMBER = {2055351},
MRREVIEWER = {Arne\ Winterhof},
       DOI = {10.1016/j.laa.2004.01.009},
       URL = {https://doi.org/10.1016/j.laa.2004.01.009},
}

@article {BS10,
    AUTHOR = {Boylan, Hatice and Skoruppa, Nils-Peter},
     TITLE = {Explicit formulas for {H}ecke {G}auss sums in quadratic number
              fields},
   JOURNAL = {Abh. Math. Semin. Univ. Hambg.},
  FJOURNAL = {Abhandlungen aus dem Mathematischen Seminar der Universit\"at
              Hamburg},
    VOLUME = {80},
      YEAR = {2010},
    NUMBER = {2},
     PAGES = {213--226},
      ISSN = {0025-5858,1865-8784},
   MRCLASS = {11L05},
  MRNUMBER = {2734687},
MRREVIEWER = {Florent\ Jouve},
       DOI = {10.1007/s12188-010-0041-0},
       URL = {https://doi.org/10.1007/s12188-010-0041-0},
}

@article {BS13,
    AUTHOR = {Boylan, Hatice and Skoruppa, Nils-Peter},
     TITLE = {A quick proof of reciprocity for {H}ecke {G}auss sums},
   JOURNAL = {J. Number Theory},
  FJOURNAL = {Journal of Number Theory},
    VOLUME = {133},
      YEAR = {2013},
    NUMBER = {1},
     PAGES = {110--114},
      ISSN = {0022-314X,1096-1658},
   MRCLASS = {11L05},
  MRNUMBER = {2981402},
MRREVIEWER = {Alina\ Ostafe},
       DOI = {10.1016/j.jnt.2012.08.001},
       URL = {https://doi.org/10.1016/j.jnt.2012.08.001},
}

@article {Zhu23,
    AUTHOR = {Zhu, Xiao-Jie},
     TITLE = {Taylor expansions of {J}acobi forms and linear relations among
              theta series},
   JOURNAL = {Dissertationes Math.},
  FJOURNAL = {Dissertationes Mathematicae},
    VOLUME = {590},
      YEAR = {2023},
     PAGES = {66},
      ISSN = {0012-3862,1730-6310},
   MRCLASS = {11F50 (11F11 11F27 11F37)},
  MRNUMBER = {4684413},
       DOI = {10.4064/dm880-12-2023},
       URL = {https://doi.org/10.4064/dm880-12-2023},
}

@article {Sai91,
    AUTHOR = {Saito, Hiroshi},
     TITLE = {A generalization of {G}auss sums and its applications to
              {S}iegel modular forms and {$L$}-functions associated with the
              vector space of quadratic forms},
   JOURNAL = {J. Reine Angew. Math.},
  FJOURNAL = {Journal f\"ur die Reine und Angewandte Mathematik. [Crelle's
              Journal]},
    VOLUME = {416},
      YEAR = {1991},
     PAGES = {91--142},
      ISSN = {0075-4102,1435-5345},
   MRCLASS = {11F46 (11F66 11L05)},
  MRNUMBER = {1099947},
MRREVIEWER = {Alexey\ A.\ Panchishkin},
       DOI = {10.1515/crll.1991.416.91},
       URL = {https://doi.org/10.1515/crll.1991.416.91},
}

@article {CS76,
    AUTHOR = {Callahan, T. and Smith, R. A.},
     TITLE = {{$L$}-functions of a quadratic form},
   JOURNAL = {Trans. Amer. Math. Soc.},
  FJOURNAL = {Transactions of the American Mathematical Society},
    VOLUME = {217},
      YEAR = {1976},
     PAGES = {297--309},
      ISSN = {0002-9947,1088-6850},
   MRCLASS = {10H10},
  MRNUMBER = {404164},
MRREVIEWER = {H.\ M.\ Stark},
       DOI = {10.2307/1997572},
       URL = {https://doi.org/10.2307/1997572},
}

@book {MH73,
    AUTHOR = {Milnor, John and Husemoller, Dale},
     TITLE = {Symmetric bilinear forms},
    SERIES = {Ergebnisse der Mathematik und ihrer Grenzgebiete [Results in
              Mathematics and Related Areas]},
    VOLUME = {Band 73},
 PUBLISHER = {Springer-Verlag, New York-Heidelberg},
      YEAR = {1973},
     PAGES = {viii+147},
   MRCLASS = {15A63 (10C05 57D65)},
  MRNUMBER = {506372},
MRREVIEWER = {Louis\ H.\ Kauffman},
}

@article {AAW14,
    AUTHOR = {Alaca, Ay\c se and Alaca, \c Saban and Williams, Kenneth S.},
     TITLE = {Double {G}auss sums},
   JOURNAL = {J. Comb. Number Theory},
  FJOURNAL = {Journal of Combinatorics and Number Theory},
    VOLUME = {6},
      YEAR = {2014},
    NUMBER = {2},
     PAGES = {127--153},
      ISSN = {1942-5600},
   MRCLASS = {11L03 (11D79 11E16 11E25 11L05 11T23)},
  MRNUMBER = {3361933},
MRREVIEWER = {Charles\ Helou},
}

@article {AD17,
    AUTHOR = {Alaca, \c Saban and Doyle, Greg},
     TITLE = {Explicit evaluation of double {G}auss sums},
   JOURNAL = {J. Comb. Number Theory},
  FJOURNAL = {Journal of Combinatorics and Number Theory},
    VOLUME = {9},
      YEAR = {2017},
    NUMBER = {1},
     PAGES = {47--61},
      ISSN = {1942-5600},
   MRCLASS = {11L03 (11L05 11T23)},
  MRNUMBER = {3728396},
}

@article {Sta67,
    AUTHOR = {Stark, H. M.},
     TITLE = {{$L$}-functions and character sums for quadratic forms. {I}},
   JOURNAL = {Acta Arith.},
  FJOURNAL = {Polska Akademia Nauk. Instytut Matematyczny. Acta Arithmetica},
    VOLUME = {14},
      YEAR = {1967/68},
     PAGES = {35--50},
      ISSN = {0065-1036},
   MRCLASS = {10.41},
  MRNUMBER = {227122},
MRREVIEWER = {T.\ Kubota},
       DOI = {10.4064/aa-14-1-35-50},
       URL = {https://doi.org/10.4064/aa-14-1-35-50},
}

@book {IK04,
    AUTHOR = {Iwaniec, Henryk and Kowalski, Emmanuel},
     TITLE = {Analytic number theory},
    SERIES = {American Mathematical Society Colloquium Publications},
    VOLUME = {53},
 PUBLISHER = {American Mathematical Society, Providence, RI},
      YEAR = {2004},
     PAGES = {xii+615},
      ISBN = {0-8218-3633-1},
   MRCLASS = {11-02 (11Fxx 11Lxx 11Mxx 11Nxx)},
  MRNUMBER = {2061214},
MRREVIEWER = {K.\ Soundararajan},
       DOI = {10.1090/coll/053},
       URL = {https://doi.org/10.1090/coll/053},
}

@article {HL24,
    AUTHOR = {Hu, Guangwei and Lao, Huixue},
     TITLE = {The circle method and shifted convolution sums involving the
              divisor function},
   JOURNAL = {J. Number Theory},
  FJOURNAL = {Journal of Number Theory},
    VOLUME = {262},
      YEAR = {2024},
     PAGES = {1--27},
      ISSN = {0022-314X,1096-1658},
   MRCLASS = {11F30 (11E76 11P55)},
  MRNUMBER = {4736664},
       DOI = {10.1016/j.jnt.2024.03.007},
       URL = {https://doi.org/10.1016/j.jnt.2024.03.007},
}

@article {Str13,
    AUTHOR = {Str\"omberg, Fredrik},
     TITLE = {Weil representations associated with finite quadratic modules},
   JOURNAL = {Math. Z.},
  FJOURNAL = {Mathematische Zeitschrift},
    VOLUME = {275},
      YEAR = {2013},
    NUMBER = {1-2},
     PAGES = {509--527},
      ISSN = {0025-5874,1432-1823},
   MRCLASS = {11F27 (20C25)},
  MRNUMBER = {3101818},
MRREVIEWER = {Solomon\ Friedberg},
       DOI = {10.1007/s00209-013-1145-x},
       URL = {https://doi.org/10.1007/s00209-013-1145-x},
}

@article {Sch09,
    AUTHOR = {Scheithauer, Nils R.},
     TITLE = {The {W}eil representation of {${\rm SL}_2(\Bbb Z)$} and some
              applications},
   JOURNAL = {Int. Math. Res. Not. IMRN},
  FJOURNAL = {International Mathematics Research Notices. IMRN},
      YEAR = {2009},
    NUMBER = {8},
     PAGES = {1488--1545},
      ISSN = {1073-7928,1687-0247},
   MRCLASS = {11F27 (11F20 11F55 11H56 17B67 81R10)},
  MRNUMBER = {2496771},
MRREVIEWER = {Rainer\ Schulze-Pillot},
       DOI = {10.1093/imrn/rnn166},
       URL = {https://doi.org/10.1093/imrn/rnn166},
}

@book {CS17,
    AUTHOR = {Cohen, Henri and Str\"omberg, Fredrik},
     TITLE = {Modular forms},
    SERIES = {Graduate Studies in Mathematics},
    VOLUME = {179},
      NOTE = {A classical approach},
 PUBLISHER = {American Mathematical Society, Providence, RI},
      YEAR = {2017},
     PAGES = {xii+700},
      ISBN = {978-0-8218-4947-7},
   MRCLASS = {11-01 (11Fxx)},
  MRNUMBER = {3675870},
}

@article {Wei64,
    AUTHOR = {Weil, Andr\'e},
     TITLE = {Sur certains groupes d'op\'erateurs unitaires},
   JOURNAL = {Acta Math.},
  FJOURNAL = {Acta Mathematica},
    VOLUME = {111},
      YEAR = {1964},
     PAGES = {143--211},
      ISSN = {0001-5962,1871-2509},
   MRCLASS = {22.65 (32.65)},
  MRNUMBER = {165033},
MRREVIEWER = {G.\ W.\ Mackey},
       DOI = {10.1007/BF02391012},
       URL = {https://doi.org/10.1007/BF02391012},
}

@book {Bru02,
    AUTHOR = {Bruinier, Jan H.},
     TITLE = {Borcherds products on {O}(2, {$l$}) and {C}hern classes of
              {H}eegner divisors},
    SERIES = {Lecture Notes in Mathematics},
    VOLUME = {1780},
 PUBLISHER = {Springer-Verlag, Berlin},
      YEAR = {2002},
     PAGES = {viii+152},
      ISBN = {3-540-43320-1},
   MRCLASS = {11F55 (11F23 11F27 11G18)},
  MRNUMBER = {1903920},
MRREVIEWER = {Rainer\ Schulze-Pillot},
       DOI = {10.1007/b83278},
       URL = {https://doi.org/10.1007/b83278},
}

@article {Bor98,
    AUTHOR = {Borcherds, Richard E.},
     TITLE = {Automorphic forms with singularities on {G}rassmannians},
   JOURNAL = {Invent. Math.},
  FJOURNAL = {Inventiones Mathematicae},
    VOLUME = {132},
      YEAR = {1998},
    NUMBER = {3},
     PAGES = {491--562},
      ISSN = {0020-9910,1432-1297},
   MRCLASS = {11F37 (11F22 14J28 17B67 57R57)},
  MRNUMBER = {1625724},
MRREVIEWER = {I.\ Dolgachev},
       DOI = {10.1007/s002220050232},
       URL = {https://doi.org/10.1007/s002220050232},
}

@article {WW25,
    AUTHOR = {Wang, Haowu and Williams, Brandon},
     TITLE = {Mathieu moonshine and {B}orcherds products},
   JOURNAL = {Commun. Number Theory Phys.},
  FJOURNAL = {Communications in Number Theory and Physics},
    VOLUME = {19},
      YEAR = {2025},
    NUMBER = {1},
     PAGES = {135--168},
      ISSN = {1931-4523,1931-4531},
   MRCLASS = {11F50 (11F46 14J28 81T30)},
  MRNUMBER = {4865807},
MRREVIEWER = {O.\ V.\ Shvartsman},
       DOI = {10.4310/cntp.250215003845},
       URL = {https://doi.org/10.4310/cntp.250215003845},
}

@article {Shi75,
    AUTHOR = {Shintani, Takuro},
     TITLE = {On construction of holomorphic cusp forms of half integral
              weight},
   JOURNAL = {Nagoya Math. J.},
  FJOURNAL = {Nagoya Mathematical Journal},
    VOLUME = {58},
      YEAR = {1975},
     PAGES = {83--126},
      ISSN = {0027-7630,2152-6842},
   MRCLASS = {10D15},
  MRNUMBER = {389772},
MRREVIEWER = {Stephen\ Gelbart},
       URL = {http://projecteuclid.org/euclid.nmj/1118795445},
}

@article {Bor00,
    AUTHOR = {Borcherds, Richard E.},
     TITLE = {Reflection groups of {L}orentzian lattices},
   JOURNAL = {Duke Math. J.},
  FJOURNAL = {Duke Mathematical Journal},
    VOLUME = {104},
      YEAR = {2000},
    NUMBER = {2},
     PAGES = {319--366},
      ISSN = {0012-7094,1547-7398},
   MRCLASS = {11H56 (11F11 11F27 11M36 51F15)},
  MRNUMBER = {1773561},
MRREVIEWER = {G.\ J.\ Heckman},
       DOI = {10.1215/S0012-7094-00-10424-3},
       URL = {https://doi.org/10.1215/S0012-7094-00-10424-3},
}

@article {LO18,
    AUTHOR = {Li, Songsong and Ouyang, Yi},
     TITLE = {Counting the solutions of
              {$\lambda_1x_1^{k_1}+\cdots+\lambda_tx_t^{k_t}\equiv c\mod
              n$}},
   JOURNAL = {J. Number Theory},
  FJOURNAL = {Journal of Number Theory},
    VOLUME = {187},
      YEAR = {2018},
     PAGES = {41--65},
      ISSN = {0022-314X,1096-1658},
   MRCLASS = {11B13 (11A25 11D79 11L03 11L05)},
  MRNUMBER = {3766901},
MRREVIEWER = {Johnny\ Edwards},
       DOI = {10.1016/j.jnt.2017.10.017},
       URL = {https://doi.org/10.1016/j.jnt.2017.10.017},
}

@book {Aig13,
    AUTHOR = {Aigner, Martin},
     TITLE = {Markov's theorem and 100 years of the uniqueness conjecture},
      NOTE = {A mathematical journey from irrational numbers to perfect
              matchings},
 PUBLISHER = {Springer, Cham},
      YEAR = {2013},
     PAGES = {x+257},
      ISBN = {978-3-319-00887-5; 978-3-319-00888-2},
   MRCLASS = {11-03 (11J06 20E05 20H10 68R15)},
  MRNUMBER = {3098784},
MRREVIEWER = {Thomas\ A.\ Schmidt},
       DOI = {10.1007/978-3-319-00888-2},
       URL = {https://doi.org/10.1007/978-3-319-00888-2},
}

@article {Via16,
    AUTHOR = {Vianna, Renato Ferreira de Velloso},
     TITLE = {Infinitely many exotic monotone {L}agrangian tori in
              {$\Bbb{CP}^2$}},
   JOURNAL = {J. Topol.},
  FJOURNAL = {Journal of Topology},
    VOLUME = {9},
      YEAR = {2016},
    NUMBER = {2},
     PAGES = {535--551},
      ISSN = {1753-8416,1753-8424},
   MRCLASS = {53D12 (53D37 53D42)},
  MRNUMBER = {3509972},
MRREVIEWER = {Stefan\ Nemirovski},
       DOI = {10.1112/jtopol/jtw002},
       URL = {https://doi.org/10.1112/jtopol/jtw002},
}

@book {Bar91,
    AUTHOR = {Baragar, Arthur},
     TITLE = {The {M}arkoff equation and equations of {H}urwitz},
      NOTE = {Thesis (Ph.D.)--Brown University},
 PUBLISHER = {ProQuest LLC, Ann Arbor, MI},
      YEAR = {1991},
     PAGES = {164},
   MRCLASS = {99-05},
  MRNUMBER = {2686830},
       URL = {http://gateway.proquest.com/openurl?url_ver=Z39.88-2004&rft_val_fmt=info:ofi/fmt:kev:mtx:dissertation&res_dat=xri:pqdiss&rft_dat=xri:pqdiss:9204827},
}

@article {BGS16,
    AUTHOR = {Bourgain, Jean and Gamburd, Alexander and Sarnak, Peter},
     TITLE = {Markoff triples and strong approximation},
   JOURNAL = {C. R. Math. Acad. Sci. Paris},
  FJOURNAL = {Comptes Rendus Math\'ematique. Acad\'emie des Sciences. Paris},
    VOLUME = {354},
      YEAR = {2016},
    NUMBER = {2},
     PAGES = {131--135},
      ISSN = {1631-073X,1778-3569},
   MRCLASS = {11G05 (14G15)},
  MRNUMBER = {3456887},
MRREVIEWER = {John\ L.\ Boxall},
       DOI = {10.1016/j.crma.2015.12.006},
       URL = {https://doi.org/10.1016/j.crma.2015.12.006},
}

@article {Che24,
    AUTHOR = {Chen, William Y.},
     TITLE = {Nonabelian level structures, {N}ielsen equivalence, and
              {M}arkoff triples},
   JOURNAL = {Ann. of Math. (2)},
  FJOURNAL = {Annals of Mathematics. Second Series},
    VOLUME = {199},
      YEAR = {2024},
    NUMBER = {1},
     PAGES = {301--443},
      ISSN = {0003-486X,1939-8980},
   MRCLASS = {11D25 (11J06 14D20 14G12 14H10 14H30 14M35 20D60)},
  MRNUMBER = {4681147},
MRREVIEWER = {David\ P.\ Roberts},
       DOI = {10.4007/annals.2024.199.1.5},
       URL = {https://doi.org/10.4007/annals.2024.199.1.5},
}

@article {Mar25,
    AUTHOR = {Martin, Daniel E.},
     TITLE = {A new proof of {C}hen's theorem for {M}arkoff graphs},
   JOURNAL = {Invent. Math.},
  FJOURNAL = {Inventiones Mathematicae},
    VOLUME = {241},
      YEAR = {2025},
    NUMBER = {2},
     PAGES = {623--626},
      ISSN = {0020-9910,1432-1297},
   MRCLASS = {11D25 (05C25)},
  MRNUMBER = {4929652},
       DOI = {10.1007/s00222-025-01346-9},
       URL = {https://doi.org/10.1007/s00222-025-01346-9},
}

@incollection {Sko08,
    AUTHOR = {Skoruppa, Nils-Peter},
     TITLE = {Jacobi forms of critical weight and {W}eil representations},
 BOOKTITLE = {Modular forms on {S}chiermonnikoog},
     PAGES = {239--266},
 PUBLISHER = {Cambridge Univ. Press, Cambridge},
      YEAR = {2008},
      ISBN = {978-0-521-49354-3},
   MRCLASS = {11F50 (11F27 11F37 11F46 11F55)},
  MRNUMBER = {2512363},
MRREVIEWER = {Rainer\ Schulze-Pillot},
       DOI = {10.1017/CBO9780511543371.013},
       URL = {https://doi.org/10.1017/CBO9780511543371.013},
}

@incollection {ES17,
    AUTHOR = {Ehlen, Stephan and Skoruppa, Nils-Peter},
     TITLE = {Computing invariants of the {W}eil representation},
 BOOKTITLE = {L-functions and automorphic forms},
    SERIES = {Contrib. Math. Comput. Sci.},
    VOLUME = {10},
     PAGES = {81--96},
 PUBLISHER = {Springer, Cham},
      YEAR = {2017},
      ISBN = {978-3-319-69711-6; 978-3-319-69712-3},
   MRCLASS = {11F70 (11Y16)},
  MRNUMBER = {3931449},
MRREVIEWER = {Chenyan\ Wu},
}

@book {Nat00,
    AUTHOR = {Nathanson, Melvyn B.},
     TITLE = {Elementary methods in number theory},
    SERIES = {Graduate Texts in Mathematics},
    VOLUME = {195},
 PUBLISHER = {Springer-Verlag, New York},
      YEAR = {2000},
     PAGES = {xviii+513},
      ISBN = {0-387-98912-9},
   MRCLASS = {11-01 (11-00)},
  MRNUMBER = {1732941},
MRREVIEWER = {S.\ W.\ Graham},
}

@book {Ter99,
    AUTHOR = {Terras, Audrey},
     TITLE = {Fourier analysis on finite groups and applications},
    SERIES = {London Mathematical Society Student Texts},
    VOLUME = {43},
 PUBLISHER = {Cambridge University Press, Cambridge},
      YEAR = {1999},
     PAGES = {x+442},
      ISBN = {0-521-45718-1},
   MRCLASS = {11-02 (11T60 11Z05 20C15 43-01)},
  MRNUMBER = {1695775},
MRREVIEWER = {Stefan\ K\"uhnlein},
       DOI = {10.1017/CBO9780511626265},
       URL = {https://doi.org/10.1017/CBO9780511626265},
}

@article{Zhu25,
 author = {Zhu, Xiao-Jie},
 title = {Explicit formulae for linear characters of {{\(\Gamma_0(N)\)}}},
 fjournal = {Communications in Algebra},
 journal = {Commun. Algebra},
 issn = {0092-7872},
 volume = {53},
 number = {11},
 pages = {4499--4510},
 year = {2025},
 language = {English},
 doi = {10.1080/00927872.2025.2488029},
 zbMATH = {8086657}
}

@article {Zhu23_2,
    AUTHOR = {Zhu, Xiao-Jie},
     TITLE = {Holomorphic {E}isenstein series of rational weights and
              special values of gamma function},
   JOURNAL = {Acta Arith.},
  FJOURNAL = {Acta Arithmetica},
    VOLUME = {210},
      YEAR = {2023},
     PAGES = {279--305},
      ISSN = {0065-1036,1730-6264},
   MRCLASS = {11F30 (11F03 11F11 11F20 33B15)},
  MRNUMBER = {4678128},
MRREVIEWER = {Deyu\ Zhang},
       DOI = {10.4064/aa221110-1-4},
       URL = {https://doi.org/10.4064/aa221110-1-4},
}

@book {Boy15,
    AUTHOR = {Boylan, Hatice},
     TITLE = {Jacobi forms, finite quadratic modules and {W}eil
              representations over number fields},
    SERIES = {Lecture Notes in Mathematics},
    VOLUME = {2130},
      NOTE = {With a foreword by Nils-Peter Skoruppa},
 PUBLISHER = {Springer, Cham},
      YEAR = {2015},
     PAGES = {xx+130},
      ISBN = {978-3-319-12915-0; 978-3-319-12916-7},
   MRCLASS = {11F50 (11F27)},
  MRNUMBER = {3309829},
MRREVIEWER = {Jaban Meher},
       DOI = {10.1007/978-3-319-12916-7},
       URL = {https://doi.org/10.1007/978-3-319-12916-7},
}

@article{Zhu24,
 author = {Zhu, Xiao-Jie},
 title = {Finite quadratic modules and lattices},
 fjournal = {Communications in Algebra},
 journal = {Commun. Algebra},
 issn = {0092-7872},
 volume = {52},
 number = {2},
 pages = {604--629},
 year = {2024},
 language = {English},
 doi = {10.1080/00927872.2023.2245924},
 zbMATH = {7808411},
 Zbl = {1546.11051}
}

@article{Wal63,
 author = {Wall, C. T. C.},
 title = {Quadratic forms on finite groups, and related topics},
 fjournal = {Topology},
 journal = {Topology},
 issn = {0040-9383},
 volume = {2},
 pages = {281--298},
 year = {1963},
 language = {English},
 doi = {10.1016/0040-9383(63)90012-0},
 zbMATH = {3342190},
 Zbl = {0215.39903}
}

@misc{CZ25,
 author = {Chen, Rong and Zhu, Xiao-Jie},
 title = {Correspondence among congruence families for generalized {Frobenius} partitions via modular permutations},
 year = {2025},
 howpublished = {Preprint, {arXiv}:2506.16823 [math.{NT}] (2025)},
 url = {https://arxiv.org/abs/2506.16823},
 arXiv = {arXiv:2506.16823}
}

@book {TY05,
    AUTHOR = {Tauvel, Patrice and Yu, Rupert W. T.},
     TITLE = {Lie algebras and algebraic groups},
    SERIES = {Springer Monographs in Mathematics},
 PUBLISHER = {Springer-Verlag, Berlin},
      YEAR = {2005},
     PAGES = {xvi+653},
      ISBN = {978-3-540-24170-6; 3-540-24170-1},
   MRCLASS = {17-01 (17-02 17Bxx 20Gxx)},
  MRNUMBER = {2146652},
MRREVIEWER = {Ivan\ Arzhantsev},
}

@article {SZ89,
    AUTHOR = {Skoruppa, Nils-Peter and Zagier, Don},
     TITLE = {A trace formula for {J}acobi forms},
   JOURNAL = {J. Reine Angew. Math.},
  FJOURNAL = {Journal f\"ur die Reine und Angewandte Mathematik. [Crelle's
              Journal]},
    VOLUME = {393},
      YEAR = {1989},
     PAGES = {168--198},
      ISSN = {0075-4102,1435-5345},
   MRCLASS = {11F37 (11F11 11F72)},
  MRNUMBER = {972365},
MRREVIEWER = {Aloys\ Krieg},
       DOI = {10.1515/crll.1989.393.168},
       URL = {https://doi.org/10.1515/crll.1989.393.168},
}

@book {Lan94,
    AUTHOR = {Lang, Serge},
     TITLE = {Algebraic number theory},
    SERIES = {Graduate Texts in Mathematics},
    VOLUME = {110},
   EDITION = {Second},
 PUBLISHER = {Springer-Verlag, New York},
      YEAR = {1994},
     PAGES = {xiv+357},
      ISBN = {0-387-94225-4},
   MRCLASS = {11Rxx (11-01 11-02)},
  MRNUMBER = {1282723},
MRREVIEWER = {M.\ Ram\ Murty},
       DOI = {10.1007/978-1-4612-0853-2},
       URL = {https://doi.org/10.1007/978-1-4612-0853-2},
}

@book {Rom06,
    AUTHOR = {Roman, Steven},
     TITLE = {Field theory},
    SERIES = {Graduate Texts in Mathematics},
    VOLUME = {158},
   EDITION = {Second},
 PUBLISHER = {Springer, New York},
      YEAR = {2006},
     PAGES = {xii+332},
      ISBN = {978-0387-27677-9; 0-387-27677-7},
   MRCLASS = {12-01 (12Fxx)},
  MRNUMBER = {2178351},
}

@article {Sch39,
    AUTHOR = {Schoeneberg, Bruno},
     TITLE = {Das {V}erhalten von mehrfachen {T}hetareihen bei
              {M}odulsubstitutionen},
   JOURNAL = {Math. Ann.},
  FJOURNAL = {Mathematische Annalen},
    VOLUME = {116},
      YEAR = {1939},
    NUMBER = {1},
     PAGES = {511--523},
      ISSN = {0025-5831,1432-1807},
   MRCLASS = {99-04},
  MRNUMBER = {1513241},
       DOI = {10.1007/BF01597371},
       URL = {https://doi.org/10.1007/BF01597371},
}

@article {Pfe53,
    AUTHOR = {Pfetzer, Werner},
     TITLE = {Die {W}irkung der {M}odulsubstitutionen auf mehrafache
              {T}hetareihen zu quadratischen {F}ormen ungerader
              {V}ariablenzahl},
   JOURNAL = {Arch. Math. (Basel)},
  FJOURNAL = {Archiv der Mathematik},
    VOLUME = {4},
      YEAR = {1953},
     PAGES = {448--454},
      ISSN = {0003-889X,1420-8938},
   MRCLASS = {10.0X},
  MRNUMBER = {59945},
MRREVIEWER = {J.\ Lehner},
       DOI = {10.1007/BF01899265},
       URL = {https://doi.org/10.1007/BF01899265},
}

@article {Ric00_1,
    AUTHOR = {Richter, Olav K.},
     TITLE = {Theta functions of quadratic forms over imaginary quadratic
              fields},
   JOURNAL = {Acta Arith.},
  FJOURNAL = {Acta Arithmetica},
    VOLUME = {92},
      YEAR = {2000},
    NUMBER = {1},
     PAGES = {1--9},
      ISSN = {0065-1036,1730-6264},
   MRCLASS = {11F27 (11E45 11F55)},
  MRNUMBER = {1739742},
       DOI = {10.4064/aa-92-1-1-9},
       URL = {https://doi.org/10.4064/aa-92-1-1-9},
}

@article {Ric00_2,
    AUTHOR = {Richter, Olav K.},
     TITLE = {Theta functions of indefinite quadratic forms over real number
              fields},
   JOURNAL = {Proc. Amer. Math. Soc.},
  FJOURNAL = {Proceedings of the American Mathematical Society},
    VOLUME = {128},
      YEAR = {2000},
    NUMBER = {3},
     PAGES = {701--708},
      ISSN = {0002-9939,1088-6826},
   MRCLASS = {11F27 (11E45 11F55)},
  MRNUMBER = {1706997},
MRREVIEWER = {Vladimir\ G.\ Zhuravlev},
       DOI = {10.1090/S0002-9939-99-05619-1},
       URL = {https://doi.org/10.1090/S0002-9939-99-05619-1},
}

@book {Ebe13,
    AUTHOR = {Ebeling, Wolfgang},
     TITLE = {Lattices and codes},
    SERIES = {Advanced Lectures in Mathematics},
   EDITION = {Third},
      NOTE = {A course partially based on lectures by Friedrich Hirzebruch},
 PUBLISHER = {Springer Spektrum, Wiesbaden},
      YEAR = {2013},
     PAGES = {xvi+167},
      ISBN = {978-3-658-00359-3; 978-3-658-00360-9},
   MRCLASS = {11H31 (11T71 51F15)},
  MRNUMBER = {2977354},
MRREVIEWER = {Caleb\ McKinley\ Shor},
       DOI = {10.1007/978-3-658-00360-9},
       URL = {https://doi.org/10.1007/978-3-658-00360-9},
}

@phdthesis{Ajo15,
author = {Ajouz, Ali},
year = {2015},
school = {University of Siegen},
title = {Hecke operators on Jacobi forms of lattice index and the relation to elliptic modular forms},
}

@article {Sie36,
    AUTHOR = {Siegel, Carl Ludwig},
     TITLE = {\"Uber die analytische {T}heorie der quadratischen {F}ormen.
              {II}},
   JOURNAL = {Ann. of Math. (2)},
  FJOURNAL = {Annals of Mathematics. Second Series},
    VOLUME = {37},
      YEAR = {1936},
    NUMBER = {1},
     PAGES = {230--263},
      ISSN = {0003-486X,1939-8980},
   MRCLASS = {99-04},
  MRNUMBER = {1503276},
       DOI = {10.2307/1968694},
       URL = {https://doi.org/10.2307/1968694},
}

@article {Yan98,
    AUTHOR = {Yang, Tonghai},
     TITLE = {An explicit formula for local densities of quadratic forms},
   JOURNAL = {J. Number Theory},
  FJOURNAL = {Journal of Number Theory},
    VOLUME = {72},
      YEAR = {1998},
    NUMBER = {2},
     PAGES = {309--356},
      ISSN = {0022-314X,1096-1658},
   MRCLASS = {11E08},
  MRNUMBER = {1651696},
MRREVIEWER = {Hidenori\ Katsurada},
       DOI = {10.1006/jnth.1998.2258},
       URL = {https://doi.org/10.1006/jnth.1998.2258},
}

@article {Shi04,
    AUTHOR = {Shimura, Goro},
     TITLE = {Inhomogeneous quadratic forms and triangular numbers},
   JOURNAL = {Amer. J. Math.},
  FJOURNAL = {American Journal of Mathematics},
    VOLUME = {126},
      YEAR = {2004},
    NUMBER = {1},
     PAGES = {191--214},
      ISSN = {0002-9327,1080-6377},
   MRCLASS = {11E25 (11F41)},
  MRNUMBER = {2033567},
MRREVIEWER = {Hideshi\ Takayanagi},
       URL =
              {http://muse.jhu.edu/journals/american_journal_of_mathematics/v126/126.1shimura.pdf},
}

@article {Sage,
    AUTHOR = {SageMath},
     TITLE = {The Sage Mathematics Software System (Version 10.4)},
   JOURNAL = {The Sage Developers},
      YEAR = {2025},
       URL = {https://www.sagemath.org},
}

@article {MS25,
    AUTHOR = {M\"uller, Manuel K.-H. and Scheithauer, Nils R.},
     TITLE = {The invariants of the {W}eil representation of {${\rm
              SL}_2(\Bbb Z)$}},
   JOURNAL = {Commun. Contemp. Math.},
  FJOURNAL = {Communications in Contemporary Mathematics},
    VOLUME = {27},
      YEAR = {2025},
    NUMBER = {9},
     PAGES = {Paper No. 2550013, 38},
      ISSN = {0219-1997,1793-6683},
   MRCLASS = {11F27 (11F11 11F50 20C33)},
  MRNUMBER = {4942310},
       DOI = {10.1142/S0219199725500130},
       URL = {https://doi.org/10.1142/S0219199725500130},
}

@unpublished {Sko25,
    AUTHOR = {Skoruppa, Nils-Peter},
     TITLE = {Finite quadratic modules, Weil representations and vector valued modular forms},
      NOTE = {preprint, 2025}
}

@article {Nik79,
    AUTHOR = {Nikulin, V. V.},
     TITLE = {Integer symmetric bilinear forms and some of their geometric
              applications},
   JOURNAL = {Izv. Akad. Nauk SSSR Ser. Mat.},
  FJOURNAL = {Izvestiya Akademii Nauk SSSR. Seriya Matematicheskaya},
    VOLUME = {43},
      YEAR = {1979},
    NUMBER = {1},
     PAGES = {111--177, 238},
      ISSN = {0373-2436},
   MRCLASS = {10C05 (14G30 14J17 14J25 57M99 57R45 58C27)},
  MRNUMBER = {525944},
MRREVIEWER = {I.\ Dolgachev},
}

\end{document}